\documentclass[clean,reqno,twoside,a4paper,11pt]{amsart}
\usepackage{mltools}
\usepackage{math62,mlthm10}
\usepackage{xy} \xyoption{all} 
\usepackage{amsmath,upref} 
\usepackage{amssymb,amsfonts,mathrsfs}
\usepackage{tensor} 
\usepackage{bm}   

\usepackage[symbol]{footmisc}
\usepackage{perpage}
\MakePerPage[5]{footnote}

\usepackage[colorlinks=true,linkcolor=blue,citecolor=blue,backref=page]{hyperref}
\usepackage{upref,mathrsfs,enumerate,amscd}
\usepackage{local}

\usepackage{euler}

\numberwithin{equation}{section} 

\allowdisplaybreaks

\newcommand{\MLrevision}{\mpar{{\color{red} CHANGED }}}

\begin{document}

\title{Modular curvature and Morita equivalence} 

\author{Matthias Lesch}
\address{Mathematisches Institut,
Universit\"at Bonn,
Endenicher Allee 60,
53115 Bonn,
Germany}

\email{ml@matthiaslesch.de, lesch@math.uni-bonn.de}
\urladdr{www.matthiaslesch.de, www.math.uni-bonn.de/people/lesch}
\thanks{Both authors were partially supported by the 
        Hausdorff Center for Mathematics, Bonn}

\author{Henri Moscovici}
\address{Department of Mathematics,
The Ohio State University,
Columbus, OH 43210,
USA}
\email{henri@math.ohio-state.edu}

\thanks{The work of H. M. was partially
 supported by the US National Science Foundation
    award no. DMS-1300548} 



\date{\today}

\begin{abstract} 
The curvature of the noncommutative torus $T^2_\theta$ 
 ($\theta \in \R \setminus \Q$) endowed with a noncommutative conformal metric
 has been the focus of attention of several recent works. 
 Continuing the approach taken in the paper
[A. Connes and H. Moscovici, Modular curvature for 
 noncommutative two-tori, J. Amer. Math. Soc. 27 (2014), 639--684] 
 we extend the study of the curvature to twisted
 Dirac spectral triples constructed out of
 Heisenberg bimodules that implement the Morita equivalence of the $C^*$-algebra
   $A_\theta = C(T^2_\theta)$ with other toric algebras $A_{\theta'}=C(T^2_{\theta'})$. 
  In the enlarged context the conformal metric on $T^2_\theta$ is exchanged with
  an arbitrary Hermitian metric on the Heisenberg   $(A_\theta, A_{\theta'})$-bimodule 
   $E'$ for which $\End_{A_{\gt'}}(E') = A_\theta $. 
 We prove that the Ray-Singer log-determinant of the 
 corresponding Laplacian, viewed as a functional on the space of all Hermitian metrics
 on $E'$, attains its extremum at the unique Hermitian metric whose corresponding
 connection has constant curvature. The gradient of the log-determinant functional
 gives rise to a noncommutative analogue of the Gaussian curvature. 
 The genuinely new outcome of this paper is that the latter is shown to be
 independent of any Heisenberg bimodule 
 $E'$ such that $A_\theta = \End_{A_{\gt'}}(E')$,
 and in this sense it is Morita invariant. To prove the above results we extend 
 Connes' pseudodifferential calculus to Heisenberg modules. The twisted version,
 which offers more flexibility even in the case of trivial coefficients, could potentially be applied
to other problems in the elliptic theory on noncommutative tori.
 A noteworthy technical feature is that we systematize the computation of the resolvent
 expansion for elliptic differential operators on noncommutative tori to an extent
 which makes the (previously employed) computer assistance unnecessary.
\end{abstract}
 
 \maketitle


\section*{Introduction}

The concept of intrinsic curvature, which lies at the very core of geometry,
has only recently begun to be comprehended in the noncommutative framework.
As its earliest form, the Gaussian curvature, arose for Riemann surfaces, it
was natural to look first at the noncommutative $2$-torus  $T^2_\gt$ 
($\gt \in \R \setminus \Q$), the simplest but nevertheless revealing example of a
noncommutative surface.  Tools for attacking this problem were developed early
on, in Connes' seminal {\em Comptes Rendus} note~\cite{Con:CAG}. They were
applied in~\cite{CohCon1992, ConTre2011} to the computation of the value at
$0$ of the zeta function of the Laplacian associated to a
translation-invariant conformal metric on $T^2_\gt$, or equivalently to the
computation of the total curvature for such a metric, verifying the validity
of the Gauss-Bonnet formula.
The calculation of the full, not just the total, curvature was completed
in~\cite{ConMos2011} and also in~\cite{FarKha2013}, with the partial aid of
different computer algebra systems. The resulting
formula involves two kinds of functions of the modular operator associated to the conformal
factor. One of them is the Bernoulli generating function in the modular operator, applied to the
Laplacian of the conformal factor. 
The other term was given in~\cite{ConMos2011} a conceptual explanation,
as a consequence of expressing the curvature as the gradient of the Ray-Singer
log-determinant of the varying Laplacian. 
It was also shown in~\cite{ConMos2011} that the Ray-Singer
functional attains its extreme value only at the flat metric.

In this paper  we extend the study of the curvature to twisted
 Dirac spectral triples constructed out of                                 
 Heisenberg bimodules that implement the Morita equivalence of the $C^*$-algebra
 $A_\theta = C(T^2_\theta)$ with other toric algebras. The enlarged context sets the
scene for exploiting the Morita equivalence, which is shown to play a triple role. 
First of all, it exchanges
the special (conformal) metric on the base with a completely general metric on the bundle. 
Secondly, it confers to the Ray-Singer functional a status akin to the Connes-Rieffel  
noncommutative Yang-Mills functional. Thirdly, and most 
surprisingly, it leads to a noncommutative analogue of the Gaussian curvature
which is Morita invariant. 

To convey the flavor of our main results in an informal yet suggestive manner,
we shall appeal to the analogy between the basic Heisenberg modules over 
$A_\gt = C(T^2_\gt)$ and the $\operatorname{Spin}^c$ structures of an elliptic
curve.  The Heisenberg equivalence bimodules are finitely generated projective
modules that implement the Morita equivalences between $T_\gt^2$ and other
tori $T^2_{\gt'}$, with ${\gt'}$ necessarily in the
orbit of $\gt$ under the action of $\PSLtZ$ on the real
projective line (\cf~\cite{Con:CAG, Rie1981}). 
Introduced in~\cite{Con:CAG}, these modules have an attractive
geometric underpinning which we quickly recall (\cf~\cite{Con1982}).  
The free homotopy classes of closed geodesics on the flat torus $T^2 = \R^2/\Z^2$ are
parametrized by the rational projective line $P^1(\Q) \equiv \Q \cup
\{\frac{1}{0}\}$. A pair of relatively prime integers $(d, c) \in \Z^2$
determines a family of lines of slope $\frac{d}{c}$, which project onto simple
closed geodesics in the same free homotopy class. Let $N_{c, d}$ denote the
primitive closed geodesic of slope $\frac{d}{c}$ passing through the base
point of $T^2$. Consider now the Kronecker foliation $\cF_\gt$ of irrational
slope $\gt \in (0,1)$. Each geodesic $N_{c, d}$ gives a complete transversal to the
foliation $\cF_\gt$, equipped with a holonomy pseudogroup.
Choosing $a, b \in \Z$ such that $ad-bc=1$, the
convolution algebra of the corresponding \'etale holonomy groupoid can
be identified with the crossed product algebra $ C(\R/\Z) \rtimes_{\gt'}\Z$, where $1 \in \Z$
acts by the rotation of angle $\gt' =\frac{a\gt +b}{c \gt +d}$,
 which is isomorphic to $A_{\gt'}$;
in particular, $A_\gt$ corresponds to $N_{0, 1}$.  
The holonomy groupoids associated to the
transversals $N_{c, d}$ and $N_{0, 1}$ are Morita equivalent, and the
$(A_{\gt'}, A_\gt)$-bimodule $ E(g,\gt)$ implementing the Morita equivalence
between their $C^*$--algebras has a compelling geometric description (\cite{Con:CAG},
also \S \ref{SSSBasicHeisenberg} below); in particular its smooth version $
\cE(g,\gt)$ carries a canonical Hermitian connection of constant  curvature. 

Fixing a complex structure on $T^2$ with modular parameter 
$\tau\in \C$, $\Im \tau > 0$, gives rise to a spectral triple over $A_\gt$
with operator $D$ isospectral to the operator $\partial_\tau +
\partial_\tau^*$ on $T^2$.  We regard $D$ as the analogue of the Dirac
operator associated to a fundamental $\operatorname{Spin}^c$-structure, and
the collection of Heisenberg bimodules $\operatorname{Mor}(T^2_\gt): =
\{E(g,\gt) ; \, g \in \PSLtZ \}$ as the counterpart of the
set of $\operatorname{Spin}^c$-structures over an elliptic curve.  For each
$\cE = \cE(g,\gt)$ one forms the spectral triple with twisted Dirac operator
$D_\cE =\partial_\cE + \partial_\cE^*$ with coefficients in $\cE$.
It has a natural transposed with
coefficients in the $(\cA_\gt , \cA_{\gt'})$-bimodule $ \cE'= \cE(g^{-1},
g\cdot \gt)$; here $\cA_\gt$ stands for the usual smooth subalgebra of $A_\gt$.   
By analogy with Connes' spectral characterization of
$\operatorname{Spin}^c$-manifolds~\cite{Con2013}, these are precisely the
spectral triples which confer to $T_\gt^2$ the structure of a noncommutative
manifold endowed with a metric structure.

Since $\cA_\gt$ coincides with the endomorphism algebra 
 $\End_{\cA_{\gt'}}(\cE')$, an arbitrary change of Hermitian
structure on the $A_{\gt'}$-module  $ \cE'$ amounts to the choice of an
invertible positive element $k \in \cA_\gt$.  Our first main result (Theorem
\ref{main-curv}) computes (for $g \neq 1$) the ``curvature densities''
$\cK^\pm_{\cE', k}$ associated to the corresponding Laplacians,
$\Lapl^+_{\cE', k} = k\partial_{\cE'} \partial^*_{\cE'} k$ on $0$-forms, and
$\Lapl^-_{\cE', k} = \partial^*_{\cE'} k^2\partial_{\cE'} $ on $(0,1)$-forms.
This is the analogue of \cite[Theorem 3.2]{ConMos2011} (which corresponds to
$g = 1$) with the distinction that the Morita equivalence trades the conformal
metric on the ``base'' (with Weyl factor $k \in \cA_\gt$) for a completely
general Hermitian metric on the ``bundle''.

 Our second result is the extension of \cite[Theorem 4.6]{ConMos2011} to
Heisenberg modules. We use a closed formula for the log-determinant of
$\Lapl_{\cE', k}$ (Theorem \ref{main-RS}) to prove that the scale invariant
version $F_{\cE'}$ of the Ray-Singer functional attains its extreme value only
if $k=1$. This means that the extremum is attained for the only Hermitian
structure whose associated Hermitian connection has constant curvature
(Theorem \ref{main-YM}). In this way, the Ray-Singer determinant acquires a
status similar to that of the Yang-Mills functional of~\cite{ConRie1987}.

Finally, our most noteworthy result (Theorem \ref{thm-grad-F}) establishes
that the gradient at $\log k^2$ of the functional $F_{\cE'}$ is equal to the
curvature of the conformal metric on $T_\gt^2$ with Weyl factor $k  \in \At$,
and thus independent of the $\operatorname{Spin}^c$-structure $\cE'$. Adopting
the value of the gradient of the Ray-Singer functional at a metric as the definition of
its Gaussian curvature, this proves the invariance of the latter under Morita equivalence. 
This kind of Morita invariance is a purely noncommutative phenomenon, which
in the commutative case passes unnoticed. Nevertheless, the
result is somewhat reminiscent of Gauss' {\em theorema egregium}, if one
is willing to liken the metric $\operatorname{Spin}^c$-structures on $T_\gt^2$ 
to the metrics inherited from embeddings of the ordinary torus in Euclidean space.

The essential technical tool which allows us to obtain the above results is the
extension of Connes' pseudodifferential calculus to $C^*$-dynamical systems on
Heisenberg modules. Although quite natural, this extension appears to be of
independent interest, in view of other potential applications, such as the
computation of the curvature for the Laplacian associated to a Riemannian
metric in the sense of J. Rosenberg~\cite{Ros2013}, or more generally that of
the index density for an elliptic differential operator on a noncommutative
$n$-torus with coefficients in a Heisenberg module (for the index itself, see
~\cite[Theorem 10]{Con:CAG}).

The other significant technical feature is that we succeed in freeing the entire
calculation from any computer assistance. This is accomplished by a systematic
use of known computational shortcuts within the
pseudodifferential symbol calculus, supplemented by the manipulation of the
modular identities between the $2$-parameter family of Laplacians naturally
associated with the datum. 

We conclude the introduction with a quick outline of the plan of the paper.
In Section \ref{Statements}, after a modicum of necessary background, we give
the precise formulation of the above mentioned results.  The extension of
Connes' pseudodifferential calculus to Heisenberg modules, 
for the general $n$-dimensional torus, is carried out in
Section \ref{SPsiDOMul}. 

Sections \ref{SResExp} and \ref{b-4} contain the key analytical results of the
paper (Theorems \ref{TResExp} and \ref{TbInt}) together with their proofs.  As
in \cite{ConTre2011}, the starting point is the recursion formula for the
resolvent. We improve the efficiency of the previous calculations in two ways.
First, we show that, modulo functions of $\xi=(\xi_1, \xi_2)$ whose average is
$0$, the relevant coefficient in the asymptotic expansion of the resolvent of
the Laplacians can be expressed as a function of $|\eta|^2$, where
$\eta=\xi_1+\taubar \xi_2$.  Secondly, in computing the integral of that
coefficient (Section \ref{b-4}) we exploit another kind of symmetry, namely
the relations between the conjugates under the action of the Tomita modular
operator on the $2$-parameter family of Laplacians associated to the twisted
spectral triple. This allows to reduce the calculation of the modular
curvature functions to the case of the graded Laplacian. Incidentally, it also
explains the seemingly magic relation between the two-variable functions $H_0$
and $H_1$ in~\cite[\S 3]{ConMos2011}.

The whole resolvent calculation is done in the algebra of pseudodifferential
multipliers. Thus, the resulting formulas are universally valid in any
effective realization of the pseudodifferential calculus, provided that one is able
to relate the operator trace to the dual trace on the multiplier algebra. This is 
exactly what is done in Section \ref{SEffPsiDO}, for the Heisenberg representation
(Theorems \ref{TEffTrace}, \ref{TEffHeatExp}) as well as for the "trivial
bundle" case (which was implicitly used in all the previous papers dealing
with the subject).  

Finally, since the trace formula for a pseudodifferential operator in Connes' calculus
is an important technical ingredient for which there is no published proof, we
devote Appendix \ref{AppA} to a complete derivation of it.

\tableofcontents  
  
\section{Background and formulation of main results}  \label{Statements}

The metric structure of a noncommutative geometric space with coordinate
$C^*$--algebra $A$ is given in spectral terms, by a triplet $(\cA, \cH, D)$,
where $\cA$ is a dense, holomorphically closed subalgebra of $A$, the latter
being represented by bounded operators on a Hilbert space $\cH$, and $D$ is an
unbounded self-adjoint operator on $\cH$ (playing the role of the Dirac
operator) such that the commutators $[D, a] = D \circ a - a \circ D$, $a \in
\cA$, are well-defined and bounded.  Local invariants reflecting the curvature
of such a space are extracted by means of spectrally defined functionals, from
the high frequency behavior of the spectrum of $D$ coupled with the action of
the algebra $\cA$ on $\cH$. In this section we shall specify these basic
notions in the case of the noncommutative torus, and thus provide the
necessary background to formulate the main results in precise terms and the
appropriate perspective. 
  
\subsection{The noncommutative torus $T_\gt^2$ and its standard metric structure}

We generally follow the notation in \cite{ConMos2011}, with some minor
deviations.  Let $\gt\in\R\setminus\Q$ be a fixed irrational number. By
$T_\gt^2$ we mean the noncommutative space, whose topology is given by the
$C^*$--algebra $A_\gt \equiv C(T^2_\gt)$ generated by two unitaries $U_j,
j=1,2$, subject to the commutation relation 
$\, U_2 U_1 = e^{\tpii \gt} U_1 U_2$.

\subsubsection{Smooth structure}
The ordinary torus $T^2 = \left( \R/\Z\right)^2$
acts on $A_\gt$ by the automorphisms $\alpha_{\sr} \in Aut(A_\gt)$, 
$\sr  \in \R^2 $, defined by
\begin{equation} \label{Eq:group-act}
\alpha_{\sr} (U_1^n \, U_2^m)=e^{2\pi i(r_1n + r_2m)}U_1^n \, U_2^m \, , \qquad
\sr = (r_1, r_2) \in \R^2 \, .
\end{equation}
The smooth structure is given by the subalgebra 
$C^\infty (T^2_\gt)\equiv \At \subset A_\gt$ consisting of the smooth elements
for the above  group of automorphisms \ie of those 
$a= \sum_{k,l\in\Z} a_{k,l} U_1^k U_2^l \in A_\gt$ such that the sequence
$\{a_{k,l}\} \subset \C$ is rapidly decreasing.  The canonical framing of
$T_\gt^2$ is given by the infinitesimal generators of the same group of
automorphisms \ie by the derivations $\delta_1, \delta_2 \in {\rm Der} (\At)$
defined by    
\begin{equation} \label{Eq:Lie-act}
\begin{split}
  &\delta_1 (U_1)= \tpii\, U_1,\qquad \delta_1 (U_2) =0, \\
     &\delta_2 (U_1)=0,\, \qquad  \delta_2 (U_2) = \tpii\, U_2 ;
     \end{split}
\end{equation}
they generate an abelian Lie algebra of derivations $\mathfrak{g}(\At) := \R \gd_1 + \R \gd_2$.

\subsubsection{Standard $\operatorname{Spin}^c$ Dirac operator} 

We fix once and for all a modular parameter $\tau\in \C$ with $\Im \tau > 0$, denote 
\begin{align} \label{Eq:c-struct}
  \gd_\tau=\gd_1+\ovl{\tau}\gd_2 , \qquad  \ovl{\gd_\tau}=\gd_1+ \tau\gd_2 ,
\end{align}
and define a $T^2$--invariant complex structure on $T_\gt^2$ via the splitting
\begin{align*}  
\mathfrak{g}_\C(\At) := \mathfrak{g}(\At) \otimes \C 
    = \C \gd_\tau \oplus \C \ovl{\gd_\tau} .
\end{align*} 
\MLrevision
\mpar{Referee item 1}

We adopt the convention that scalar products are complex antilinear in the
first and linear in the second argument.  Scalar products will be denoted by
$\scalar{\cdot,\cdot}$ or $\scalarL{\cdot,\cdot}$, while the notation
$\inn{\cdot,\cdot}$ will be reserved for $C^*$--valued inner products.  With
$\varphi_0$ denoting the unique normalized trace on $\At$, we let
$\cH_0(\At)=L^2(\At,\varphi_0)$ be the completion of $\At$ with respect to the
scalar product
\begin{equation*}
       \scalarL{a,b}= \varphi_0(a^*b).   
\end{equation*}
On the other hand we let $\cH^{(1,0)}$ be the unitary bimodule over $A_{\gt}$
given by the Hilbert space completion of the universal derivation bimodule
$\Omega^1(\At)$ of finite sums $\sum \, a \, d(b)$, $a,b \in \At$  with
respect to the inner product
\begin{equation*} 
(a\, d(b) , a'\, d(b') ) 
    = \vp_0 (a^* \, a' \, \gd_\tau (b') \, \gd_\tau(b)^*) \, , 
        \quad a,a',b,b' \in \At \, .
\end{equation*}
\MLrevision
\mpar{Referee item 2}
We denote by $\Omega^{(1,0)} (\At) = \bigsetdef{\sum \, a \,
\partial_\tau(b)}{ a,b\in \At} $ the
canonical image of $\Omega^1(\At)$ in $\cH^{(1,0)}(\At)$. Then  
$\At \ni a \mapsto \partial_\tau(a) \in \gO^{(1,0)}(\At)$ 
defines an unbounded operator from $ \cH_0(\At)$ to  $\cH^{(1,0)}(\At)$.
By \cite[Lemma 1.5]{ConMos2011},  $\cH^{(1,0)}(\At)$ can be identified
with $\cH_0(\At)$, via the map of $\At$--modules 
$\kappa: \cH^{(1,0)}(\At) \to \cH_0(\At)$ extending the assignment
\begin{align} \label{Eq:canon-id}
\gO^{(1,0)}(\At)\ni \sum a \, \partial_\tau(b) \mapsto 
\kappa \left(\sum a \, \partial_\tau(b)\right) :=\sum a\, \gd_\tau (b) \in \cH_0(\At).
\end{align} 
Under this identification, $\partial_\tau$ becomes the derivation $\gd_\tau$,
 viewed as unbounded operator from $\cH_0(\At)$ to $\cH^{(1,0)}(\At)$.

The $\operatorname{Spin}^c$ Dirac operator associated to the fixed complex structure and to the
flat metric on $T_\gt^2$ is the $(\partial_\tau + \partial_\tau^*)$--operator  
\begin{equation}\label{Eq:spectrip1}
 D= \left(
  \begin{array}{cc}
    0 & \partial_\tau^*\\
    \partial_\tau & 0\\
  \end{array}
\right)    \quad  \text{acting on } \quad \tilde{\cH} = \cH_0(\At) \oplus \cH^{(1,0)}(\At) \, ,
\end{equation}
which is isospectral to the usual $\operatorname{Spin}^c$ Dirac operator on the complex torus
$\C/\Gamma$, $\Gamma = \Z + \tau \Z$.
Both the left and the right action of $A_\gt$ on $\tilde{\cH}$ are unitary and 
give rise to spectral triples, $(A_\gt, \tilde{\cH}, D)$ and $(A_\gt^{\rm op},\tilde{\cH}, D)$.
The transposed of the latter in the sense of \cite[Def.~1.4]{ConMos2011}
is isomorphic to the former, but with the opposite grading. 

\begin{remark}\label{D-forms}
Note that the space of $1$--forms corresponding to the first spectral triple,
$\gO_D (\At) = \{\sum a [D, b] ; \, a, b \in \At \}$, is isomorphic as
$\At$--bimodule to the direct sum
$\gO^{(1,0)}(\At) \oplus \gO^{(0,1)}(\At)$, where
$\gO^{(0,1)}(\At) :=  \{\sum a \partial_\tau^*(b) ; \, a, b \in \At\}$.
\end{remark}

\subsection{Heisenberg modules and their holomorphic structures}
As mentioned in the introduction, we view the basic Heisenberg modules over
the noncommutative space $T^2_\gt$ as analogues of the fundamental complex
spinor bundles over $T^2$.  We recall below their definition as well as their
main features.

\subsubsection{Basic Heisenberg modules}
\label{SSSBasicHeisenberg}
 
We generally adopt the notation in~\cite{PolSchwa2003}, except for a few
changes which makes it compatible with Rieffel's general 
construction~\cite[Sec.~2,3]{Rie1988}.  For the convenience of the reader we
have collected a synopsis of the latter in the Appendix \ref{AppA}. To make
the connection between the Appendix and the following material see in
particular the Example \ref{SAppExample}.

Fix $g=\begin{pmatrix} a&b \\ c & d \end{pmatrix}\in \GL(2,\Z)$, and let
$g\gt=\frac{a\gt +b}{c\gt +b}$.  In the sequel we use the abbreviation $\gt'$
for $g\gt$. Unless otherwise specified, it will be assumed that $c \neq 0$.
For an explanation of the following formulas see Section \ref{SAppExample}. 

Let $ \cE(g,\gt)$ be the Schwartz
space $\sS(\R)^{|c|} \equiv \sS(\R\times \Z_c)$, with $\Z_c:=\Z/c\Z$.  
The (smooth) algebra $\At$
acts on the right by
\begin{equation}\label{Eq05251}  
        (fU_1)(t,\ga):= e^{\tpii (t- \frac{\ga d}{c})} f(t,\ga)\, , \quad
        (fU_2)(t,\ga):= f(t-\frac{c\gt + d}{c} ,\ga-1);
\end{equation}
the algebra $\Atp$, with generators denoted $V_1, V_2$, acts on the left by
\begin{equation}\label{Eq05234}  
        (V_1f)(t,\ga):= 
          e^{\tpii \bigl(\frac{t}{c\gt+d}- \frac{\ga}{c}\bigr)} f(t,\ga)\,, \quad
        (V_2f)(t,\ga):= f(t-\frac 1c,\ga-a).
\end{equation}
By analogy with vector bundles over an elliptic curve,  one defines the rank,
degree and slope of $\cE(g,\gt)$ as the numbers 
\begin{equation}\label{rk-deg-slop}
\begin{split}
 \rk\cE(g,\gt) &= c\gt +d \, , \quad
\deg\cE(g,\gt) = c \, , \text{ resp. }\\
  \operatorname{\mu}\bl\cE (g,\gt)\br&:=\frac{\deg\cE(g,\gt)}{\rk\cE(g,\gt)}= \frac1{\gt+\frac dc}.
\end{split}
\end{equation}
For $f_1,f_2\in \cE(g,\gt)$ let
\begin{equation}\label{Eq05236}  
 \scalarL{f_1,f_2}:= \int_{\R\times \Z_c} \ovl{f_1(t,\ga)} f_2(t,\ga) dt d\ga 
\end{equation}
denote the $L^2$--scalar product, where the integration is with respect to the
Lebesgue measure on $\R$ and the counting measure on $\Z_c$.  This determines
an $\At$--valued  inner product 
$\innr{\cdot,\cdot}:\cE(g,\gt)\times \cE(g,\gt)\longrightarrow \At$
(antilinear in the first argument), and also an $\Atp$--valued inner product
$\innl{\cdot,\cdot}:\cE(g,\gt)\times \cE(g,\gt)\longrightarrow \Atp$
(antilinear in the second argument), both uniquely characterized by the
identity
\begin{equation}\label{Eq05237}  
 |c\gt + d|\, \varphi'_0(\innl{f_2,f_1})\, = \, \scalarL{f_1,f_2}\,
 =\, \varphi_0(\innr{f_1,f_2}),
 \end{equation}
\cf\cite[Sec.~1]{PolSchwa2003}.
 
The completion $E(g,\gt)$ of $\cE(g,\gt)$ with respect to  
$\|\innr{\cdot,\cdot}\|^{1/2}$ is a full right $C^*$--module over $A_\gt$,
whose endomorphism algebra is $\End_{A_\gt}(E(g,\gt)) = A_{\gt'}$;
at the same time, it is a full left $C^*$--module over $A_{\gt'}$, and
$\End_{A_\gt'}(E(g,\gt)) = A_{\gt}$.  

On the other hand, the
completion $ \cH_0(g,\gt)$ of $\cE(g,\gt)$ with respect to the scalar product
\Eqref{Eq05236} is the Hilbert space $L^2(\R\times \Z_c)=L^2(\R)^{|c|}$.  Note
that the Hilbert space $ \cH_0(g,\gt)$ also coincides with the interior tensor
product $E(g,\gt)\otimes_{A_\gt} L^2(\At,\varphi_0)$.

The bimodules $\cE(g,\gt)$, $\gt \in (0, 1)$, $g \in \GL(2, \Z)$, which are
called \emph{basic}, correspond to pairs of relatively prime integers $(d,c)$.
A similar construction
obviously applies to any pair $(d, c) \in \Z^2$, but the resulting bimodule is
isomorphic to a direct sum of $m$ copies of the module associated to the pair
of relatively prime integers $(\frac{d}{m}, \frac{c}{m})$, where $m$ is the
greatest common divisor of the pair $(d, c)$. 
 
The standard connection $\nabla^{\cE}$ on $\cE = \Egt$ is given by the
directional covariant derivatives 
$\nabla_j : \cE(g,\gt) \rightarrow \cE(g,\gt)$, $j=1,2$, defined by 
\begin{equation}\label{Eq052311}  
 (\nabla_1 f)(t,\ga) := \frac{\pl}{\pl t} f(t,\ga),\quad
 (\nabla_2 f)(t,\ga) := {\tpii}\cdot {\mu\bl\cE (g,\gt)\br} \cdot t \cdot  f(t,\ga) .
\end{equation}
It is actually a bimodule connection, in the sense that it is compatible with the
standard derivations $(\delta_1 , \delta_2)$ on $\At$, as well as with the
(normalized) derivations $\delta_j':=\frac{1}{\rk\cE(g,\gt)} \delta_j$, $j=1,2$,
on $\Atp$. Specifically, for $f\in \cE(g,\gt), a\in \Atp, b\in\At$, one has
\begin{equation}\label{Eq052312}  
 \nabla_j(a\cdot f\cdot b) = a \cdot (\nabla_j f) \cdot b +  \delta'_j(a) \cdot f\cdot b +
 a \cdot f\cdot \delta_j(b) , \qquad j=1,2.
\end{equation}
Furthermore, the connection is {\em Hermitian}, meaning that 
\begin{equation}\label{Eq052313}   
  \begin{split}
   \delta_j\bigl( \innr{f_1,f_2}\bigr)&= \innr{\nabla_j f_1,f_2}+ \innr{f_1,\nabla_j f_2} , \\
      \delta_j'\bigl(\innl{f_1,f_2}\bigr)&=  \innl{\nabla_j f_1,f_2} + \innl{f_1,\nabla_j f_2}.
  \end{split}
\end{equation}

The bimodules of the form $\Egt$ are called {\em basic Heisenberg modules}.
The designation ``Heisenberg modules'' comes from the fact that the standard
connection satisfies the Heisenberg commutation relation 
\begin{equation} \label{const-curv}
[\nabla_1,\nabla_2]=\, {\tpii}\cdot{\mu\bl\cE (g,\gt)\br}\cdot \Id ,
\end{equation}
which expresses the fact that it has {\em constant curvature}. 

This relation is  the infinitesimal form of a unitary representation of the
(polarized) Heisenberg group $H = \R^2\times \R$ with multiplication law
\begin{align} \label{H-law}
(x, \zeta) \cdot (y, \eta) \, := \, \left(x+y, \zeta +\eta + \gs (x, y)\right) ,
\end{align} 
where $ \gs : \R^2 \to \R$ is the standard symplectic form  
\begin{equation}\label{EqProRepSymForm}  
 \gs (x,y)\, := \, x_1 y_2-x_2 y_1 , \qquad  x =(x_1, x_2) , \, y =(y_1, y_2) \in \R^2 .
\end{equation} 
Indeed, one easily checks that the formula
\begin{equation}\label{pi-0}   
   \bigl( \pi_0(x, \zeta)f\bigr)(t):= e^{\pi i {\mu \bl\cE(g,\gt)\br}\, 
   (\zeta  + x_1 x_2 + 2x_2 t)}\;
   f(t+x_1) , \quad f \in L^2(\R), 
\end{equation}
defines a unitary representation of $H$ on $L^2(\R)$. This is in fact the
unique (up to unitary equivalence) representation of $H$ with central
character 
\begin{align*}
\text{Center} (H) \ni (0, \zeta) \mapsto e^{\pi i {\mu \bl\cE(g,\gt)\br}\zeta} \in T .
\end{align*} 
By a slight abuse, we shall continue to use the same notation for its multiple
$\pi_0 \otimes \Id$ acting on the Hilbert space $L^2(\R\times \Z_c) = L^2(\R)
\otimes \C^{|c|}$.  The space of $C^\infty$--vectors of the latter is
precisely $\cE(g,\gt)$. 

For later use, let us point out that there is a $2$--parameter family of
unitarily equivalent representations to $\pi_0$.  Indeed, associating to each
$w\in\R^2$ the unitary character $\chi_w :H \to T$,
\begin{align*}
\chi_w (x, \zeta) \, := \, e^{i   \inn{w,x}} ,
\end{align*}
one forms the interior tensor product representation
\begin{equation}\label{eq:1209061}  
\pi_w \, := \, \chi_w \otimes \pi_0  ,
  \end{equation}
which has the same central character as $\pi_0$.  It will be convenient to
treat this family as projective representations of 
$\R^2$, $\pi_w(x) := \pi_w(x, 0)$, satisfying the cocycle identity
\begin{equation}\label{Eq:1209062}  
    \pi_w(x)\pi_w(y)= e^{\pi i\mu\bl\cE(g,\gt)\br\gs(x,y)} \; \pi_w(x+y) .
\end{equation}
At the infinitesimal level one has
 \begin{align}\label{Eq:1209064}   
 \begin{split}
    \frac{\pl}{\pl x_j}\big|_{x=0} \pi_w(x) &= \nabla_j + i w_j,\qquad
    \text{analogous to} \\ 
      \frac{\pl}{\pl x_j}\big|_{x=0} \ga_x  &= \delta_j, \qquad
     \frac{\pl}{\pl x_j}\big|_{x=0} \ga_x' = \delta_j', \qquad j=1,2 .
 \end{split}
\end{align}
In particular, \Eqref{Eq052311} and \Eqref{Eq052312} follow from
\Eqref{Eq:1209064}.  We also note that $\pi_w$ implements the action of $T^2$
on the $\Atp$ and $\At$, \ie satisfies 
\begin{equation}\label{Eqpiw}
\begin{split}
       \pi_w(x) V_j \pi_w(x)^* &= e^{\frac{\tpii}{\rk\cE(g,\gt)} x_j} V_j=:
       \ga_x'(V_j),\\
       \pi_w(x) U_j \pi_w(x)^* &= e^{\tpii x_j}
       U_j=: \ga_x(U_j) .
\end{split}
\end{equation}

\subsubsection{$\operatorname{Spin}^c$ Dirac operators with coefficients in Heisenberg modules}   
Given a Heisenberg module as above $\cE= \cE(g,\gt)$ with $c >0$, and equipped
with its standard connection $\nabla^\cE$, one can apply the general recipe
(\cf\cite[VI.1]{Con1994}) to form the twisted version of the standard
$\operatorname{Spin}^c$ Dirac operator $D = \partial_\tau + \partial_\tau^*$ with
coefficients in $\cE$:
\begin{align}\label{Eq:E-twist}
\begin{split}
D_\cE&= \left(
  \begin{array}{cc}
    0 & D_\cE^{-}\\
D_\cE^{+} & 0\\
  \end{array} \right), \quad \text{where}  \quad  D_\cE^{-} = \left(D_\cE^{+}\right)^* \\ 
  \text{and} \qquad 
D_\cE^{+} (f\otimes a)&=\partial_\tau  (f)\, a \, +\, f \otimes \partial_\tau (a) , 
\qquad f\in\cE, \, a\in\At.
\end{split}
 \end{align} 
Here the standard connection is regarded as a map 
$\nabla^\cE : \cE \to \cE\otimes_{\At}\gO^1_D (\At)$. In view of the decomposition
of $1$--forms from Remark \plref{D-forms}, combined with the canonical identification 
(see~\cite[VI.3, Lemma 12]{Con1994}) of $\gO^1_D (\At)$ with off-diagonal matrices
in  $M_2 (\At)$,
the connection $\nabla^\cE$ on $\cE= \cE(g,\gt)$  splits into a \emph{holomorphic} 
and an \emph{anti-holomorphic} component
\begin{equation}\label{Eq:nabla-form}
\nabla^\cE= \partial_\cE \oplus \partial_\cE^* ,
\quad \text{where}  \quad \partial_\cE:=\nabla_1+\taubar\nabla_2 .
\end{equation}
The operator $\partial_\cE : \Omega^0(\cE)\to \Omega^{(1,0)} (\cE)
=\cE\otimes_{\At}\Omega^{(1,0)} (\At)$  defines a {\em holomorphic
structure}, \ie satisfies the property
\begin{equation} \label{holo-struct}
\partial_\cE (fa)=\partial_\cE  (f)\, a \, +\, f\, \gd_\tau (a) ,
\qquad f\in\cE, \, a\in\At. 
 \end{equation}
Together with the identity in the second line of \Eqref{Eq:E-twist}, 
which can be equivalently written as
\begin{align*}
D_\cE^{+} (f\, a\otimes 1)&=\partial_\cE (f)\, a \, +\, f \, \gd_\tau (a) \otimes 1 ,
\qquad f\in\cE, \, a\in\At ,
\end{align*}
this shows that
$D_\cE^{+}$ is just the extension of $\partial_\cE$ to an unbounded operator
from  $ \cH_0(g,\gt)$ to $\cH^{(1,0)} (g,\gt) := E(g,\gt)\otimes_{\At}\cH^{(1,0)}(\At)$. 
The latter can be identified to $ \cH_0(g,\gt)$ via the bimodule isomorphism
\begin{align} \label{Eq:canon-id-Eg}
\kappa_{g,\gt} =\Id \otimes_{\At}\kappa : E(g,\gt)\otimes_{\At}\cH^{(1,0)}(\At) \to
\cH_0(g,\gt) = E(g,\gt)\otimes_{\At}\cH_0(\At).
\end{align} 
It follows that
\begin{equation}\label{Eq:spectripEg}
 D_\cE= \left(
  \begin{array}{cc}
    0 & \partial_\cE^*\\
    \partial_\cE & 0\\
  \end{array} \right)    
\quad  \text{acting on } \quad \tilde{\cH}(g,\gt)
 = \cH_0(g,\gt) \oplus \cH^{(1,0)}(g,\gt)\, ,
\end{equation}
giving a spectral triple $\left(A_\gt^{\rm op}, \tilde{\cH}(g,\gt), D_\cE\right)$ for the 
right action of $A_\gt$.

As a matter of fact, the holomorphic component $\partial_\cE$ 
 satisfies the analogous identity to \Eqref{Eq052312},
\begin{equation} \label{holo-struct2}
\partial_\cE (afb)= a\, (\partial_\cE  f)\, b \, +\, \gd_\tau'(a) \, f\, b \, +\, a \, f\, \gd_\tau (b) , 
\quad f\in\cE, \, a\in\At',  \, b\in\At ,
 \end{equation} 
where $ \gd_\tau = \gd_1 + \ovl{\tau} \gd_2$ and  $ \gd_\tau' = \gd'_1 + \ovl{\tau} \gd'_2$.

Such a map is called in~\cite{PolSchwa2003} a \emph{standard holomorphic structure}.
There is a one-parameter family of such structures, given by
\begin{equation} \label{holo-struct-z}
\partial_{\cE,z}:=  \partial_\cE + z \cdot \Id ,
 \qquad  z\in\C . 
 \end{equation}
 It turns out that these are the only ones. Actually, an even stronger statement
 holds true.

\begin{prop} \label{uni-bi-holo}
 Assume $g\in\GL(2, \Z)$ with $c\not=0$ and let $\tilde\partial$ be a holomorphic structure
 on $\cE(g,\gt)$ which is compatible with
 some $\tilde\delta \in \mathfrak{g}(\Atp)$, \ie satisfies
 \begin{align*}
 \tilde\partial(af)= a \tilde\partial (f) + \tilde\delta(a)f , \qquad f\in \cE(g,\gt), \, a \in\Atp.
\end{align*}
 Then $\tilde\delta=\delta_\tau'=\delta_1'+\taubar\delta_2'$ and
 $\nabla=\partial_{\cE,z}$ \, for some
 $z\in\C$.
\end{prop} 
\begin{proof} In view of \Eqref{holo-struct2}, the
difference $\tilde\partial -\partial_\cE$  is $\At$--linear, hence
it is given by the left action of some $\go \in \Atp$.
Let us show that  $\tilde\partial - \partial_\cE$  is also $\Atp$--linear.  
As $\tilde\partial = \partial_\cE + \go \cdot$,
 for any $a\in\Atp, f\in\cE(g,\gt)$ we have on the one hand,
\begin{equation}\label{Eq052314}  
\tilde\partial(af)= \tilde\delta (a) f + a \partial_\cE f + a \go f ,
\end{equation}
and on the other hand
\begin{equation}\label{Eq05252}  
 \tilde\partial (af)= \partial_\cE (af) + \go a f = a \partial_\cE f + \delta_\tau'(a) f +
 [\go,a] f + a \go f.
\end{equation}
This implies $\tilde\delta = \delta_\tau'+ [\go,\cdot]$. Hence
$\tilde\delta-\delta_\tau'$ is an inner derivation of $\Atp$.
Since it is also a linear combination of $\delta_j'$, it follows
that $\tilde\delta=\delta_\tau'$. But then $\go$ is in the center
of $\Atp$, and so $\go \in \C \cdot \Id$.
\end{proof}

\begin{remark}\label{rem:double-triple}
The property \Eqref{holo-struct2} ensures that the operator $D_\cE$ also
gives rise to a spectral triple $\left(A_{\gt'}, \tilde{\cH}(g,\gt), D_\cE\right)$ for the 
left action of $A_{\gt'}$. 
 \end{remark}

\begin{remark}\label{rem:int-pert}
Using the family of holomorphic structures \Eqref{holo-struct-z} 
one could define a family of operators $\{D_{\cE, z} \, ; \, z \in \C \}$. 
The corresponding spectral triples would not be essentially different however,
because $D_{\cE, z}$ is merely an internal perturbation (in the sense of
\cite{Con1996}) of the operator $D_\cE$.
\end{remark}

\begin{prop}  \label{harmonic-osc}
Assume $\Im (\tau) > 0$ and $\deg (\cE) \neq 0$. 
\begin{itemize}
\item[(1)]
If $\mu(\cE) > 0$ then $\dim \Ker \partial_\cE = |\deg (\cE)|$ and $\Ker \partial^*_\cE = 0$;
if $\mu(\cE) < 0$ then $\dim \Ker \partial^*_\cE = |\deg (\cE)|$ and $\Ker \partial_\cE = 0$.

\item[(2)] 
The zeta function 
$\zeta_{\Lapl_\cE} (s) = \Tr \left(\Lapl_\cE^{-s}\right)$, $\, \Re(s) > 1$,
has a meromorphic continuation to $\C$; it is regular at $s=0$ and
one has
\begin{align}  \label{Eq:zeta0}
\zeta_{\Lapl_\cE} (0) &=  - \frac12 |\deg (\cE)| , \\ \label{Eq:zetaprime0}
\zeta'_{\Lapl_\cE} (0)
   &= \frac12 |\deg (\cE)| \cdot  \log \bl 2|\mu(\cE)| \Im(\tau)\br, \\
\Res_{s=1} \zeta_{\Lapl_\cE}(s) & = 
             \frac{|\rk(\cE)|}{4\pi\Im\tau}. \label{EqZetaConstant}
\end{align}
\end{itemize}
\end{prop} 

\begin{proof} 
Claim (1), also proved in~\cite[Proof of Prop.~2.5]{PolSchwa2003}, is easy
to justify. By its very definition, $\partial_\cE$ is a direct sum of $|\deg (\cE)|$ copies of 
the operator $\, D= \frac{d}{d t} + \tpii \mu(\cE) \ov{\tau} \, t$. 
The latter has $1$--dimensional kernel in $\sS (\R)$ and no cokernel 
when $ \mu(\cE) > 0$, since then 
$\Re(\tpii \mu(\cE) \taubar) = 2 \pi  \mu(\cE) \Im (\tau) > 0$. 

When $ \mu(\cE) < 0$
the same argument applies to the operator $\partial^*_\cE$.

Claim (2) also follows from a routine calculation involving the harmonic oscillator.
Indeed, the Laplacian $\Lapl_\cE = \partial^*_\cE \partial_\cE$ is a 
direct sum of $|\deg (\cE)|$ copies of the harmonic oscillator
\begin{align} \label{Eq:ho}
H := D^*D =
 - \frac{d^2}{d t^2} + 4 \pi^2 \mu(\cE)^2 |\tau|^2 t^2 
 - 4 \pi i \mu(\cE) \Re(\tau) \, t\, \frac{d}{d t}- \tpii \mu(\cE) \taubar \Id . 
\end{align}
One easily checks that  $[D, D^*] = 4\pi  \mu(\cE) \Im (\tau)\Id$, and so
$DD^* = H + 4\pi  \mu(\cE) \Im (\tau)\Id$,
which implies
\begin{align} \label{forward}
HD^* = D^*(DD^*) = D^* \big(H + 4\pi  \mu(\cE) \Im (\tau)\Id\big) .
\end{align} 
On the other hand, \,
$DH = (DD^*)D = \big(H + 4\pi  \mu(\cE) \Im (\tau)\Id\big)D$, \,
whence
\begin{align} \label{backward}
HD = D \big(H - 4\pi  \mu(\cE) \Im (\tau)\Id\big) .
\end{align}
Thus, by \Eqref{forward}
$D^*$ shifts forward each eigenspace $V_\lambda$ of $H$ to 
$V_{\lambda + 4\pi  \mu(\cE) \Im (\tau)}$, while by \Eqref{backward}
 $D$ shifts backward
$V_\lambda$ onto $V_{\lambda - 4\pi  \mu(\cE) \Im (\tau)}$. 
Since $\dim (\Ker D) =1$ for $\mu(\cE)>0$ resp. $\dim (\Ker D^*) = 1$
for $\mu(\cE)<0$, one concludes that the spectrum of $H$ is 
\[
\spec H = 4\pi  |\mu(\cE)| \Im (\tau)  \begin{cases}
                \Z_+, & \mu(\cE)>0,\\
                \Z_+\setminus\{0\},& \mu(\cE)<0,
    \end{cases} 
\]
and each eigenvalue has multiplicity $1$.
In any case, the $\zeta$--function is
\[
\zeta_{\Lapl_\cE} (s) = |\deg(\cE)|\cdot \bl 4\pi |\mu(\cE)\Im\tau\br^{-s} \cdot
   \zeta_R(s)
\]
with the Riemann zeta function $\zeta_R(s)=\sum_{n=1}^\infty n^{-s}$.

The equations \Eqref{Eq:zeta0}, \eqref{Eq:zetaprime0} and
\eqref{EqZetaConstant} immediately follow from the corresponding values of the
Riemann zeta function,
$\Res_{s=1} \zeta_R(s)=1, \zeta_R(0) = -\frac12$ and  $\zeta_R'(0) = -\frac12 \log (2\pi)$.
\end{proof}

\subsection{Conformal twisting and curvature}
\subsubsection{Conformal change of background metric}

We begin by constructing, in the manner of~\cite[\S 1.5]{ConMos2011}, 
a bi-spectral triple associated to a `noncommutative 
 Weyl factor' $k=e^{h/2} \in \At$ with $h=h^*\in \At$. 
 
 Let $\varphi = \varphi_h$ be defined by
\begin{align*} 
  \varphi (a) \, = \, \varphi_0 (a e^{-h}) , \qquad a \in \At
\end{align*}
and let $\cH_\varphi = L^2(\At,\varphi)$ be the Hilbert space completion of
$\At$ with respect to the inner product
\begin{align*} 
  (a, b) \, = \, \varphi (a^*b) , \qquad a, b \in \At .
\end{align*}
Denote by $\cH_{\cE,\varphi} = \cH_\varphi(g,\gt)$ the interior tensor product 
$E(g,\gt)\otimes_{A_\gt} L^2(\At,\varphi)$ and by  $\cH_{\cE}^{(1,0)}=
\cH_0(g,\gt)$ the interior tensor product   
$E(g,\gt)\otimes_{A_\gt}\cH_0$, where $\cH_0 = L^2(\At,\varphi_0)$. 
Both are $(A_{\gt'}, A_{\gt})$--bimodules but the natural (right) action of $A_\gt$ on 
$\cH_{\cE,\varphi}$ is no longer unitary. To rectify this, one replaces it,
as in~\cite[\S1.5]{ConMos2011}, by the unitary action
\begin{align*} 
a^{\rm op} (f) \, = \, f \cdot k^{-1} a k \, , \quad a \in \At, \,  f \in \Egt .
\end{align*}
 
In the present setting, Lemma 1.7 in~\cite{ConMos2011} has the following counterpart:
\begin{lemma}\label{Unitary}
The right action 
of $k=e^{h/2}$ on $\Egt$,
\begin{align*} 
W (f) \, = \, f \cdot k \, , \qquad f \in \Egt ,
\end{align*}
extends to an isometry $W: \cH_0(g,\gt) \to  \cH_{\cE,\varphi}$, which
establishes an isometric $(\Atp, \At)$--bimodule isomorphism between 
$\cH_0(g,\gt)$ and $\cH_{\cE,\varphi}$.
\end{lemma} 

\begin{proof} Indeed, if $\, f , g \in \Egt$ then
\begin{align*} 
 \scalar{W(f), W(g)}_{\cH_{\cE,\varphi}}
 = \varphi(\innr{g\cdot k,f\cdot k})=\varphi_0(\innr{f,g})
 =\scalar{f, g}_{\cH_0(g,\gt)} .
\end{align*}
By its very definition $W$ is $\Atp$--linear. On the other hand, for any  $a \in \At$
\[
a^{\rm op}(W(f)) =  f \cdot k  \cdot k^{-1} a k = W(f \cdot a) .\qedhere
\]
\end{proof} 

Let now $\partial_{\cE, \varphi}$ be the closure of   $\partial_\cE$ viewed as unbounded 
operator from  $\Egt \subset \cH_\varphi(g,\gt)$ to  $\cH^{(1,0)}(g,\gt) $. 
The analogue of~\cite[Cor. 1.9 (ii)]{ConMos2011} reads as follows.

\begin{prop}\label{bi-spec-trip}
The operator
\begin{align}
D_{\cE, \varphi} \, =\,  
\begin{pmatrix}
  0   &   \partial_{\cE, \varphi}^*  \\
\partial_{\cE, \varphi} & 0
\end{pmatrix}  \quad \text{acting on} \quad
\tilde{\cH}_{\cE,\varphi} = \cH_{\cE,\varphi} \oplus \cH^{(1,0)}(g,\gt) 
\end{align}
gives rise to a {\em twisted} graded spectral triple 
$\left(A_\gt^{\rm op},\, \tilde{\cH}_{\cE, \varphi}, \, D_{\cE, \varphi}\right)$ with respect 
to the right action of $A_\gt$.
On the other hand the left action of $A_{\gt'}$ yields an {\em ordinary} graded
spectral triple $\left(A_{\gt'},\, \tilde{\cH}_{\cE, \varphi}, \, D_{\cE, \varphi}\right)$.
\end{prop} 

\begin{proof}
The twisting by the automorphism $\gs \in \operatorname{Aut} (A_\gt)$, 
 $\, \gs (a) = k^{-1} a k$, becomes self-evident when the operator $D_{\cE, \varphi}$
 is transferred via the isometry $W$. Indeed, denoting 
 $\, \tilde{W} = W \oplus \kappa_{g,\gt}^{-1}$, one has
\begin{align} \label{W-conj}
\tilde{W}^* D_{\cE, \varphi} \tilde{W} =
\begin{pmatrix}
  0   &   R_k\partial_{\cE}^*  \\ 
\partial_{\cE}R_k & 0
\end{pmatrix} \, \text{acting on} \,
\tilde{\cH}(g,\gt) = \cH_0(g,\gt) \oplus \cH_0(g,\gt),
\end{align} 
 \end{proof}
 
\begin{lemma}\label{AntiUnitary}
 For $g\in\GL(2, \Z)$ with $c\not=0$ the map defined by  
\begin{equation}\label{Eq05246}  
 J_{g,\gt}(f)(x,\ga)=   \ovl{f((c\gt+d)x,\, -d^{-1}\ga)} 
\end{equation}
is an antiisomorphism of $C^*$--bimodules
$\cE(g,\gt)\to \cE(g^{-1}, g\gt)$.
More precisely, $J=J_{g,\gt}$ satisfies the following identities
for $f , f_1, f_2 \in\Egt, a\in\Atp, b\in\At$:
\begin{align}
 J(a f b)& = b^* J(f) a^*,\label{Eq05247a}\\
      \tensor[_{\At}]{\langle}{} J(f_1), J(f_2) \rangle 
      & = \innr{f,g},\label{Eq05247b}\\
      \langle J(f_1),J(f_2)\rangle_{\Atp} &
      = \innl{f_1, f_2},\label{Eq05247c}\\
     |c\gt+d| \scalarL{J(f_1),J(f_2)} &= \scalarL{f_2,f_1}.\label{Eq05247d}
\end{align}
In addition, \, $J_{g,\gt}^{-1} = J_{g^{-1},g\gt}$.
 \end{lemma} 
 
\begin{proof} This is checked by direct calculation.
 Note that \Eqref{Eq05247c}, \eqref{Eq05247d} follow
 from \Eqref{Eq05247a} and \eqref{Eq05247b}.
\end{proof}

\begin{prop}\label{TST} 
The transposed of the twisted spectral triple 
$\left(A_\gt^{\rm op}, \tilde{\cH}(g,\gt), D_{\cE, \varphi}\right)$  is
isomorphic to the twisted spectral triple 
$\left(A_\gt, \, \cH_0(g^{-1},\gt')\oplus
\cH_0(g^{-1},\gt'), \,\ovl{D}_{\cE', k}\right)$ 
where  
\begin{align} \label{Trans-op}
&\ovl{D}_{\cE', k} \, =\, - \rk(\cE')
\begin{pmatrix}
  0   &   k {\partial}_{\cE'} \\
  {\partial}_{\cE'}^* k & 0 
\end{pmatrix} , \quad
 \text{with} \quad {\partial}_{\cE'} =  \nabla'_1 + \ovl{\tau}  \nabla'_2 .
\end{align}
In turn, the spectral triple $\left(A_{\gt'},\, \tilde{\cH}_{\cE, \varphi}, \, D_{\cE, \varphi}\right)$
is isomorphic to the spectral triple $\left(A_{\gt'}^{\rm op}, \, \cH_0(g^{-1},\gt')\oplus
\cH_0(g^{-1},\gt'), \,\ovl{D}_{\cE', k}\right)$. 
\end{prop} 

\begin{proof}
The conjugate of the operator $\partial_{\cE}$ by the above 
antiunitary isometry  
is related to $\partial_{\cE'}$ by the identity  (see~\cite[p. 175]{Pol2004}): 
 \begin{align} \label{eq:push-conn}
J_{g,\gt}\circ \partial_{\cE}^*\circ J_{g,\gt}^{-1} \, &= \,  
-\frac{1}{\rk\cE(g,\gt)} \,{\partial}_{\cE'} .
\end{align}
Setting  $\, \tilde{J} = J_{g,\gt} \oplus \left(- J_{g,\gt}\right)$, 
it follows that 
\begin{align} \label{JW-conj}
\tilde{J}\tilde{W}^* D_{\cE, \varphi} \tilde{W} \tilde{J}^{-1} =
-\frac{1}{\rk\cE(g,\gt)}
\begin{pmatrix}
  0   &   k \partial_{\cE'} \\
  \partial_{\cE'}^* k & 0
\end{pmatrix} .
\end{align}
It remains to notice that $\frac{1}{\rk\cE(g,\gt)} =  \rk(\cE')$.
\end{proof}
 
The Heisenberg module underlying the above bi-spectral triple 
is the $(\At, \Atp)$--bimodule $\cE' = \cE(g^{-1},\gt')$, with $\At$ 
identified to $\End_{\Atp} (\cE')$. Moreover, as can be seen 
from \Eqref{Eq05237},  
in this picture turning on the Weyl factor $k =e^{\frac{h}{2}} \in \At$
amounts to passing to the Hermitian structure
\begin{align}  \label{k-met}
 (f'_1,f'_2)_k:=\innr{k f'_1,f'_2} , \qquad    f'_1, f'_2 \in \cE' .
\end{align}

It will be shown below that such a change uniquely determines a 
holomorphic connection on $\cE'$ which is Hermitian with respect
to the new metric. However, for notational convenience, we
momentarily switch the roles of the two Morita equivalent algebras
and phrase the uniqueness result in terms of
the original Heisenberg $(\Atp, \At)$--bimodule $\cE = \cE(g,\gt)$.
 
\pagebreak[3]
\begin{prop} \label{H-conn}
 Let  $K\in\End_{\At}(\cE)=\Atp$
be positive definite and consider the inner product 
\begin{equation}\label{Eq052315}  
 H(f,g):=\innr{K f,g}. 
\end{equation}
Then there is a unique connection $\nabla^K$ on $\cE$ such that
\begin{align}  \label{Eq052316c}
 \partial^K = \nabla_1^K+\taubar \nabla_2^K \quad
 \text{is a holomorphic structure}, 
 \end{align}
and which is Hermitian with respect to $H$, \ie satisfies
\begin{align}  \label{Eq052316b}
     \delta_j H(f_1,f_2)\, = \, H(\nabla_j^K f_1,f_2)+ H(f_1,\nabla_j^K f_2) ,
     \quad j=1,2.
  \end{align}
Moreover, if $\partial^K$ is a standard holomorphic structure then
the connection $\nabla^K$ has constant curvature iff
$K$ is a multiple of the identity.  
\end{prop}

\begin{proof} We write $\partial^K=\partial_\cE+Z$ with $Z \in \End_{\At} (\cE)$
and make the Ansatz 
\begin{equation}\label{Eq05241}  
            \nabla_j^K=  \nabla_j +\go_j , \qquad \go_j\in\Atp .
\end{equation}
Then
\begin{multline}\label{Eq05242}  
  \delta_j H(f,g)= \delta_j \innr{K f,g} = \innl{\nabla_j(Kf),g}+
  \innr{Kf,\nabla_j g}\\
  = \innl{(\delta_j'K)f,g}+ H(\nabla_j^Kf,g)+ h(f,\nabla_j^K g)
    -\innr{K \go_j f,g}-\innr{Kf,\go_j g}.
\end{multline}
Thus \Eqref{Eq052316b} is satisfied iff
\begin{equation}\label{Eq05243}  
   \delta_j'K =  K \go_j + \go_j^* K.
\end{equation}
\Eqref{Eq052316c} is equivalent to 
\begin{equation}\label{Eq05244}  
      \go_1+\taubar \go_2 = Z.     
\end{equation}
It then follows
\begin{equation}\label{Eq05245}  
 \begin{split}
   \delta_\tau'K&=K(\go_1+\taubar¸\go_2)+(\go_1^*+\taubar \go_2^*) K\\
   &= KZ+Z^*K+(\taubar-\tau)\go_2^*K.
  \end{split}  
\end{equation}
Since $\Im\tau\not=0$ this determines $\go_2$ and thus $\go_1$ uniquely.

Conversely, defining $\go_2$ by \Eqref{Eq05245} and $\go_1$
by \Eqref{Eq05244} one easily sees that \Eqref{Eq05243},
\eqref{Eq05244} and hence \Eqref{Eq052316c}, \eqref{Eq052316b}
are fulfilled.

To prove the last claim we recall, \cf\Eqref{const-curv}, that
the standard connection has constant curvature. 
The holomorphic structure being standard implies that $Z$ is
a central element of $\At$, hence $Z=z\cdot 1$. 
The curvature
of $\nabla_j^K$ is given by
\begin{equation}\label{Eq05248}  
 [\nabla_1^K,\nabla_2^K]= [\nabla_1,\nabla_2]+\delta_1' \go_2+\delta_2'\go_1.
\end{equation}
Note that in view of \Eqref{Eq05244} and $Z=z I$ we have $[\go_1,\go_2]=0$.
The curvature is constant if and only if 
\begin{equation}\label{Eq05249}  
 \delta_j'( \delta_1'\go_2+\delta_2'\go_1)=0\quad \text{for } j=1,2.
\end{equation}
Multiplying by $\rk\cE(g,\gt)^2$ we see in view of \Eqref{Eq05244} 
that this is equivalent to
\begin{equation}\label{Eq052410}  
 \delta_j(\delta_1 -¸\taubar \delta_2)\go_2=0,\qquad \text{for } j=1,2. 
\end{equation}
Writing $\go_2=\sum_{j,l\in\Z} a_{j,l} V_1^k V_2^l$ we find
\begin{equation}\label{Eq052411}  
 (\delta_1-\taubar \delta_2)\go_2 =\tpii \sum_{j,l\in \Z} (j-\taubar l) a_{j,l}
 V_1^k V_2^l. 
\end{equation}
This is constant iff $(j-\taubar l) a_{j,l}=0 $ for $(j,l)\not = (0,0)$. But
since
$\Im \tau\not=0$ this is equivalent to $\go_2$ itself being constant, say
$\go_2=\gl I$. By
\Eqref{Eq05244} $\go_1$ is constant, too. Write $\tilde \gl=z+\ovl{z}+(\taubar
-\tau)\gl$. Then in view of \Eqref{Eq05245} we see that $K$ satisfies
the equation $\delta_\tau' K=\tilde\gl K$. Expanding $K$ into its Fourier
series it then follows that $K$ is a constant multiple of the identity.

For such $K$ we have $\delta_\tau'K=0$ and \Eqref{Eq05245} implies 
$\go_2\in\C\cdot I$ and thus, by \Eqref{Eq05244}, also $\go_1\in\C\cdot I$.
\end{proof}

\subsubsection{Curvature of twisted spectral triple}  

The square of the operator \eqref{Trans-op} is
\begin{align} \notag
&\ovl{D}_{\cE', k}^2 \, = \, -\rk(\cE')^2
\begin{pmatrix}
\Lapl^{+}_{\cE', k} & 0\\
 0 &   \quad \Lapl^{-}_{\cE', k} 
\end{pmatrix} , \qquad  \text{where} \\ \label{Laps}
&\Lapl^+_{\cE', k}:=  k \partial_{\cE'} \partial_{\cE'}^* k 
\quad \text{and} \quad \Lapl^{-}_{\cE', k}:= \partial_{\cE'}^*  k^2   \partial_{\cE'} .
\end{align}

The curvature functionals are
defined, as in~\cite{ConMos2011} (see also Remark 5.4 therein),
by means of the constant term in the asymptotic expansion 
\begin{align} \label{heat-asy}
  \Tr \left(a\, e^{-t \Lapl^\pm_{\cE', k}} \right)\, \sim_{t \searrow 0} \,
  \sum_{j=0}^\infty {\rm a}_{2j}(a, \Lapl^\pm_{\cE', k})\, t^{j-1} ;
\end{align} 
specifically, the {\em curvature functional} associated to the 
twisted spectral triple $\left(A_\gt, \, \cH_0(g^{-1},\gt')\oplus
\cH_0(g^{-1},\gt'), \,\ovl{D}_{\cE', k}\right)$ is the functional
\begin{align} \label{total-curv}
\At \ni a \mapsto \cR^\pm_{\cE', k} (a) \, := \, {\rm a}_{2}(a, \Lapl^\pm_{\cE', k}) .
\end{align}
The existence of the expansion  \eqref{heat-asy}  follows from the
pseudodifferential calculus for twisted crossed products (Section
\ref{SPsiDOMul}), the resolvent expansion for pseudodifferential multipliers
(Section \ref{SResExp}, Theorems \ref{TResExp}, \ref{TbInt}) and its
realization on Heisenberg modules (Section \ref{SEffPsiDO}, Theorem
\ref{TEffHeatExp}). Theorem \ref{TEffHeatExp} shows
moreover that the  functional $\cR^\pm_{\cE', k} $ is 
given by a density $\cK^\pm_{{\cE', k}} \in \At$ 
(see \Eqref{c-density} below) with respect to the 
\MLrevision
\mpar{Referee item 3}
natural trace on $\End_{\Atp}(\cE') = \At$, namely 
\begin{align} \label{nat-tr}
\vp_{\cE'} \, := \, |\rk(\cE')| \vp_0 .
\end{align}
Adjusted by a normalization factor (to be explained below), this
density, denoted $\cK^\pm_{{\cE', k}}$, is given by
\begin{align} \label{c-density}
 \cR^\pm_{\cE', k} (a) =     \frac{1}{4\pi \Im\tau} \vp_{\cE'} (a \, \cK^\pm_{{\cE', k}})
 =    \frac{|\rk(\cE')|}{4\pi \Im\tau}\, \vp_0 (a \, \cK^\pm_{{\cE', k}}) , \qquad a \in \At.
\end{align} 

\begin{remark} \label{tpii}
When comparing with~\cite{ConMos2011} one has to remember
that our basic derivations \eqref{Eq:Lie-act} are $\tpii$ multiples of
those used therein. Furthermore, there are different conventions for the
modular functions which lead to another overall factor of $-2$ for the one
variable functions and to $-4$ for two variable functions,
\cf Sec.~\ref{SSComparison} below.   
Multiplying the above factor $\frac{1}{4\pi \Im\tau}$ by
$-4\pi^2$ gives precisely the overall factor $- \frac{\pi}{\Im\tau}$ of
the expression (3.14) in~\cite{ConMos2011}. 
\end{remark}

The first main result of the paper computes the precise expression of these
densities, thus producing formulas for the {\em local curvature} of the
twisted spectral triple $\left(A_\gt, \, \cH_0(g^{-1},\gt')\oplus
\cH_0(g^{-1},\gt'), \,\ovl{D}_{\cE', k}\right)$. 
With a slight change of
notation for the curvature-defining functions introduced in \cite[\S 3.1]{ConMos2011}, 
 the result can be stated as follows. \footnote{For the classical case of a
Riemann surface, see \cite[\S 1.5]{Bos1987}, in particular formula (1.5.4).}

\begin{theorem}[Curvature densities] \label{main-curv}
Let $h=h^* \in \At^{\rm sa} $ and $k=e^{\frac{h}{2}}$. 
Then   
\begin{equation} \label{main-a2+}
 {\rm a}_{2}(a, \Lapl^\pm_{\cE', k})  \equiv  \cR^\pm_{\cE', k} (a) =  
      \frac{|\rk(\cE')|}{4\pi \Im\tau}\, \vp_0 (a \, \cK^\pm_{{\cE', k}}) , \quad a \in \At.
\end{equation}
where  
\begin{equation} \label{loc-curv}
\begin{split}
 \cK^\pm_{\cE', k} & =  
                        K_\pm(\nabla_h)(\Lapl_\tau(h))
                      + H_\pm^\Re(\nabla_h^1,\nabla_h^2)\bl\square^\Re (h)\br \\
       & \quad +H_\pm^\Im(\nabla_h^1,\nabla_h^2)\bl (\square^\Im(h)\br
           \mp  2 \pi \, \Im\tau  \, \mu(\cE')\, 1.
\end{split}
\end{equation}  
Here, $\nabla_h  =-\operatorname{ad} h$ is the modular derivative,
$\Lapl_\tau h = \delta_\tau \delta_\tau^*$ is the Laplacian
of the complex structure (\cf \Eqref{Eq:c-struct}),
\MLrevision
\mpar{Referee item 4}
\begin{equation*}
    \square^\Re(h)  = \frac 12  (\dtau h \cdot \dtau^* h
                              + \dtau^* h \cdot \dtau h),\quad
    \square^\Im(h)  =  \frac 12 (\dtau h \cdot \dtau^* h
                              - \dtau^* h \cdot \dtau  h),
\end{equation*}
and $\nabla_h^i, \, i=1, 2$, signifies that $\nabla_h$ is acting on the
$i$--th factor.

$K_+$ ($K_-$) equals $K_{0,0}$ ($K_{0,1}$) in \Eqref{EqKeps}, $H_+^\Re$
($H_-^\Re$) equals $H_{0,0}^\Re$ ($H_{0,1}^\Re$) in \Eqref{EqHReps}
and  $H_+^\Im$ ($H_-^\Im$) equals $H_{0,0}^\Im=0$ ($H_{0,1}^\Im$) in
\Eqref{EqHIeps}.
\end{theorem}

\begin{proof} With the modular operator $\modu := k^{-2} \cdot k^2 =
e^{\nabla_h}$ 
 we first note that 
$\Lapl^+_{\cE',k}= k \pl_{\cE'}\pl_{\cE'}^* k=
  \modu^{\frac 12}\bl k^2 \pl_{\cE'}\pl_{\cE'}^* \br$,
hence (see \ref{SSSConjArg}) 
\[
   \Tr\bl    a e^{-t \Lapl^+_{\cE',k}} \br
    = \Tr\bl  k  a k\ii  e^{-t k^2 \pl_{\cE'} \pl_{\cE'}^*} \br.
\]
We therefore apply Theorem \ref{TEffHeatExp} to the operator
$ k^2 \pl_{\cE'}\pl_{\cE'}^*$. Its symbol equals
(\cf Sec.~\ref{SSSDiffOp2} and Sec.~\ref{SSSetUp}) 
$k^2 |\eta|^2 + \frac 12 k^2 c_\tau$, where 
$c_\tau=[\pl_{\cE'},\pl_{\cE'}^*] = 4\pi \mu(\cE')\Im\tau$
is the structure constant (\cf \Eqref{EqudbarC}). 
According to Theorem \ref{TEffHeatExp} and Corollary
\ref{CA2} the constant term $\frac 12 k^2 c_\tau$ in
the symbol contributes to $a_2(a,\Lapl^+_{\cE',k})$
the summand
\[
\frac{N(\spi)}{4\pi \Im\tau} \varphi_0\bl k a k\ii (-k^{-2} \frac 12 k^2
c_\tau)\br
    = -\frac{|\rk(\cE')| \mu(\cE')}{2} \varphi_0(a)
    = -\frac{\mu(\cE')}{2} \varphi_{\cE'} (a). 
\]
Note that for the module $\cE'$ the universal constant
$N(\spi)$ (\cf Theorem \ref{TEffTrace}) equals $|\rk(\cE')|$.
The first large summand on the right of \Eqref{loc-curv} follows
from Corollary \ref{CA2} as well. The normalization constant in front
of the last expression in \Eqref{main-a2+}
is (\cf Theorems \ref{TEffTrace}, \ref{TEffHeatExp})
$\frac{N(\spi)}{4\pi\Im\tau} = \frac{|\rk(\cE')|}{4\pi\Im\tau}.$
\end{proof}

\begin{remark} 
An equivalent formulation of \Eqref{loc-curv} is 
\begin{align} \label{main-curv-comp}
 \cK^\pm_{\cE', k} \, = \,   \cK^\pm_k \, \mp \,  2 \pi \, \Im\tau \,\mu(\cE')\,1 ,
\end{align} 
where $ \cK^\pm_k$ is the intrinsic scalar curvature of $T^2_\gt$ (\cf
\cite[Sec.~4]{ConMos2011}) equipped 
with the conformal metric of Weyl factor $k=e^{\frac{h}{2}}$. Note though that
due to the change of conventions (\cf Remark \ref{tpii}) the
latter is a $-4\pi^2$ multiple of that defined in~\cite{ConMos2011}.
\end{remark} 
\medskip

\begin{cor}[Riemann-Roch density] \label{main-RR}
The Riemann-Roch density $\, \cK^\gamma_{\cE', k} =\cK^+_{\cE', k} - \cK^-_{\cE', k} $ 
is given by the formula
\begin{equation} \label{loc-curvRR}
\begin{split}
 \cK^\gamma_{\cE', k} & =  
                        K_\gamma(\nabla_h)(\Lapl_\tau(h))
                      + H_\gamma^\Re(\nabla_h^1,\nabla_h^2)\bl\square^\Re (h)\br \\
       & \quad +H_\gamma^\Im(\nabla_h^1,\nabla_h^2)\bl (\square^\Im(h)\br
           -  4 \pi \, \Im\tau  \, \mu(\cE')\, 1.
\end{split}
\end{equation}  
\end{cor}
The expressions of $K_\gamma, H_\gamma^\Re, H_\gamma^\Im$ are explicitly given
in \Eqref{EqKgamma}, \eqref{EqHRgamma}, \eqref{EqHIgamma}.

The case $g=\begin{pmatrix} 1 & 0 \\ 0 & 1 \end{pmatrix}$, \ie of the trivial ``line bundle'' with a standard holomorphic structure
was treated in~\cite{KhaMoa2014}.

\subsection{Ray-Singer determinant functional and curvature} \label{RSDetF}

\subsubsection{Ray-Singer determinant for Heisenberg bimodules}  
In the preceding section we started by performing a conformal change of metric
in the base. However, after passing to the transposed spectral triple, the
initial algebra $\At$ becomes the algebra of endomorphisms $\End_{\Atp}
(\cE')$. In the new picture, the {\em conformal change on the base} amounts to
a completely {\em general change of Hermitian metric} on the right
\mpar{Referee item 5: clarify ``completely general change''.
Well that's what it is. The bundle is in a sense one dimensional, thus
there is no difference between conformal change and general change ...}
$\Atp$--module $\cE'$.
\MLrevision
Indeed, any Hermitian structure on $\cE'$, viewed as a Hilbert C*-module over 
$\Atp$, can be obtained from the standard one by composition with a
positive $\Atp$-linear endomorphism, \ie a positive element in $\At$.

From now on we assume $\mu(\cE') > 0$.
By Proposition \ref{harmonic-osc} (1) this implies 
$$
\Ker \Lapl^+_{\cE'} = \Ker  \partial_{\cE'}^* = 0 .
$$
It then follows that $\Ker \Lapl^+_{\cE', k} = 0$ for any   
invertible $k =e^{\frac{h}{2}} \in \At^{\rm sa}$. 
Once the above choice of the sign of $\mu(\cE')$ was made, we will only deal with the
Hodge-Laplace operator on ``functions''. Accordingly,
we shall routinely omit the superscript $^+$
when referring to $\Lapl^+_{\cE', k}$, as well as to the rest of symbols
affected by the same notation. 

Having no zero modes, the zeta function of $\Lapl_{\cE', k}$ is
\begin{align*}
 \zeta_{\Lapl_{\cE', k}} (z)\, =\, \Tr \bl\Lapl_{\cE', k}^{-z}\br
 \, =\, \frac{1}{\Gamma (z)} \int_0^\infty t^{z-1}
 \,  \Tr\bl e^{-t \Lapl_{\cE', k}}\br\, dt .
\end{align*} 
Its value at $0$ is independent of the Weyl factor $k \in \At$, as shown by
the proof of \cite[Theorem 2.2]{ConMos2011}.  By Proposition
\ref{harmonic-osc} (2) this value is
\begin{equation}\label{zeta0}
\zeta_{\Lapl_{\cE', k}}(0) \, = \, \zeta_{\Lapl_{\cE'}}(0) 
\, = \, -\frac{|\deg(\cE')|}{2} .
\end{equation}
On the other hand, the corresponding Ray-Singer determinant varies and gives
rise to the functional  
\begin{align} \label{hol-tor}
\At^{\rm sa} \ni h^*=h \mapsto RS_{\cE', k} := \log \Det(\Lapl_{\cE', k}) :\, =\, 
 - \zeta^\prime_{\Lapl_{\cE', k}} (0) .
\end{align}
To obtain the variation formula for this functional, one proceeds as in
\cite[\S 4.1]{ConMos2011}. First, by \cite[Eq. (4.1)]{ConMos2011} 
\begin{align} \label{ds-zeta'}
  - \frac{d}{ds} \zeta^\prime_{\Lapl_{\cE', k^s}} (0)\, = \,
  \zeta_{\Lapl_{\cE', k^s}} (h, 0) .
\end{align}
Applying Theorem \ref{main-curv} to the right hand side one obtains
\pagebreak[3]
\begin{align}  \label{var-z'}
 - \frac{d}{ds}& \zeta^\prime_{\Lapl_{\cE', k^s}} (0) = \\ \notag
 &\frac{|\rk(\cE')|}{4\pi \Im\tau}  \vp_0 \left( h\Bl s K_+(s\nabla_h)(\Lapl(h)) 
   + s^2 H_+^\Re\bl s\nabla_h^1, s\nabla_h^2) (\square^\Re (h)\br \Br \right)\\  \notag
   & -\,\frac 12 |\rk(\cE')| \mu (\cE')\, \vp_0 (h) .
\end{align} 
Notice that
\begin{align} \label{rk--deg}
|\rk(\cE')| \mu (\cE')  =  \sgn( \rk (\cE')) \deg (\cE') =  |\deg (\cE')| ,
\end{align} 
since 
\begin{align*}
\sgn(\rk (\cE')) \sgn(\deg (\cE')) = \sgn \mu (\cE') = 1 .
\end{align*}
By integrating \Eqref{var-z'} and using the observation \eqref{rk--deg}
one obtains the variation formula
\begin{align} \label{int-var-zeta'}  
 &\log\Det(\Lapl_{\cE', k})\, =  \log \Det(\Lapl_{\cE'})\, 
-\,\frac 12 |\deg (\cE')| \, \vp_0 (h)  \\ \notag
 &\,+ \frac{|\rk(\cE')|}{4\pi \Im\tau} \int_0^1 
      \vp_0 \Bigl[ h\Bl s K_+(s\nabla_h)(\Lapl h)
        + s^2 H_+^\Re(s\nabla_h^1, s\nabla_h^2)\bl\square^\Re (h)\br\Br \Big] ds .
\end{align}        
The integrand above is carefully evaluated in \cite[Lemmas 4.2 - 4.4]{ConMos2011}.
As a result, the above identity becomes 
\begin{equation} \label{zeta'-integrand}
\begin{split}
   \log&\Det(\Lapl_{\cE', k})\, = \,\log \Det(\Lapl_{\cE'}) \, 
   - \, \frac 12 |\deg (\cE')| \, \vp_0 (h) \\
&- \frac{ |\rk(\cE')|}{16 \pi \Im\tau} \biggl(\frac{1}{3} \vp_0(h\Lapl h) +  \vp_0
 \Bl K_2(\nabla_h^1 )\bl \square^\Re(h)\br \Br\biggr),
\end{split}
\end{equation}
with  the function $K_2$ given by  
\begin{equation}\label{K2}
    K_2(2s)=\frac{1}{3}- \biggl(\frac{{\rm coth}(s)}{s} - \frac{1}{s^2}\biggr) , \quad s \in \R .
\end{equation}

\begin{theorem}[Determinant formula]
 \label{main-RS}
The exact formula for the Ray-Singer determinant corresponding to the 
change of Hermitian metric on the Heisenberg right $\Atp$--module
$\cE'$ by the factor $k = e^{\frac 12 h} \in \At^{\rm sa}$ is    
 \begin{align} \label{zeta'-exact}
\begin{split}
  \log\Det(\Lapl_{\cE', k}) & =\, \frac12 |\deg (\cE')|
   \log \big(2|\mu(\cE')| \Im(\bar\tau)\big)
   \,  - \, \frac 12 |\deg (\cE')| \, \vp_0 (h)\\
&\,- \frac{ |\rk(\cE')|}{16 \pi \Im\tau}\biggl(\frac{1}{3} \vp_0(h\Lapl h) +  \vp_0
 \Bl K_2(\nabla_h^1 )\bl \square^\Re(h)\br\Br\biggr).
\end{split}
\end{align}
\end{theorem}

\begin{proof} The stated formula follows from \Eqref{zeta'-integrand}, using
\Eqref{Eq:zetaprime0} to 
express the value of the log-determinant for the constant curvature metric.
\end{proof}

\subsubsection{Extremal of Ray-Singer determinant functional}

We take the point of view the Ray-Singer determinant is a functional
on the positive cone of Hermitian metrics on the $\Atp$--module $\cE'$.
Actually, we shall regard it as a functional on the space $\At^{\rm sa}$,
by composing with the map $ \At^{\rm sa} \ni h \mapsto k=e^{\frac{h}{2}}$.

Under the rescaling $h \mapsto h+\epsilon$ one has
\begin{align*}
\begin{split}
\zeta^\prime_{\Lapl_{\cE', ke^{\epsilon/2}}} (0) &= \,  \frac{d}{dz}\mid_{z=0}
\left(e^{-\epsilon z}  \zeta_{\Lapl_{\cE, k'}} (z)\right) 
= \,  \zeta^\prime_{\Lapl_{\cE', k}} (0)\, - \, \epsilon  \zeta_{\Lapl_{\cE', k}}(0)   \\
&=\, \zeta^\prime_{\Lapl_{\cE', k}} (0)\, + \,  \epsilon  \frac{|\deg(\cE')|}{2} \, .
\end{split}
\end{align*}
 It follows that the modified functional 
(analogous to that in~\cite{OPS1988})
 \begin{align} \label{OPS}
F_{\cE'} (h) \, =\, - \log \Det(\Lapl_{\cE', k}) 
\, - \, \frac12 |\deg(\cE')| \vp_0 (h) 
\end{align}
is scale invariant. 
The identity \eqref{zeta'-integrand} 
yields 
\begin{align*}  
\begin{split}
F_{\cE'} (h) &=\, -\frac12 |\deg(\cE')| \vp_0 (h) \,- \,\log \Det(\Lapl_{\cE'}) \, 
   +\,\frac 12 |\deg (\cE')| \, \vp_0 (h)   \\
&\,+ \frac{ |\rk(\cE')|}{16 \pi \Im\tau}\left(\frac{1}{3} \vp_0(h\Lapl h) +  \vp_0
 \left(K_2(\nabla_h^1 )(\square^\Re(h))\right)\right) ,
\end{split}
\end{align*}
which after cancelation becomes
\begin{align}  \label{OPS-ex}
\begin{split}
F_{\cE'} (h) =&\, - \,\log \Det(\Lapl_{\cE'}) \\
&\, + \frac{ |\rk(\cE')|}{16 \pi \Im\tau}\left(\frac{1}{3} \vp_0(h\Lapl h) +  \vp_0
 \left(K_2(\nabla_h^1 )(\square^\Re(h))\right)\right) .
\end{split}
\end{align}

\begin{theorem}[Extremal value] \label{main-YM} 
The scale invariant Ray-Singer determinant  
$k^2 \mapsto F_{\cE'} \big(\log (k^2)\big)$, viewed as a
functional on the (positive cone of) metrics on the Heisenberg left
$\At$--module $\cE'$, attains its minimum only at the metric whose
corresponding unique metric connection compatible with the holomorphic
structure has constant curvature.
\end{theorem}

\begin{proof} This follows from \Eqref{OPS-ex} and the crucial positivity
result established in the proof of \cite[Theorem 4.6]{ConMos2011}
\begin{align*}
\frac{1}{3} \vp_0(h\Lapl h) +  \vp_0
 \left(K_2(\nabla_h^1 )(\square^\Re(h))\right) \geq 0 ,
\end{align*}
for all $h^* = h \in \At$, with the equality holding only for $h=0$.
\end{proof}

Recalling that  by~Prop.~\ref{H-conn} the standard holomorphic structure and the 
Hermitian metric uniquely determines the metric connection,
the above result places
the (spectral) Ray-Singer determinant functional in a role similar to that of
the (local) Yang-Mills functional, albeit restricted to connections compatible
with the holomorphic structure.

\subsubsection{Gradient of Ray-Singer determinant}
\label{SSSGradDet}
  
In this section we recover the intrinsic curvature of the conformal metric on the base
from the gradient of  the Ray-Singer determinant functional associated to a Heisenberg bimodule.

We define the gradient of the functional $F_{\cE'}$ by the equation 
\begin{multline} \label{gradF}
\langle \grad_h F_{\cE'} , a \rangle_ {\cE'} \equiv 
 \frac{|\rk(\cE')|}{4\pi \Im\tau}\,\vp_0(a\cdot  \grad_h F_{\cE'}) \\:=
\frac{d}{d\epsilon}\big|_{\epsilon=0}F_{\cE'}(h + \epsilon a) , \quad
 a = a^\ast \in \At^\infty .
\end{multline}
Its explicit expression can be computed as in \cite[\S 4.2]{ConMos2011}.
Indeed, in the absence of zero-modes, the calculation in the proof of 
\cite[Thm.~4.8]{ConMos2011} gives for the gradient of the Ray-Singer functional
\begin{align} \label{RSF}
RS_{\cE'} (h) \, :=\, -  \zeta^\prime_{\Lapl_{\cE', k}} (0) 
\end{align}
the following formula
\begin{align} \label{4.8-analog}
\frac{d}{d\epsilon} \bigg|_{\epsilon=0}  RS_{\cE'}(h + \epsilon a)
\,= \, \frac{1}{2}\,{\rm a}_2 \left(\int_{-1}^1 e^{\frac{u h}{2}} \, 
 a \, e^{-\frac{uh}{2}} \, du , \, \Lapl_{\cE', k}\right) ,  \qquad a \in \At .
\end{align}
From the definition \eqref{OPS} it follows that
\begin{align*} 
\frac{d}{d\epsilon}\big|_{\epsilon=0}F(h + \epsilon a)
 \,= \,- \frac{1}{2}\,{\rm a}_2 \left(\int_{-1}^1 e^{\frac{u h}{2}} \, 
 a \, e^{-\frac{uh}{2}} \, du , \, \Lapl_{\cE', k}\right) 
 \, - \, \frac12 |\deg (\cE')| \vp_0(a) .
\end{align*}
Using \Eqref{rk--deg} and the above definition of the gradient, this yields
\pagebreak[3]
\begin{multline} \label{preGr}
 \frac{|\rk(\cE')|}{4\pi \Im\tau}\,\vp_0(a \cdot \grad_h F_{\cE'}) \\
= - \frac{1}{2}\,{\rm a}_2 \left(\int_{-1}^1 e^{\frac{u h}{2}} 
 a e^{-\frac{uh}{2}} \, du , \, \Lapl_{\cE', k}\right) 
 -  \frac12 \mu(\cE') |\rk(\cE')| \vp_0 (a)  .
\end{multline}
By Theorem \ref{main-curv}, \Eqref{main-a2+}, one has
\begin{align*} 
\begin{split}
 {\rm a}_2 \left(\int_{-1}^1 e^{\frac{u h}{2}} 
 a e^{-\frac{uh}{2}} du , \, \Lapl_{\cE', k}\right)  
&=    \frac{|\rk(\cE')|}{4\pi \Im\tau}\, \vp_0 \left(a \cdot \int_{-1}^1 e^{\frac{u h}{2}} 
 \cK_{\cE', h}  e^{-\frac{uh}{2}} du\right) \\
&=  \,  \frac{|\rk(\cE')|}{4\pi \Im\tau}\, \vp_0\left(a \cdot \tilde\cK_{\cE', k} \right) ,
 \end{split}
\end{align*}
where we have denoted
\begin{align} \label{tilde-curv}
\tilde\cK_{\cE', k} := 
 \int_{-1}^1 e^{\frac{u h}{2}}  \cK^+_{\cE', k}  e^{-\frac{uh}{2}} du \,
 \equiv \, 2 \frac{\sinh (\nabla/2)}{\nabla/2} \left( \cK^+_{\cE', k}\right) .
 \end{align}

\begin{theorem}[Gradient formula] \label{thm-grad-F}
\begin{align} \label{grad-F-K}
 \grad_h F_{\cE'} \,= \,  -\frac12 \, \tilde\cK_{k} .
  \end{align}
 \end{theorem}
 
 \begin{proof} 
 With the above notation the identity \eqref{preGr} becomes
 \begin{align*} 
 \frac{|\rk(\cE')|}{4\pi \Im\tau}\, \vp_0 (a\,  \grad_h F_{\cE'}) 
 \,= \,-  \frac{|\rk(\cE')|}{8\pi \Im\tau}\,\vp_0 \left(a \cdot \tilde\cK_{\cE', k} \right) 
 \, - \,  \frac12 \mu(\cE') |\rk(\cE')| \vp_0 (a)  .
\end{align*}
 Hence
 \begin{align*}  
  \grad_h F_{\cE'} \,= \, -\frac12  \, \tilde\cK_{\cE', k}  \, - \,
2 \pi \, \Im\tau \, \mu (\cE') 1.
  \end{align*}
 By  \Eqref{main-curv-comp}, the right hand side is equal to 
 \[ 
 -\frac12 \left(\tilde\cK_{k}  - 2 \pi \, \Im\tau \, \mu (\cE') 2 \frac{\sinh (\nabla/2)}{\nabla/2} (1) \right)
 - 2 \pi \, \Im\tau \, \mu (\cE') 1 \, =\,
-\frac12  \, \tilde\cK_{k} \,  .\qedhere
  \]
  \end{proof}
 
Now Theorem 4.8 in \cite{ConMos2011} states
that for the case of trivial coefficients,
\begin{align*} 
 \grad_h F \,= \,  -\frac12 \, \tilde\cK_{k} .
  \end{align*}
 Thus, the above result can be rephrased as follows.
 
 \begin{cor}[Morita invariance]
\begin{align*}  
 \grad_h F_{\cE'} \,= \,    \grad_h F .
  \end{align*}  
 \end{cor}  
 
 If one adopts the point of view that the  {\em Gaussian curvature} is, by definition, 
  {\em the gradient of the Ray-Singer determinant functional},  
  the above identity establishes both its Morita invariance and
  its independence of the $\operatorname{Spin}^c$--structure.  
  Indeed, the {\em non-twisted} spectral triples
in Proposition \ref{bi-spec-trip} or Proposition \ref{TST}
 define noncommutative metric structures on $A_{\gt'}$ in the sense of
 Connes' foundational axioms~\cite{Con1996}. These are obtained by coupling
 the standard Dirac spectral triple with arbitrary Hermitian metrics on the 
 ``bundles''  $\cE'$, which are automatically as in \Eqref{k-met}, for some positive, 
 invertible $k \in \End_{\cA_{\gt'}}(\cE') = \At$.
 (In the case of 
commutative tori, by Connes' reconstruction theorem~\cite{Con2013}, they
actually exhaust all the $\operatorname{Spin}^c$--Dirac operators on such tori.)
What we have proved is that the Gaussian curvature of such a metric on the 
noncommutative torus $A_{\gt'}$ 
coincides with the intrinsic metric on $A_\gt$ corresponding to the same 
$k \in \At$, and in particular is independent on the  
``$\operatorname{Spin}^c$-structure'' $\cE'$.

\section{Pseudodifferential multipliers and symbol calculus}
\label{SPsiDOMul}

\subsection{Twisted crossed product}
To establish an appropriate pseudodifferential calculus on Heisenberg
modules we will need an extension to twisted crossed products of
Connes' \cite{Con:CAG} and Baaj's \cite{Baa:CPDII,Baa:CPDI}
pseudodifferential calculus for $C^*$--dynamical systems.
We therefore consider a $C^*$--dynamical system $(\sA,\R^n,\ga)$ with
unital $\sA$. That is $\ga$ is a strongly continuous action of the additive
group $\R^n$ as automorphisms on the $C^*$--algebra $\sA$.
Furthermore, let $B=(b_{kl})_{k,l=1}^n$ be a
skew-symmetric real $n\times n$--matrix.  Put 
\begin{equation}\label{Eq1206041}  
      e(x,y):= e^{i\gs(x,y)}= e^{i \inn{Bx,y}}, \quad \gs(x,y):= \inn{Bx,y}.
\end{equation}
The skew-symmetry of $B$ implies $e(x,x)=1$ which will be used silently
many times.
Our main examples are the $C^*$--dynamical systems associated to projective
representations of $\R^2$ on the modules $\cE(g,\gt)$ as described in Section 
\ref{SSSBasicHeisenberg} and in the Appendix \ref{AppA}. The bilinear
form $\gs$ in \Eqref{Eq1206041} corresponds to $\mu(g,\gt) \gs$ in 
\Eqref{EqProRepSymForm}, cf.~also \Eqref{Eq:1209062}, with 
\begin{equation}\label{eq:1209191}  
 B=\pi\, \mu(g,\gt)  \twodimJ, \quad
 b_{21}=-b_{12}=\pi\, \mu(g,\gt).
\end{equation}

For $a\in\Ainf$ and a multiindex $\gamma\in\Z_+^n$ we put as usual
\begin{equation}\label{EqDeltaOnA}
  \gd^\gamma  a := i^{|\gamma|} \pl_x^\gamma\bigl|_{x=0} \ga_x(a)=
   i^{-|\gamma|} \pl_x^\gamma\bigl|_{x=0} \ga_{-x}(a).
\end{equation}
$\gd^\gamma$ plays the role of the partial derivative
$i^{-|\gamma|}\frac{\pl^\gamma}{\pl x^\gamma}$. $\Ainf$ together with
the seminorms $\|a\|_\gamma:= \|\gd^\gamma a\|$ is a Fr{\'e}chet
space. The Schwartz space $\SRnA$ is then defined as usual or as the
projective tensor product of the (scalar) Schwartz space $\sS(\R^n)$
and $\Ainf$.

For the following see also Appendix \ref{SSApp2}. 
$\SRnA$ is a pre-$C^*$--module with inner product
\begin{equation}\label{Eq1206042}  
      \inn{f,g}=\int_{\R^n} f(x)^* g(x)dx.    
\end{equation}
Put 
\begin{align} 
  (a \, f)(x) &= \ga_{-x} (a) f(x)\,, \qquad a \in \Ainf ; \\
  (\uU_y f)(x) &=  e(x,-y) f(x-y).
\end{align}

$\uU_y, y\in\R^n$, is a projective family of unitaries which implements the
group of automorphisms $\ga_y, y\in\R^n$:
\begin{align}
     \uU_x^* &=\uU_{-x}, \quad \uU_x\, \uU_y= e(x,y) \uU_{x+y}, \quad x,y\in\R^n,\\
     \uU_x a \uU_{-x}&= \ga_x(a), \quad a\in\sA^\infty.
\end{align}
By associating to $f\in\SRnA$ the multiplier $M_f=\int_{\R^n} f(x) \uU_x dx$
the space $\SRnA$ becomes a $*$--algebra.
Explicitly, $M_f\circ M_g = M_{f * g}$ and $M_f^*=M_{f^*}$, where
\begin{equation}\label{EqTwistedConvolution}  
   \begin{split}
      f^*(x)    &= \ga_x\bigl( f(-x)^*\bigr), \\
      (f*g)(x)  &= \int_{\R^n} f(y) \ga_y\bigl( g(x-y)\bigr) e(y,x) dy.
   \end{split}   
\end{equation}
Note that the formula for $f^*$ is the same as in the untwisted case.
If we want to emphasize the dependence of the product
on the twisting $\gs$ we write $*_{\gs}$.

\subsubsection{Dual Trace on $\SRnA$}\label{SSSDualTrace} 
If $\psi$ is an $\ga$--invariant (finite) trace on $\sA$ then the
\emph{dual trace} $\psihat$ on $\SRnA$ is given by
\begin{equation}\label{EqDualTrace}  
 \psihat(f)= \psi\bl f(0)\br=\int_{\R^n} \psi\bl\hat{f}(\xi)\br
 \dxi, \quad \dxi= (2\pi)^{-n} d\xi.        
\end{equation}
Note that $\dxi$ is the Plancherel measure of the dual group $(\R^n)^\wedge$
w.r.t. the duality pairing $(x,\xi)\mapsto e^{i\inn{x,\xi}}$.

\subsection{Pseudodifferential multipliers}
\label{SSPsiDO}
The symbol spaces (of H\"ormander type $(1,0)$) $\SmA$
are defined as in \cite{Con:CAG,Baa:CPDII,Baa:CPDI}.
We will also consider classical ($1$--step polyhomogeneous) symbols
$f\in\CSmA$ which have an asymptotic expansion
\[
  f\sim \sum_{j=0}^\infty f_{m-j}
\]
with
$f_{m-j}(\gl\xi)=\gl^{m-j}\cdot f_{m-j}(\xi), |\xi|\ge 1, \gl\ge 1$. 
Additionally, for the resolvent expansion we will need to consider parameter
dependent symbols $\operatorname{S}^m(\R^n\times \Gamma,\sA^\infty)$ and
corresponding pseudodifferential multipliers. This calculus is a little more
subtle than the name suggests. The parameter dependent class is more
restricted than just having operator families depending on a parameter. It is
designed to analyze the resolvent expansion.  Here $\Gamma$ is an open conic
subset of the complex plane and the parameter $\gl\in\Gamma$ is (for the
resolvents of second order operators) treated as a variable of degree $2$.
This is in complete analogy to \cite[\S 9]{Shu:POS}.  For a survey, see
\cite[\S 3]{Les:PDO}.   

To motivate the definition of pseudodifferential multipliers we calculate for
Schwartz functions $f,u\in\SRnA$, with $f^\vee$ denoting the inverse 
Fourier transform of $f$, 
\begin{align}
  (M_{{f^\vee}}&u)(x)
         := \Bigl( \int_{\R^n} {f^\vee}(y) \uU_y u \,
                  dy\Bigr)(x)\label{EqPsiDO}\\
         &= \int_{\R^n} \ga_{-x}({f^\vee}(y)) u(x-y) e(x,-y) \, dy
         = \int_{\R^n} \ga_{-x}({f^\vee}(x-y)) u(y) e(x,y) \, dy
                  \nonumber \\
         &= \int_{\R^n}\int_{\R^n} e^{i \inn{x-y,\xi -Bx}} \ga_{-x}( f(\xi))
  u(y) dy\,\dsl\xi\label{EqPsiDOa}\\
         &= \int_{\R^n} e^{i\inn{x,\xi}} \ga_{-x}(f(\xi)) \hat u (\xi-Bx) \,
         \dsl\xi    
         = \int_{\R^n} e^{i\inn{x,\xi}} \ga_{-x}(f(\xi+Bx)) \hat u(\xi) \,
         \dsl\xi \label{EqPsiDOb}\\
         & =: (P_f u)(x).\label{eq.rev2}
\end{align}
The last three integrals exist (\Eqref{EqPsiDOa} as iterated integral)
also if $f\in\SmA$ is a symbol of order $m$. We call the so defined
multiplier $P_f$ a \emph{(twisted) pseudodifferential multiplier with
symbol} $f$. 

As in the untwisted case the symbol can be recovered from
the multiplier.  Namely, for $t\in\R, b\in\sA^\infty$ we have
\begin{equation}\label{Eq1206045}  
  e^{-i\inn{t,x}} \bigl( P_f e^{i\inn{t,\cdot}} b \bigr)(x)
     =\ga_{-x}(f(t+Bx))\, b.  
\end{equation}
$f(t)$ is recovered by setting $x=0$ and $b=1_{\sA}$.

\begin{dfn}\label{DefPsiDO} 
By $\pdo^m_\gs(\R^n,\Ainf)$ we denote the space of
pseudodifferential multipliers of the form  \Eqref{EqPsiDOa}, \eqref{EqPsiDOb}. 
$f=:\sigma(P_f)$
 is called the symbol of $P_f$. The spaces of \emph{classical} symbols and
 multipliers are denoted by a decorator $\operatorname{C}$ in front, \ie
 $\CSmA$ resp. $\CL^m_\gs(\R^n,\Ainf)$.
\end{dfn} 

The space of twisted (classical) pseudodifferential multipliers is a
$*$--algebra:

\begin{theorem} For symbols $f\in\SmA, g\in \SmAvar{m'}$
the composition $P_f\circ P_g$ is a pseudodifferential multiplier with
symbol $h\in\SmAvar{m+m'}$. $h$ has the following asymptotic expansion
\begin{equation}\label{EqPsiDOproduct}  
  h(t)\sim \sum_{\gamma} \frac{i^{-|\gamma|}}{\gamma!} (\pl^\gamma f)(t)
  \pl_y^\gamma\big\vert_{y=0} \Bl \ga_{-y}\bl g(t+By)\br\Br.
\end{equation}

Furthermore, $P_f^*$ is a pseudodifferential multiplier with symbol
\begin{equation}\label{EqSymAdjoint}  
 \sigma(P_f^*)\sim \sum_\grg \frac{1}{\grg!} \pl_t^\grg \delta^\grg f(t)^*.
\end{equation}
\end{theorem}
\begin{remark} We compare these formulas with the corresponding
 formulas in the untwisted case, \cf\cite{Con:CAG}:

The formula for the symbol of the adjoint is unchanged.
For the product denote by $h_{\textup{untwisted}}$ the symbol of
the composition $P_f\circ P_g$ in the untwisted ($B=0$) calculus. Then
\begin{equation}\label{Eq1206047}  
  h_{\textup{untwisted}}(t)\sim  \sum_{\gamma} \frac{i^{-|\grg|}}{\gamma!}
      (\pl^\gamma f)(t) \pl_y^\gamma\big\vert_{y=0}
                   \ga_{-y}(g(t)).
\end{equation}
Expanding $\pl_y^\gamma\vert_{y=0}$ in \Eqref{EqPsiDOproduct} and
counting orders we see that up to terms of order $m+m'-3$ we have
\begin{equation}\label{EqPsiDOproductA}  
     h(t)=    h_{\textup{untwisted}}(t) + \frac 1i \sum_{j,k=1}^n (\pl_jf)(t)
     b_{kj} (\pl_k g)(t) \mod \SmAvar{m+m'-3}.
\end{equation}
\end{remark}

\begin{remark}\label{r.trace_psido}\MLrevision\mpar{Remark inserted}
The dual trace (Sec.~\ref{SSSDualTrace}) gives rise to
a \emph{natural}  trace on the algebra of pseudodifferential multipliers 
of order $<-n$. Namely, if $\psi$ is an $\ga$--invariant (finite) trace
 on $\sA$ then for $f\in \SmA$, $m<-n$, put (\cf \Eqref{EqPsiDO} --
\eqref{eq.rev2}) 
\begin{equation}\label{EqDualTrace_psido}  
 \Tr_\psi(P_f)= \psihat\bl f^\vee \br=\int_{\R^n} \psi\bl f(\xi) \br
 \dxi.
\end{equation}
Then $\Tr_\psi$ is a trace on $\cup_{m<-n} \pdo^m_\gs(\R^n,\Ainf)$.
It is important to understand that this trace in general differs from
the Hilbert space trace induced by the action on a Heisenberg module. This issue
is addressed in Sec.~\ref{SEffPsiDO}.
\end{remark}

\begin{proof} From \Eqref{Eq1206045} we infer
$\bigl(P_g e^{i\inn{t,\cdot}} 1_\sA\bigr)(x)= \ga_{-x}(g(t+Bx))e^{i\inn{t,x}}$.
Thus

\begin{equation}\label{Eq1206049}   
  \begin{split}
    h(t)&= \bigl(P_f P_ge^{i\inn{t,\cdot}} 1_\sA\bigr)(0)\\
      & = \bigl(\int_{\R^n} {f^\vee}(y) U_y \ga_{-\cdot}(g(t+
      B\cdot))e^{i\inn{t,\cdot}} dy \bigr)(0) \\
      &= \int_{\R^n}\int_{\R^n} e^{i\inn{y,\xi-t}} f(\xi) \ga_y(g(t-By))dy
      \,\dsl\xi\\
      &=\int_{\R^n}\int_{\R^n} e^{-i\inn{y,\xi}} f(\xi+t) \ga_{-y}(g(t+By))
      dy\,\dsl\xi.
  \end{split}
\end{equation}
The usual stationary phase argument implies that the symbol expansion is
obtained by Taylor expanding $f(t+\xi)\sim \sum_\gamma \frac{\pl^\gamma
f(t)}{\gamma!}\xi^\gamma$, replacing $\xi^\gamma e^{-i\inn{y,\xi}}$ 
by $i^{|\grg|} \pl_y^\grg e^{-i\inn{y,\xi}}$ and then integrating by parts.
Then by the Fourier inversion Theorem the summand of index $\gamma$
becomes
\pagebreak[2]
\begin{multline}\label{Eq12060410}  
 \frac{1}{\gamma!} (\pl^\gamma f)(t) 
 \int_{\R^n}\int_{\R^n} e^{-i\inn{y,\xi}}  
    i^{-|\grg|} \pl_y^\gamma\bigl(\ga_{-y}\bl g(t+By)\br\bigr) dy \,\dsl\xi\\
    = \frac{i^{-|\grg|}}{\gamma!} (\pl^\gamma f)(t)
  \pl_y^\gamma\big\vert_{y=0}\Bl \ga_{-y}\bl g(t)\br\Br,
\end{multline}
and the product formula is proved. 

For the adjoint note first that 
$P_f^*=M_{f^\vee}^*=M_{{f^\vee}^*}=P_{({f^\vee}^*)^\wedge}=:P_h$.
By \Eqref{EqTwistedConvolution} we find
\begin{equation}\label{eq:1209126}   
   {f^\vee}^*(x)=\ga_x\bigl({f^\vee}(-x)^*\bigr)=
       \int_{\R^n} e^{i \inn{x,\xi}} \ga_x\bigl(f(\xi)^*\bigr) \dxi.
\end{equation}
Thus 
\begin{equation}\label{eq:1209127}   
  \begin{split}
   h(t)&= \int_{\R^n}\int_{\R^n} e^{i\inn{x,\xi-t}} \ga_x\bigl(f(\xi)^*\bigr)
   \dxi \, dx\\
       &= \int_{\R^n}\int_{\R^n} e^{i\inn{x,\xi}} \ga_x\bigl(f(t+\xi)^*\bigr)
   \dxi \, dx.
  \end{split}
\end{equation}
As before we expand 
$f(t+\xi)\sim \sum_\gamma \frac{\pl^\gamma f(t)}{\gamma!}\xi^\gamma$, 
replace $\xi^\gamma e^{i\inn{x,\xi}}$ 
by $i^{-|\grg|}\pl_x^\gamma e^{i\inn{x,\xi}}$ and then integrate by parts.
The Fourier inversion Theorem then gives the result.
\end{proof}

\subsection{Differential multipliers} Among pseudodifferential multipliers
the class of differential multipliers is characterized by having symbols
which are polynomial in $\xi$. Thus
\begin{dfn}\label{DefDiffOp} 
$P_f\in\LmA$ is called a \emph{differential multiplier} of order $m$ if
$f\in \Ainf[\xi_1,\ldots,\xi_n]$ is a polynomial of degree at most
$m$, \ie
\begin{equation}\label{EqDiffOpSym}
  f(\xi)= \sum_{|\gamma|\le m} a_\gamma \xi^\gamma; \quad a_\gamma\in\Ainf.
\end{equation}
Here the sum runs over all multiindices $\gamma\in \Z_+^n$ with
$|\gamma|\le m$.  We denote the space of all differential multipliers by
$\Diff^m_\gs(\R^n,\Ainf)$.  Clearly, differential multipliers are
classical pseudodifferential multipliers.
\end{dfn}

Recall that in the ordinary pseudodifferential calculus the symbol of the basic
derivatives $i^{-|\gamma|} \pl_x^\gamma$ is given by $\xi^\gamma$.
Therefore, for a multiindex $\gamma$ we put
$\ud^\grg:=P_{\xi^\gamma}$.

\begin{prop}\label{PDFormula} For $u\in \SRnA$ we have
\[ 
\begin{split}
(\ud^\gamma u)(x)  & = (P_{\xi^\grg}u)(x)=
  i^{|\grg|} \pl_y^\grg\bigl|_{y=0} \uU_y u(x) \\
    & = i^{-|\grg|} \pl_y^\grg\bigl|_{y=0}\bl e(x,y) u(x+y) \br.
\end{split}
\]
\end{prop}
\begin{proof} We use the definition \Eqref{EqPsiDOb} and find
\begin{equation}\label{eq:1209124}   
  \begin{split}
      (\ud^\grg u)(x)
    & = \int_{\R^n} e^{i\inn{x,y}}  \bl \xi+Bx\br^\grg \hat u(\xi) \,
      \dsl \xi \\
      & = i^{-|\grg|} \pl_y^\grg\bigl|_{y=0} 
      \int_{\R^n} e^{i \inn{x+y,\xi+Bx}} \hat u(\xi)\, \dsl\xi\\
    & = i^{-|\grg|} \pl_y^\grg\bigl|_{y=0}\bl e(x,y) u(x+y)
      \br.\qedhere
\end{split}
\end{equation}
\end{proof}
\begin{remark}      
1. It is important to note that due to the twisting in
general $\ud^{\grg}\ud^{\grg'}\not= \ud^{\grg+\grg'}$, as can
be seen either directly  or by the just proved product formula.

2. As in the ordinary pseudodifferential calculus it is
in general not true that $P_f^*=P_{f^*}$. However,
$\sigma(P_f)^*=\sigma(P_f^*)\mod \SmAvar{m-1}.$

Furthermore, $\ud^\gamma$ is formally self-adjoint and thus for
any \emph{differential} multiplier we have indeed $P_f^*=P_{f^*}$.
\end{remark}

\subsubsection{Differential multipliers of order $1$ and $2$}
 \label{SSSDiffOp12}
Let $e_j, j=1,\ldots, n$ be the canonical basis vectors of $\R^n$.
We abbreviate $\ud_j:=\ud^{e_j}$. Then by
Prop. \ref{PDFormula}
\begin{equation}\label{eq:1209128}  
 \ud_j u(x)= i\ii\pl_{y_j}\bigl|_{y=0} e^{i\inn{Bx,y}} u(x+y)
           = \bigl(\frac 1i \pl_{x_j} + b_{jl} x_l\bigr) u(x),
\end{equation}
were summing over repeated indices is understood. Thus
\begin{align}\label{eq:1209129}  
 \ud_j \ud_k &= -\pl_{x_j}\pl_{x_k} - i b_{js} x_s \pl_{x_k} - i  b_{ks} x_s
  \pl_{x_j} - i b_{kj} + b_{js} b_{kr} x_s x_r,\\
\intertext{in particular}
[\ud_j,\ud_k] &= 2 i \, b_{jk}.\label{EqdbarComm}
\end{align}
The product formulas \Eqref{EqPsiDOproduct}, \eqref{EqPsiDOproductA} now give
\begin{equation}\label{EqSymProd}   
  \begin{split}
     \sigma\bigl(\ud_j \cdot \ud_k \bigr)  
       &= \sigma(\ud_j)\cdot \sigma
          (\ud_k)+ \frac 1i \sum_{r,s} \pl_{\xi_r}(\xi_j) \, b_{sr}
                \pl_{\xi_s}(\xi_k) \\
       & = \xi_j \cdot \xi_k + i b_{jk} = \sigma\bigl(\ud^{e_j+e_k}\bigr)+ i
       b_{jk}.
  \end{split}
\end{equation}
Thus we get the following explicit formula for $\ud^{e_j+e_k}$
\begin{equation}
 \ud^{e_j+e_k} = \ud_j\cdot \ud_k - i b_{jk},
\end{equation}
as well as again the commutator formula $ [\ud_j,\ud_k] = 2 i b_{jk}$.

\subsubsection{Differential multipliers in dimension $n=2$} 
\label{SSSDiffOp2}
We consider the case $n=2$ and the complex Wirtinger derivatives
$\pl$ and $\ovl{\pl}$ in this context. 
Fix $\tau\in\C$ with $\Im\tau>0$. Put
\begin{equation}\label{EqudbarA}  
    \udtau := \ud_1 + \taubar \ud_2, \quad \udtau^*=\ud_1 + \tau \ud_2.
\end{equation}
These are constant coefficient differential multipliers and therefore
with \Eqref{EqdbarComm}
\begin{equation}\label{EqudbarB}   
    [\udtau,\udtau^*]= - 4\cdot \Im\tau\cdot         b_{12}=:c_\tau.
\end{equation}
We define the Laplacian by
\begin{equation}\label{EqudbarC}  
     \uDelta_\tau:= \frac 12 (\udtau^*\udtau+\udtau\udtau^*)
        = \ud_1^2+ |\tau|^2 \ud_2^2 + \Re\tau ( \ud_1 \ud_2+\ud_2\ud_1).
\end{equation}
Note
\begin{align}
  k^2 \udtau\udtau^* &=k^2  \uDelta_\tau + \frac 12 k^2 c_\tau, \\
  \udtau k^2\udtau^* &= k^2 \uDelta_\tau +(\udtau k^2) \udtau^* + \frac 12 k^2
  c_\tau,\\
  \udtau^* k^2 \udtau &=k^2  \uDelta_\tau +(\udtau^* k^2 )\udtau - \frac 12 k^2 c_\tau.
\end{align}
This symmetric definition of $\uDelta_\tau$ has the advantage that its
symbol equals $|\eta|^2:=|\xi_1+\tau \xi_2|^2$.

%

\section{The resolvent expansion for elliptic differential multipliers}
\label{SResExp}

\subsection{Set-up and formulation of the result}
\label{SSSetUp}

We continue to work in the framework of Section \plref{SPsiDOMul}. 
The algebra $\sA$ stands for, \eg either $\Atp, \At$, or $\Atp\ot \At$.
$\psi$ stands for the corresponding normalized trace
($\varphi_0$ resp. $\varphi_0\otimes\varphi_0$).\MLrevision\mpar{sentence
inserted} 
We consider the differential  multiplier
\begin{equation}\label{EqDiffMulP1}
  P= P_{\eps_1 , \eps_2} := k^2 \uDelta_\tau + \eps_1 (\dtau k^2) \udtau^* + \eps_2
  (\dtau^*k^2) \udtau + a_0,
\end{equation}
\MLrevision
\mpar{second partial was not bold face}
where $a_0\in\sA$ and $\eps_1, \eps_2$ are real parameters. 

We want to compute the first three terms in the expansion of the
resolvent $(P-\gl)\ii$ in the parameter dependent pseudodifferential
calculus. For simplicity we restrict ourselves here to the case $n=2$
and the special $P$ as above. We emphasize, however, that to a large
extent the computation could be done in greater generality; we intend
to come back to this in a future publication.

Recall from Subsection \ref{SSSDiffOp12} that the symbol of $P$ takes the form
\begin{equation}\label{EqResExpSymbol}  
    \sigma_P(\xi):= a_2(\xi) + a_1(\xi) +a_0,
\end{equation}
where $a_0\in\Ainf$ is the same as above and 
\begin{align*}
    a_2(\xi) &= k^2 |\xi_1+\taubar \xi_2|^2 =: k^2 |\eta|^2, \\
    a_1(\xi) &= \eps_1 (\dtau k^2) \etabar + \eps_2 (\dtau^* k^2) \eta,
             &\eta &:= \xi_1+\taubar \xi_2, \\
             &=: \varrho_1 \etabar + \varrho_2 \eta, &\vrho_1 &:= \eps_1
  \dtau k^2, \varrho_2 := \eps_2 \dtau^* k^2.
\end{align*}

Below we calculate the resolvent expansion in terms of
functions of $\xi$. The resolvent trace density, at least with
respect to the dual trace \plref{SSSDualTrace}, is obtained by
integration over $\R^2$ with respect to $\xi$. Let $f(\eta_1,\eta_2)\in \CSmA$
be a symbol of order $<-2$ which is given as a function of
$\eta=\xi_1+\taubar \xi_2=:\eta_1+i\eta_2$. Changing variables we have
\begin{equation}\label{EqResExpIntegral1}
  \int_{\R^2} f(\eta_1,\eta_2) \dxi 
     = \frac{1}{(2\pi)^2 |\Im \tau|} \int_{\R^2} f(\eta_1,\eta_2)
     d\eta,\quad d\eta:=d\eta_1 d\eta_2.
\end{equation}
Furthermore, we note that if $f(\eta_1,\eta_2)=g(|\eta|^2)$ 
depends only on $|\eta|^2$ then 
\begin{equation}\label{EqResExpIntegral2}  
  \int_{\R^2} \eta_1^{\ga_1} \eta_2^{\ga_2} \, g(|\eta|^2) d\eta = 0, 
\end{equation}
whenever $\ga_1$ or $\ga_2$ is odd. Furthermore, regardless of the
parity of $\ga_j$ the integral on the left of \Eqref{EqResExpIntegral2}
is invariant under permutations of $(\ga_1,\ga_2)$. In particular
\MLrevision
\mpar{Referee item 6}
\begin{equation}\label{EqResExpIntegral3}  
  \begin{split}
    \int_{\R^2} \eta_k^2 g(|\eta|^2) d\eta &= \frac 12 \int_{\R^2} |\eta|^2
 g(|\eta|^2)  d\eta,\quad k=1,2 \, ;\\
    \int_{\R^2} \eta^2 f(|\eta|) d\eta &=  \int_{\R^2} \bl \eta_1^2 -
\eta_2^2 + 2 i \eta_1\eta_2 \br f(|\eta|) d\eta = 0.
 \end{split}
\end{equation}

For functions $f,g$ of $\xi$ we therefore introduce
the notation $f\dot = g$ if $f$ and $g$ have the same
$\xi$--integral, and hence the same $\eta$--integral.

\pagebreak[3]
Now we are able to state the main result about the resolvent of the
multiplier $P$.
\pagebreak[2]

\begin{theorem}\label{TResExp} Let $P$ be the differential multiplier
  \Eqref{EqDiffMulP1} in the pseudodifferential multiplier calculus
  and denote by $\eta:=\xi_1+\taubar \xi_2$ the symbol of $\udtau$
  and let $b(\xi):= (k^2 |\eta|^2-\gl)\ii$. 

Then the resolvent $(P-\gl)\ii$ of $P$ is a parameter dependent 
pseudodifferential multiplier with polyhomogeneous symbol
$b_{-2}+b_{-3}+b_{-4}+\ldots$. Up to a function of total
$\xi$--integral $0$ we have the following closed formulas for
the first three terms in the symbol expansion of $(P-\gl)\ii$:
\begin{align}
b_{-2}  = b & = (k^2 |\eta|^2-\gl)\ii,\label{EqResExp1}\\
  b_{-3}    & = - b k^2 \bl \eta \dtau^*+ \etabar \dtau\br b - b a_1 b,
                 \label{EqResExp2} \\
  b_{-4}    & = \bl 2b k^2 |\eta|^2 - 1 - \eps_1-\eps_2\br b k^2
                    \Lapl_\tau b 
                    \label{EqResExp3} \\
    &\quad + \gl b k^2 \bl (\dtau^* b) ( \dtau b) 
                    + (\dtau b)(\dtau^* b)\br 
                    \label{EqResExp4} \\
    &\quad + \eps_1\cdot \gl b (\dtau k^2) b \dtau^* b         
              + \eps_2 \cdot \gl b (\dtau^* k^2) b \dtau b   
                    \label{EqResExp5} \\
           &\quad + \eps_1\eps_2\cdot 
               |\eta|^2 b\cdot \bl (\dtau k^2) b (\dtau^* k^2)
               +(\dtau^* k^2) b \dtau k^2 \br \cdot b - b a_0 b.
                    \label{EqResExp6} 
\end{align}
\end{theorem}

\subsubsection{The second heat coefficient in terms of $k^2$} 
To formulate the main result about the second heat coefficient
we introduce the following abbreviations: 
\begin{equation}\begin{split}
  \tsRk 
       & = \frac 12  (k^{-2} \dtau k^2 \cdot k^{-2}\dtau^* k^2
         + k^{-2} \dtau^* k^2 \cdot k^{-2} \dtau k^2)\\
  \tsIk 
       & =  \frac 12 (k^{-2} \dtau k^2 \cdot k^{-2} \dtau^* k^2 
                -  k^{-2} \dtau^* k^2 \cdot k^{-2} \dtau k^2).
\end{split}
\end{equation}
Furthermore, denote by $\modu = k^{-2}\cdot k^2 = e^{\nabla_h}$ the
modular operator.

\MLrevision\mpar{statement of Thm expanded}
\begin{theorem}\label{TbInt} Let $P$ 
  be the differential multiplier \Eqref{EqDiffMulP1}. Then there is
an asymptotic expansion
\[
  \Tr_\psi \bl a e^{-t P} \br \sim_{t \searrow 0} 
       \sum_{j=0}^\infty a_{2j}(a,P)\, t^{j-1}.
\]
Here, $\Tr_\psi$ is the natural trace on the algebra
of pseudodifferential multipliers (\cf Remark \ref{r.trace_psido})
induced by the (unique) normalized trace $\psi$ on $\sA$,
which is invariant under the natural $\R^2$--action.

We have 
\begin{equation}\label{eq.rev4}
     a_0(a,P) = \frac{1}{4\pi |\Im\tau| }\psi ( a k^{-2} ).
\end{equation}

Furthermore, there exist functions
  $F(u), G^\Re(u,v),
  G^\Im(u,v)$ such that the second heat coefficient
  of $P$ 
is given by
\begin{multline}\label{EqbInt}
  a_2(a,P) = \frac{1}{4\pi |\Im\tau|} \psi\Biggl[ a 
         \Bl F(\modu)(k^{-2} \Lapl_\tau k^2)  - \cL_0(\modu)(k^{-2} a_0) \\
   + G^\Re(\modua,\modub)\bl \tsRk \br
   + G^\Im(\modua,\modub)\bl \tsIk \br \Br
            \Biggr],
\end{multline}
The functions $F, G^\Re, G^\Im$ depend only on $P$ but not on
$\tau$. They are linear combinations of simple \footnote{That means at most
the third divided difference occurs. Recall that the divided differences
of a smooth function $f$ are recursively defined by  
\[
     [x_0] f := f(x_0),\quad
     [ x_0,\ldots,x_n] f := \frac {1}{x_0-x_n}\bl [x_0,\ldots, x_{n-1}] f 
         - [x_1,\ldots,x_n] f\br.
\]
For a review and further references see \cite[Appendix A]{Les2014}.
} divided differences of $\log$.
In particular they are analytic on the Riemann surface of $\log$.
$\cL_0(u)=\frac{\log u}{u-1}$ is the generating function of the 
Bernoulli numbers.
\end{theorem}
The functions will be given explicitly in \Eqref{EqFgamma},
\eqref{EqGRgamma}, \eqref{EqGIgamma}, \eqref{EqF}, \eqref{EqGR},
and \eqref{EqGI} below.

\subsection{Second heat coefficient in terms of $\log k^2$}
To express the second heat coeffient in terms of $\log k^2$ we 
write for $h\in\sA^\infty$
\begin{equation} \begin{split}
  \dtau h\cdot \dtau^* h   &=:  \square^\Re(h) + \square^\Im(h) \\
  \dtau^* h\cdot \dtau h  &=:  \square^\Re(h) - \square^\Im(h).
\end{split}
\end{equation}
Clearly,
\begin{equation}\begin{split}
  \square^\Re(h) & = \frac 12  (\dtau h \cdot \dtau^* h
                              + \dtau^* h \cdot \dtau h)\\
  \square^\Im(h) & =  \frac 12 (\dtau h \cdot \dtau^* h
                              - \dtau^* h \cdot \dtau  h).
\end{split}
\end{equation}

\begin{cor}\label{CA2} There exist entire functions $K(s), H^\Re(s,t),
  H^\Im(s,t)$, such that with $h:=\log k^2$ 
  the second heat coefficient of $P$ takes the form
\begin{equation}\label{EqbIntCor}
\begin{split}  
  a_2(a,P) = \frac{1}{4\pi |\Im\tau|} \psi\Biggl[ &a 
         \Bl K(\nabla)(\Lapl_\tau h)  - k^{-2} a_0 \\
  &+ H^\Re(\nabla^{(1)},\nabla^{(2)})\bl \square^\Re(h) \br
   + H^\Im(\nabla^{(1)},\nabla^{(2)})\bl \square^\Im(h) \br \Br
            \Biggr],
\end{split}
\end{equation}
\end{cor}  
The functions will be given explicitly in the proof in
\Eqref{EqFK}, \eqref{EqGHR}, \eqref{EqGHI}.

In the sequel we will always use the following synonyms for
variables: $s=\log(u), t=\log(v)$.
\begin{proof} Recall from \cite[Lemma 5.1]{FarKha2013}, see also 
\cite[Proof of Lemma 3.3]{ConTre2011}, \cite[Sec.~6.1]{ConMos2011},
resp. in the notation of divided differences \cite[Ex.~3.13]{Les2014}
\pagebreak[3]
\begin{align}  
  k^{-2}\ud_\tau^{(*)} k^2  & = \frac{e^\nabla-1}{\nabla}
  \ud_\tau^{(*)} h \\
  k^{-2} \uDelta_\tau k^2   & = \frac{e^\nabla-1}{\nabla} \uDelta_\tau h
  \nonumber\\
  &\quad + 2 [0,\nabla^{(1)},\nabla^{(1)}+\nabla^{(2)}]\exp (\square^\Re
  (h)),
\end{align}
where 
\begin{equation}
\begin{split}
    [0,s,s+t]\exp &= \frac{(e^{s+t}-1)s - (e^s-1)(s+t)}{st(s+t)} \\
              & = \frac{(uv-1)\lgu - (u-1)\lguv}{\lgu \lgv \lguv}.
\end{split}
\end{equation}
I.~e.,~$[\cdot,\cdot,\cdot]\exp$ is the second divided difference of the
exponential function.

Inserting these formulas into \Eqref{EqbInt} immediately yields the claim.
Moreover, we get the following explicit formulas for the functions
$K,H^\Re, H^\Im$ in terms of $F, G^\Re, G^\Im$:
\begin{align}
  K (s)        & = F(u) \frac{u-1}{\lgu} = F(e^s) \frac{e^s-1}{s},\label{EqFK} \\
  H^\Re(s,t)   & = 
     2 F(uv) \frac{(uv-1)\lgu - (u-1)\lguv}{\lgu \lgv \lguv}, \nonumber \\
&\quad + G^\Re(u,v) \frac{u-1}{\lgu} \frac{v-1}{\lgv} \nonumber\\
   & = 2 F(e^{s+t}) \frac{(e^{s+t}-1) s - (e^s-1)(s+t)}{s t (s+t)}, \nonumber \\
&\quad + G^\Re(e^s,e^t) \frac{e^s-1}{s} \frac{e^t-1}{t} \label{EqGHR}\\
    H^\Im (s,t) & = G^\Im(u,v)  \frac{u-1}{\lgu} \frac{v-1}{\lgv}
     \nonumber\\
    & = G^\Im(e^s,e^t)  \frac{e^s-1}{s} \frac{e^t-1}{t}.
          \label{EqGHI}
\end{align}
Recall, that $s=\lgu, t=\lgv$.
\end{proof}

\subsection{Proof of Theorem \ref{TResExp}}\label{SSResRec} 
\newcommand{\plxi}[1]{\pl_{\xi_{#1}}}
Let $\gl$ be the resolvent parameter and consider the parameter
dependent symbol 
$\sigma_\gl(\xi):=\sigma_P(\xi)-\gl=:a_2'+a_1+a_0; a_2':=a_2-\gl$, 
which is elliptic in the parameter dependent calculus. 
To find the resolvent parametrix $B(\gl)$ we need to solve
\begin{equation}\label{EqResExpSymEqn}  
          1 = \sigma_\gl * \sigma_B  
\end{equation}
with respect to the product \Eqref{EqPsiDOproduct}; here we have to
treat $\gl$ as a variable of order $2$.

In the following, unless otherwise said, we understand summation
convention over repeated indices running from $1$ to $2$.
The product formula \Eqref{EqPsiDOproduct} shows that up to terms
of order $\le -1$ we have $\sigma_\gl * \sigma_B = a_2' \cdot b_{-2}$.
This already shows $b:=b_{-2}=(a_2')\ii$ \Eqref{EqResExp1}. 
Furthermore, note that up to terms of order $\le -3$ the defect
between the product formula \Eqref{EqPsiDOproduct} and the untwisted
product formula (\cf\Eqref{EqPsiDOproductA}) is given by 
\[
  \frac 1i \bl \plxi j a_2 \br  b_{jk} \bl \plxi k b \br,
\]  
and this vanishes since $b_{jk}$ is skew and $a_2$ and $b$ are
functions of $|\eta|^2$. Thus up to terms of order $\le -3$ we may
employ the usual untwisted product formula and hence we find
up to a term of order $\le -3$       
\begin{equation}\label{eq:1209174}  
   \sigma_\gl * \sigma_B =  \sigma_\gl \cdot \sigma_B + \pl_{\xi_j}
   (a_2+a_1)\cdot \gd_j \sigma_B 
   + \frac 12 (\pl_{\xi_k}\pl_{\xi_l} a_2) \gd_k \gd_l \sigma_B.
\end{equation}

We rewrite the individual summands on the right as follows:
\begin{align}
   \pl_{\xi_j} a_2 \gd_j  
          & = k^2 \gd_{\xi_j} |\eta|^2 \gd_j =
              k^2 \bl \etabar \dtau + \eta\dtau^* \br, \\
     \pl_{\xi_j} \bl a_1(\xi)\br  \gd_j 
          &= \plxi j \bl \vrho \etabar +\varrho_2 \eta\br \gd_j = \vrho
     \dtau^* +\varrho_2 \dtau,\\
\frac 12 (\pl_{\xi_k}\pl_{\xi_l} a_2) \gd_k \gd_l 
          & = 
       \frac{k^2}{2} \gd_{\xi_k} \gd_k \bl \etabar\dtau +\eta \dtau^* \br 
       = \frac{k^2}{2} \bl \dtau^*\dtau + \dtau\dtau^*\br 
       = k^2 \Lapl_\tau.
\end{align}       

Here, the operators $\dtau:=\frac 1i \delta_\tau= \frac 1i (\delta_1+\taubar\delta_2),
\dtau^*:=-\frac 1i \delta_\tau^*= \frac 1i (\delta_1+\tau\delta_2),
\Lapl_\tau:=\dtau^*\dtau=\dtau\dtau^*$ are the counterparts to the multipliers
\Eqref{EqudbarA} acting on $\sA$. As opposed to the multipliers where
the structure constant \Eqref{EqudbarB} may be nonzero the operators $\dtau$
and $\dtau^*$ commute.

Now we expand $\sigma_B\sim b_{-2}+ b_{-3}+b_{-4}+\ldots$ into 
homogeneous terms of order $-2, -3, -4,\ldots$. Ordering terms in
\Eqref{eq:1209174}  by homogeneity and writing $b$ for $b_{-2}$ we find
\begin{align}\label{EqResExpRec}   
  1\sim \sigma_\gl * \sigma_B   
      & =  a_2'\cdot  b \quad (\text{order } 0)\nonumber\\
      & \quad + a_2'\cdot b_{-3} + a_1 \cdot b + 
             k^2 \bl \etabar \dtau +\eta\dtau^*) b \quad(\text{order } -1)\\
      & \quad + a_2' \cdot b_{-4} + k^2 (\etabar \dtau+\eta\dtau^*) b_{-3} 
              +   k^2 \Lapl_\tau b \quad (\text{order } -2) \nonumber\\ 
      & \quad   + a_1 \cdot b_{-3} + (\varrho_1
               \dtau^*+\varrho_2 \dtau) b + a_0 \cdot b, \nonumber
\end{align}
and hence 
\begin{align}  
 b_{-2}     & = (k^2 |\eta|^2-\gl)\ii,\label{Eqb2}\\
   b_{-3}   & = - b k^2 (\etabar \dtau+\eta\dtau^*) b - b \cdot a_1 \cdot  b, \label{Eqb3}\\
   b_{-4}   & = -b k^2 (\etabar \dtau+\eta\dtau^*) b_{-3} - b\cdot a_1 \cdot b_{-3}  \label{Eqb4}\\
                &\quad  - b k^2 \Lapl_\tau b - b \cdot(\varrho_1
  \dtau^*+\varrho_2 \dtau) b
                    - b\cdot a_0 \cdot b.\nonumber
\end{align}
We rewrite the individual summands of $b_{-4}$ modulo functions of vanishing
total integral (\cf~the discussion before Theorem \ref{TResExp}). 
We make frequent use of the formulas
\begin{equation}
  k^2 |\eta|^2  = a_2'+\gl = b\ii +\gl, \quad b k^2 |\eta|^2= 1+ \gl \cdot b,
                                         \label{EqResExpAbbrev3}
\end{equation}

\subsubsection{First summand involving $b_{-3}$} We replace $b_{-3}$ by the 
rhs of \Eqref{Eqb3} and find for each term

\begin{equation}
  \begin{split}
    b k^2 &(\etabar\dtau+\eta \dtau^*) \bl b k^2
      (\etabar\dtau+\eta\dtau^*)b\br\\
      & \dot= b k^2 \dtau\bl b k^2 |\eta|^2 \dtau^* b\br 
           + b k^2 \dtau^*\bl b k^2 |\eta|^2 \dtau b\br\\
      & = b k^2 \dtau\bl (1+\gl\cdot b) \dtau^* b\br 
            + b k^2 \dtau^*\bl (1+\gl\cdot b) \dtau b\br\\
      & = 2 b k^2 |\eta|^2  b k^2 \Lapl_\tau b 
            + \gl b k^2 \bl (\dtau^* b)\dtau b + (\dtau b)\dtau^* b \br,
\end{split}
\end{equation}
resp.
\begin{equation}
  \begin{split}
  - b k^2 \bl \etabar \dtau +\eta \dtau^*\br \bl 
    &- b(\vrho \etabar + \varrho_2 \eta) b\br \\
    & \dot= k^2 |\eta|^2 b \bl \dtau^* (b \vrho b) + \dtau (b \varrho_2 b) \br \\
    & = b k^2 \bl \dtau^*\bl\eps_1 b |\eta|^2 (\dtau k^2) b\br \br  
       +b k^2 \bl \dtau  \bl\eps_2 b |\eta|^2 (\dtau^* k^2) b\br \br \\
    & = - (\eps_1+\eps_2) b k^2 \Lapl_\tau b.
     \end{split}     
\end{equation}

\subsubsection{Second summand 
  $- b a_1 b_{-3} = -b (\vrho \etabar+ \varrho_2\eta) b_{-3}$
involving $b_{-3}$} Similarly,
\begin{equation}
  \begin{split}
    - b (\varrho_1 \etabar + \varrho_2\eta) &\bl -b k^2
    (\etabar\dtau+\eta\dtau^*) b \br \\
    & \dot= b \varrho_1 |\eta|^2 b k^2 \dtau^* b
            + b \varrho_2 |\eta|^2 b k^2 \dtau b \\
     &= b\vrho (1+\gl\cdot b) \dtau^* b
        + b \varrho_2 (1+\gl\cdot b) \dtau b,
  \end{split}   
\end{equation}     
resp.
\begin{equation}
-b (\varrho_1 \etabar+\varrho_2 \eta) (-b (\varrho_1 \etabar+\varrho_2\eta) b) \dot = 
    |\eta|^2 b\cdot (\vrho b \varrho_2 +\varrho_2 b \vrho)\cdot b. 
\end{equation}    
Summing up we have
\[\begin{split}
    b_{-4} & = \bl 2 k^2 |\eta|^2 b - 1-\eps_1-\eps_2 \br b k^2 \Lapl_\tau b 
             + \gl b k^2 \bl (\dtau^*b) \dtau b + (\dtau b) \dtau^*
              b\br \\
           &\quad  
               - b \bl \varrho_1 \dtau^* b+ \varrho_2 \dtau b\br
               + b \varrho_1 (1+\gl b)\dtau^* b + b \varrho_2 (1+\gl b)\dtau b \\
           &\quad + |\eta|^2 b\cdot (\vrho b \varrho_2 +\varrho_2 b
              \vrho)\cdot b - b a_0 b\\
           & =  \bl 2 k^2 |\eta|^2 b - 1 - \eps_1-\eps_2 \br b k^2 \Lapl_\tau b 
               + \gl b k^2 \bl (\dtau^*b) \dtau b + (\dtau b) \dtau^* b\br \\
           &\quad 
                + \eps_1 \gl b (\dtau k^2)   b \dtau^* b 
                + \eps_2 \gl b (\dtau^* k^2) b \dtau b
                + |\eta|^2 b\cdot (\vrho b \varrho_2 +\varrho_2 b
              \vrho)\cdot b - b a_0 b,
\end{split}\]
and the Theorem is proved.\hfill\qed

\section{The integral of $b_{-4}$ and the proof of Theorem \ref{TbInt} } \label{b-4}

In the sequel we use decorators $\eps_1,\eps_2$ to emphasize the
dependence of an object on the parameter values $\eps_1,\eps_2$. 
E.g. $F_{\eps_1,\eps_2}$ denotes the
universal function $F$ for the multiplier $P_{\eps_1,\eps_2}$, etc.

\MLrevision\mpar{major change; material moved from proof of 
Thm 5.3}
We will always have to integrate a symbol depending
only on $|\eta|^2$ over $\R^2$. By change of variables we have 
for such a function (\cf\Eqref{EqResExpIntegral1})
\begin{equation}\label{EqResExpIntegral5}
  \int_{\R^2} f(|\eta|^2) \dxi = \frac{1}{4\pi|\Im\tau|} \int_0^\infty
  f(r) dr.
\end{equation}

The existence of the asymptotic expansion in Theorem \ref{TbInt}
follows from Theorem \ref{TResExp} by the usual contour integral
argument. Namely,
\begin{equation} \label{EqContourIntHeat}
     e^{-t P} = \frac1{\tpii} \int_{C} e^{-t\gl} (\gl-P)^{-1} d\gl
     = - \frac1{\tpii \, t} \int_{C} e^{-t\gl} (\gl-P)^{-2} d\gl
\end{equation}
where $C$ is the usual contour encircling the positive real axis.
One now argues as in \cite[Sec.~1.8]{Gil:ITH} to translate the
resolvent expansion into the heat expansion. Concretely, one gets
for $j>0$
\begin{equation}\label{eq.rev3}
   a_{2j}(a,P)  = \frac1{\tpii} \int_C e^{-t\gl} \int_{\R^2}
     \psi\bl a b_{-2j}(\eta) \br \dxi d\gl.
\end{equation}
For the constant term we find by \Eqref{EqResExp1}, 
and Theorem \ref{TEffTrace}
\begin{align*} 
   a_0(a,P) & = - \frac1{\tpii} \int_C e^{-t\gl} \int_{\R^2}
     \psi\bl a (k^2 |\eta|^2-\gl)^{-2} \br \dxi d\gl\\
    & = \frac{-1}{\tpii \cdot 4\pi |\Im\tau|} \int_C e^{-t\gl} \int_0^\infty
      \psi\bl a ( k^2 r -\gl)^{-2}\br dr d\gl \\
    & = \frac{1}{4\pi |\Im\tau| } \psi\bl a k^{-2} 
           \frac{-1}{\tpii}
         \int_C e^{-t\gl} \gl\ii d\gl \br \\
    & = \frac{1}{4\pi |\Im\tau| }\psi ( a k^{-2} ),
\end{align*}
\cf \Eqref{EqResExpIntegral5}.

It remains to compute $a_2(a,P)$.  In view of \Eqref{eq.rev3} this can be done  by
integrating each summand of $b_{-4}$ and then applying the Rearrangement Lemma
\cite[Lemma 6.2]{ConMos2011}, \cite{Les2014}. We will adopt a slightly
different approach here which will free us from
calculating some of the more tedious integrals and which at the same
time shows several a priori relations between the functions $F,
G^{\Re/\Im}$ for various multipliers. First, we prove
the Theorem \ref{TbInt} for $P_\gamma=k^2 \Lapl_\tau \oplus \bl k^2 \Lapl_\tau
+ (\dtau k^2) \dtau^* \br$\MLrevision\mpar{parenthesis was missing} and the \emph{graded} trace. 
That is we show that the second heat coefficient of 
\[
   \Tr_{\psi,\gamma}\bl a e^{-t P_\gamma} \br
   = \Tr_\psi \bl a e^{-t k^2 \Lapl_\tau} \br - \Tr_\psi\bl
      e^{-t (k^2 \Lapl_\tau + (\dtau k^2) \dtau^*)} \br
\]
is of the form \Eqref{eq.rev3} and we need to identify the corresponding
functions $F, G^\Re, G^\Im$.

Then by employing a simple symmetry argument we will derive a priori
relations between the functions $F_{1,1}$ and $F_{0,0}$ resp.
$G^{\Re}_{1,1}$ and $G_{0,0}$, \Eqref{EqConjA}, \eqref{EqConjB}. 
This will allow us to express the
functions for all parameter values of $(\eps_1,\eps_2)$ in terms of
those for the graded case.

\subsection{Some concrete integrals}

For the evaluation of the integrals we use freely the notation
and results of \cite[Sec.~4,5]{Les2014}. Actually, from this
paper we will only need a few explicit integrals which
can be evaluated using the residue calculus. In order to
be self-contained, we collect here the relevant formulas.

Let $p,m\in \Z_+, p\ge 1$, and $\ga\in\Z_+^p$ a
multiindex. Then for $s\in(\C\setminus \R_-)^p$ consider the
integral
\begin{align}
    H_{\ga}^{(p)}(s,m) &:= \int_0^\infty x^m (1+x)^{-\ga_0-1}
                            \prod_{j=1}^p (x+s_j)^{-\ga_j-1} dx
                                \label{EqHIntA}\\
           & = \int_0^\infty x^{|\ga|+p-1-m} (1+x)^{-\ga_0-1}
                            \prod_{j=1}^p (1+s_jx)^{-\ga_j-1} dx
                            \label{EqHIntB}\\
           & = (-1)^{m+|\ga|+p-1} [1^{\ga_0+1},
s_1^{\ga_1+1},\ldots,s_p^{\ga_p+1}] \id^m \log.
\end{align}
Here, $[1^{\ga_0+1}, s_1^{\ga_1+1},\ldots,s_p^{\ga_p+1}] \id^m \log$
stands for the divided difference of order $|\ga|+p+1$ of the function
$x \mapsto x^m \log x$,  where the $1$ is repeated $\ga_0+1$ times, $s_1$
is repeated $\ga_1$ times etc. We will only need the cases $p=1, m=0$
and $p=2, m\in \{0,1\}$. Furthermore, we will always have $\ga_p=0$.

\subsubsection{$p=0, m=0, \ga_0=\ga, \ga_1=0$}
Then, \cf \cite[Sec.~5]{Les2014}, 
\begin{equation}\label{EqOneVarInt}
  \begin{split}
    H_{(\ga,0)}^{(1)}(u,m=0) & = (-1)^\ga [1^{\ga+1},u] \log\\
      & = \frac{(-1)^\ga}{(u-1)^{\ga+1}} 
      \Bl \log u - \sum_{j=1}^m 
          \frac{(-1)^{j-1}}j (u-1)^j \Br =:\cL_\ga(u),
\end{split}
\end{equation}  
which is the modified logarithm introduced in \cite[Sec.~3 and 6]{ConTre2011}.

\subsubsection{$p=2, \ga_2=0$} Using either the calculus of divided
differences or by direct verification the function
$H_{(\ga_0,\ga_1,0)}^{(2)}(u,v,m=0)$ can be expressed in terms of
the modified logarithm, \cf\cite[Eq.~5.1]{Les2014}:
\begin{equation}\label{EqTwoVarIntA}
    H_{(\ga_0,\ga_1,0)}^{(2)}(u,v,m=0) = \frac{(-1)^{\ga_1+1}}{\ga_1!}
\pl_u^{\ga_1} \Bl \frac 1{v-u} \bl \cL_{\ga_0}(v) - \cL_{\ga_0}(u)\br \Br. 
\end{equation}

Finally, for $m=1$ and $\ga_1>0$ one infers from
\[
      (\ga_1+ u\pl_u) (1+xu)^{-\ga_1} = \ga_1 (1+ux)^{-\ga_1-1}
\]
and differentiation under the integral \Eqref{EqHIntB} that
\begin{equation}\label{EqTwoVarIntB}
    H_{(\ga_0,\ga_1,0)}^{(2)}(u,v,m=1) =  
    \frac 1{\ga_1} (\ga_1+u \pl_u) H_{(\ga_0,\ga_1-1,0)}^{(2)}(u,v,m=0).
\end{equation}


\subsection{The contributions of the summands of $b_{-4}$}

Each summand of $b_{-4}$ in Theorem \ref{TResExp} is of the form
\[
  k^2  f_0(|\eta|^2 k^2) \cdot a\cdot f_1(|\eta|^2 k^2 ) 
\]
or of the form 
\[
  k^2  f_0(|\eta|^2 k^2) \cdot a\cdot f_1(|\eta|^2 k^2 ) 
   \cdot b \cdot f_2(|\eta|^2 k^2 ). 
\]
In the first case integration over $\R^2$ with respect to $\dxi$ and
application of the Rearrangement Lemma (\cite[Lemma 6.2]{ConMos2011},
\cite{Les2014}) yields (up to the factor in front)
\[
  \int_0^\infty k^2 f_0(x k^2 ) \cdot a \cdot f_1(x k^2) dx
   = F(\modu)(a)
\]
with $F(u) = \int_0^\infty f_0(x) f_1(xu) dx$.

In the second case we have similarly
\[
  \int_0^\infty k^2 f_0(x k^2 ) \cdot a \cdot f_1(x k^2) \cdot
  b\cdot f_2(x k^2) dx = G(\modua,\modub)(a\cdot b)
\]
with 
\[
  G(u,v) =  \int_0^\infty f_0(x) f_1(xu) f_2(x u v) dx.
\]
Often we will have $\modu(a)$ instead of $a$. Note that
$G(\modua,\modub)(\modu(a)\cdot b) = \tilde G(\modua,\modub)(a\cdot
b)$ with $\tilde G(u,v)=u G(u,v)$. 

\subsection{The summands \Eqref{EqResExp6} of $b_{-4}$}
We first discuss the terms in the last line
\Eqref{EqResExp6} of the expression for $b_{-4}$ 
in Theorem \ref{TResExp}. As usual (\cf\cite[Sec.~6]{ConMos2011}) we
put $\gl = -1$. Furthermore, by slight abuse of notation
we also write $b(x)=(1+x)\ii$ under the integral. This is
justified since each term $(k^2|\eta|^2-\gl)\ii$ in $b_{-4}$
contributes, for $\gl=-1$, to a factor 
$b(x)=(1+x)\ii$ under the integral on the right of
\Eqref{EqResExpIntegral5}. In all integrals in the sequel
we will tacitly omit the overall factor $\frac1{4\pi |\Im\tau|}$.

\subsubsection{The summand $-b a_0 b$}
The total integral of $-b a_0 b$ equals, \cf \Eqref{EqOneVarInt},
\begin{equation}
  -\int_0^\infty b(x) a_0 b(x) dx = - \cL_0(\modu)(k^{-2} a_0).
\end{equation}

\subsubsection{The summand
$|\eta|^2 b\cdot \bl (\dtau k^2) b (\dtau^* k^2)
               +(\dtau^* k^2) b \dtau k^2 \br \cdot b$}
We move all $k$--powers to the left and rewrite this as
\[
    \ldots =  
     |\eta|^2 k^4 b\cdot \bl \modu (k^{-2} \dtau k^2) b
       (k^{-2} \dtau^* k^2) +\modu(k^{-2}\dtau^* k^2) b
       (k^{-2} \dtau k^2) \br \cdot b.
\]       
The total integral therefore equals
\[
  2 \modua g_3(\modua,\moduab)(\tsRk)
    =: G_3(\modua,\modub)(\tsRk)
\]
with
\begin{align}
    g_3(u,v) & = 2 \int_0^\infty x b(x) b(xu) b(xv) dx \nonumber \\
     & = 2 H_{(0,0,0)}^{(2)}(u,v) = -\frac{2}{v-u}
  (\cL_0(v)-\cL_0(u)),
\end{align}
(\cf \Eqref{EqTwoVarIntA}) resp.
\begin{equation}\label{EqG3}
  G_3(u,v)  = u g_3(u,uv)
    = 2 \frac{(uv-1)\log(u) - (u-1) \log(uv)}{(u-1)(v-1)(uv-1)}.
\end{equation}

\subsection{Proof in the graded case}
For the multiplier $P_\gamma$
($\eps_1=1, \eps_2=0$) and the graded trace several summands of
$b_{-4}$ vanish and it remains
\[
  b k^2 \Lapl_\tau b + b \dtau k^2 b \dtau^* b.
\]

\subsubsection{$b k^2 \Lapl_\tau b$}
For the first summand we find
\begin{equation}
  \begin{split}
     b k^2 \Lapl_\tau b  = b k^2 \dtau \dtau^* b 
      &  = - |\eta|^2 k^4 b^2 k^{-2}( \Lapl_\tau k^2) b  \\
      &\quad + |\eta|^4 k^6 b^2 \modu(k^{-2}\dtau k^2) b (
     k^{-2}\dtau^* k^2) b \\
      &\quad + |\eta|^4 k^6 b^2 \modu(k^{-2}\dtau^* k^2 )b (k^{-2}\dtau k^2) b.
      \end{split}
\end{equation}      

The total integral of $- |\eta|^2 k^4 b^2 k^{-2}( \Lapl_\tau k^2) b$
equals $F_\gamma(\modu)(k^{-2} \Lapl_\tau k^2)$, where,
\cf \Eqref{EqTwoVarIntA},
\[
  F_{\gamma}(u) = - \int_0^\infty x b(x)^2 b(xu) dx 
  = - \cL_1(u) 
  = \frac{\lgu - u +1}{(u-1)^2}.
\]

The remaining two summands of $\Lapl_\tau b$ contribute a summand
\[
\modua g_1(\modua,\moduab)(\tsRk)=:G_1(\modua,\modub)(\tsRk)
\]
to $G^\Re$, where
\begin{align}
       g_1(u,v) & = 2 \int_0^\infty x^2 b(x)^2 b(xu) b(xv) dx \\
        & = 2 H_{(1,0,0)}^{(2)}(u,v) =
          \frac{-2}{v-u}(\cL_1(v)-\cL_1(u)),
\end{align}
(\cf \Eqref{EqTwoVarIntA}) resp.
\begin{align}
  G_1(u,v) & = u g_1(u,uv) 
     = \frac{2}{v-1} (F_{\gamma}(uv)-F_{\gamma}(u)) \nonumber\\
     & = \frac{-2 u ( uv+u-2)\lgu}{(u-1)^2(uv-1)^2}+\frac{2\lgv}{(v-1)(uv-1)^2}
     +\frac{2u}{(u-1)(uv-1)}.\label{EqG1}
\end{align}

\subsubsection{$b \dtau k^2 b \dtau^* b$}
Writing $\dtau^* b = - |\eta|^2 b (\dtau k^2) b$ the total integral of
this summand equals 
$\modua g_2(\modua,\moduab) (\tsRk+ \tsIk)
=: G_2(\modua,\modub)(\tsRk+\tsIk)$.
Since this is the only summand which is not symmetric in $\dtau k^2$
and $\dtau^* k^2$ it is the only one contributing to $G_\gamma^\Im$.
Thus we have $G_\gamma^\Im(u,v) = G_2(u,v)= u g_2(u,uv)$ 
and this function contributes another summand to 
$G_\gamma^\Re$. Explicitly,
\begin{align*}
    g_2(u,v) & = - \int_0^\infty x b(x) b(xu)^2 b(xv) dx \\
      & = - H_{(0,1,0)}^{(2)}(u,v,m=1)\\
      & = - (1 + u \pl_u) H_{(0,0,0)}^{(2)}(u,v,m=0)
    = (1+u\pl_u) \frac{1}{v-u} (\cL_0(v) -\cL_0(u)).
\end{align*}  
In the last line we have used \Eqref{EqTwoVarIntB}. Thus
\begin{align}
  G_2(u,v)& = G_\gamma^\Im(u,v) = u g_2(u,uv) \nonumber\\
    & = \frac{v\lgv}{(v-1)^2(uv-1)}
               + \frac{u\lgu}{(u-1)^2(uv-1)}
               - \frac{1}{(u-1)(v-1)}.\label{EqG2}
\end{align}

The proof for $P_\gamma$ and the graded trace is thus complete.
Summing up we have:       
\begin{align}
  F_\gamma(u) & = \frac{\lgu - u +1}{(u-1)^2} \label{EqFgamma}\\
      G^\Im_\gamma(u,v) & = 
                   \frac{u\lgu}{(u-1)^2(uv-1)}
                  +  \frac{v\lgv}{(v-1)^2(uv-1)}
                  - \frac{1}{(u-1)(v-1)} \label{EqGIgamma}\\
      G^\Re_\gamma(u,v)
          & = G_1(u,v) + G_\gamma^\Im(u,v)\nonumber \\
          & =  \frac{2}{v-1} (F^{\gamma}(uv)-F^{\gamma}(u))
                       + G^\Im_\gamma(u,v) \nonumber\\
                       & =        
                  -\frac{u (uv+2u-3)\lgu}{(u-1)^2(uv-1)^2}
               +\frac{(uv^2+v-2)\lgv }{(v-1)^2(uv-1)^2} \nonumber \\
               &\quad  + \frac{uv-2u+1}{(u-1)(v-1)(uv-1)}.
               \label{EqGRgamma}
\end{align}  

\subsection{Proof in general by reduction to the graded case} 
Integrating each summand of $b_{-4}$ over $\R^2$ and applying the
Rearrangement Lemma (\cite[Lemma 6.2]{ConMos2011}, \cite{Les2014}) the
existence of the functions $F,G^\Re, G^\Im$ is clear in general. It remains to
find an explicit expression for them which will then also prove the remaining
properties. 

\subsubsection{Conjugation argument} 
\label{SSSConjArg}
Note that Leibniz' rule implies
\[
  \modu(P_{0,0}) = k^{-2} k^2 \Lapl_\tau k^2 = P_{1,1} +
  \Lapl_\tau(k^2).
\]  
Hence (\cf Remark \ref{r.trace_psido})
\MLrevision\mpar{trace corrected}
\[
    \Tr_\psi \bl a e^{-t P_{1,1}} \br 
  = \Tr_\psi \bl (k^2 a k^{-2}) 
      e^{-t (P_{0,0} - k^2 (\Lapl_\tau k^2) k^{-2})}\br
\]
and comparing with \Eqref{EqbInt} for $P_{1,1}$ and $P_{0,0}$
yields the a priori relations
\begin{align}
  F_{1,1}(u) & = u F_{0,0}(u) + \frac{\lgu}{u-1},\label{EqConjA}\\
    G_{1,1}^\Re(u,v) & = uv\cdot G^\Re_{0,0}(u,v),\label{EqConjB}
\end{align}  
resp. comparing with \Eqref{EqbIntCor} yields the a priori relations
\begin{align}
  K_{1,1}(s) & = e^s F_{0,0}(s) + 1\label{EqConjKA}\\
    H_{1,1}^\Re(s,t) & = e^{s+t}\cdot H^\Re_{0,0}(s,t)
        + 2 \frac{(uv-1)\lgu-(u-1)\lguv}{(uv-1)\lgu\lgv} \nonumber\\
     & = e^{s+t}\cdot H^\Re_{0,0}(s,t)
        + 2 \frac{(e^{s+t}-1)s-(e^s-1)(s+t)}{(e^{s+t}-1)s t}.\label{EqConjKB}
\end{align}  

\subsubsection{The one variable function $F$} The first summand
\Eqref{EqResExp3} of $b_{-4}$ is the only summand 
which contributes to the one variable function $F$. Therefore,
\begin{equation}
    F_{\eps_1,\eps_2} = F_{0,0} - (\eps_1+\eps_2) F_\gamma.
\end{equation}
Applying this and \Eqref{EqConjA} to $F_{1,1}$ we find
\[
  u F_{0,0}(u) + \frac{\lgu}{u-1} = F_{1,1} = F_{0,0} - 2 F_\gamma,
\]
and solving for $F_{0,0}$ we obtain
\begin{equation}\label{EqF}
    F_{\eps_1,\eps_2}(u)  = -\frac{2+(\eps_1+\eps_2)(u-1)}{u-1} F_\gamma(u) -
  \frac{\log(u)}{(u-1)^2}.
\end{equation}  

\subsubsection{The two variable functions} 
The terms in \Eqref{EqResExp5} are the only ones which contribute
to $G^\Im_{\eps_1,\eps_2}$, thus
\begin{equation}\label{EqGI}
  G^\Im_{\eps_1,\eps_2}(u,v) = G^\Im_{\eps_1-\eps_2,0} =
  (\eps_2-\eps_1)\cdot G_\gamma^\Im.
\end{equation}
Similarly for $G^\Re$
\begin{align}
  G^\Re_{\eps_1,\eps_2}  
     & = G^\Re_{\eps_1+\eps_2,0} + \eps_1\eps_2\cdot G_3,\\
  G^\Re_{\eps,0} & = G_{0,0}^\Re - \eps G_\gamma^\Re.
\end{align}
As for $F$ this allows to compute $G_{0,0}^\Re$ using \Eqref{EqConjB}:
\[
  uv G_{0,0}^\Re(u,v) = G_{2,0}^\Re(u,v) + G_3(u,v)
     = G_{0,0}(u,v) - 2 G_\gamma^\Re(u,v) + G_3(u,v),
   \]
and hence
\begin{equation}\label{EqGR}
  G^\Re_{\eps_1,\eps_2} (u,v)= -
  \frac{2+(\eps_1+\eps_2)(uv-1)}{(uv-1)}G^\Re_\gamma(u,v)
    + \frac{1+\eps_1\eps_2(uv-1)}{uv-1}G_3(u,v).
\end{equation}

\subsection{Explicit formulas in terms of $\log k^2$}
\newcommand{\css}{\cosh(s)}
\newcommand{\sss}{\sinh(s)}
\newcommand{\cst}{\cosh(t)}
\newcommand{\sst}{\sinh(t)}
\newcommand{\ssst}{\sinh(s+t)}
\newcommand{\cssh}{\cosh(\frac s2)}
\newcommand{\sssh}{\sinh(\frac s2)}
\newcommand{\csth}{\cosh(\frac t2)}
\newcommand{\ssth}{\sinh(\frac t2)}
\newcommand{\sssth}{\sinh(\frac {s+t}2)}
In view of \Eqref{EqFK}, \eqref{EqGHR}, and \eqref{EqGHI}
the formulas \Eqref{EqF}, \eqref{EqGR}, and \eqref{EqGI}
can immediately be translated into formulas
for the functions $K, H^\Re, H^\Im$.
\begin{align}
  K_\gamma(u) 
  & = \frac{u-1}{\lgu} F_\gamma(u) \label{EqKgamma}
           = \frac{\lgu -u +1}{\lgu (u-1)}\\
      & = \frac1{e^s-1} - \frac 1s  = \frac 12  \coth(s/2) -
           \frac1{2s} -\frac 12.\nonumber\\
           K_{\eps_1,\eps_2}(s) \label{EqKeps}
     & = - \frac{2+(\eps_1+\eps_2)(u-1)}{u-1} K_\gamma(s) - \frac1{u-1} \\
     & = - \frac{2+(\eps_1+\eps_2+1)(u-1)}{(u-1)^2}
         + \frac{2+(\eps_1+\eps_2)(u-1)}{s(u-1)}\nonumber\\
         H^\Im_\gamma(s,t)  &= \label{EqHIgamma}
         \frac{v(u-1)}{s(v-1)(uv-1)}
         + \frac{u ( v-1)}{t (u-1)(uv-1)} - \frac1{st}\\
H_{\eps_1,\eps_2}^\Im(s,t) & = (\eps_2-\eps_1) H_\gamma^\Im(s,t)
            \label{EqHIeps} \\
H^\Re_\gamma(s,t) \label{EqHRgamma}
      & = H_\gamma^\Im(s,t)\\
      & \quad - \frac{2 u (v-1)}{(uv-1)(u-1) t }  +
  \frac{2}{s(s+t)}\nonumber \\
      & = 
         \frac{t (u-1) v}{s(s+t) (v-1)(uv-1)}
          - \frac{s u (v-1)}{t (s+t) (u-1)(uv-1)} \nonumber\\
      &\quad    + \frac{u-v}{(s+t) (u-1)(v-1)} + \frac{t-s}{st (s+t)}
          \nonumber\\
          H^\Re_{\eps_1,\eps_2}(s,t) \label{EqHReps}
  & = - \frac{2+ (\eps_1+\eps_2)(uv-1)}{uv-1} H_\gamma^\Re(s,t)\\
     &\quad + 2 \eps_1\eps_2 \cdot
\frac{(uv-1)s-(u-1)(s+t)}{s\cdot t\cdot (uv-1)}\nonumber
\end{align}  

Note in particular, that
\begin{equation}\label{EqHrel}
  H_{1,0}^\Re(s,t) = \frac{1+uv}2 \cdot H_{0,0}(s,t).
\end{equation}

\subsection{Comparison with \cite{ConMos2011}}
\label{SSComparison}
\subsubsection{Degree $0$}
When comparing with \cite{ConMos2011} one has to take into account
that in degree $0$ they work with $k\Lapl k = \modu^{1/2}(k^2 \Lapl)$ 
instead of  $k^2\Lapl$; recall that $\modu = k^{-2}\cdot k^2$
denotes the modular operator and $\Lapl$ the Laplacian. 
This means that their one variable functions must be
multiplied by $u^{-1/2}$ and their two variable functions must
be multiplied by $(uv)^{-1/2}$ before comparing with our functions.
Also note that there are slightly different sign and normalization conventions
which lead to an overall factor $-2$ for one variable functions 
resp. $-4$ for two variable functions.

The adjusted functions of \cite{ConMos2011} are
\newcommand{\CM}{{\textup{CM}}}
\begin{align}
  K_0^\CM(s) & = \frac{2(2+e^s(s-2)+s)}{s (e^s-1)^2}
                = 2\frac{(u+1) \lgu -2u+2}{\lgu (u-1)^2}\\
             & = - 2 K_{0,0}(s)\nonumber   \\
H_0(s,t)^\CM  & = e^{-(s+t)/2} \cdot \\ 
     &\qquad \frac{\begin{subarray}{l}
       t(s+t)\cosh(s) - s (s+t)\cosh(t)\\
       \qquad\qquad    +(s-t)\bl s+t+\sinh(s)+\sinh(t)-\sinh(s+t)\br
       \end{subarray}}{%
                    st(s+t)\sinh(s/2)\sinh(t/2)\sinh^2((s+t)/2)}
           \nonumber \\
    & =\nonumber 
    8\cdot \frac{ \begin{subarray}{r} t^2 (u-1)^2 v - s^2 u (v-1)^2 + st (u-v)(uv-1)\\ 
                 \qquad\qquad+ (t-s)(u-1)(v-1)(uv-1)
             \end{subarray}}{%
                  s t (s+t) (u-1)(v-1)(uv-1)^2}    \\
                  &=  \nonumber
    \frac8{uv-1} H^\Re_\gamma(s,t)= - 4 H^\Re_{0,0}(s,t).
\end{align}

\subsubsection{Degree $1$} 
From \cite[Sec.~3.1.1]{ConMos2011} we deduce 
$2 K_1^\CM = K_0^\CM - K_\gamma^\CM$ (here $K_0^\CM$ is the unadjusted function from
loc.~cit.), hence
\begin{align}
    K_1^\CM(s) & = \frac{4 s u - 2 u^2 +2}{s (u-1)^2} = -2
  K_{1,0}(s).\\
   H_1^\CM(s,t) & = \cosh(\frac{s+t}2) \cdot H_0^\CM(s,t) 
  = -4 \frac{uv+1}2 H_{0,0}^\Re(s,t) = -4 H_{1,0}^\Re(s,t),
\end{align}
\cf\cite[Eq.~3.3]{ConMos2011} and \Eqref{EqHrel}.
\begin{align}
  S^\CM(s,t) & = \frac{s+t-t\css-s\cst - \sss-\sst-\ssst}{%
  s t \sssh \ssth \sssth } \\
    & = 4 \Bigl[
       \frac1{ st} - \frac{(u-1)v}{s(v-1)(uv-1)}
          - \frac{u(v-1)}{t(u-1)(uv-1)}\Bigr]\nonumber \\
         & = - 4 H_\gamma^\Im(s,t) = 4 H_{1,0}^\Im(s,t). \nonumber
\end{align}
\section{Effective pseudodifferential operators and resolvent}
\label{SEffPsiDO}

In this section we will use the notation of the Appendix 
\ref{AppA}, in particular \ref{SSAppLattice}. The algebra
$\sB$ introduced in \ref{SSAppLattice} may be thought of
being $A_\gt$ or $A_{\gt'}$. 

The effective implementation of the pseudodifferential calculus amounts to
passing from its realization on multipliers to a direct action on Heisenberg
modules (or on the Hilbert space $L^2(\sB,\vfB)$ itself).  More concretely,
let $\spi:G\to \cL(\cH)$ be a projective unitary representation of
$G=\R^n\times(\R^n)^\wedge$. \footnote{For definiteness we discuss the
pseudodifferential calculus on $\R^n$ only. The extension for groups of the
form $\R^n\times F$, with $F$ finite, does not lead to any new aspects.} For a
symbol $f\in \SmB$ the assignment
\[
  \Op(f):= \int_G f^\vee(y) \spi(y) dy
\]
provides a homomorphism of the algebra $L^m_\gs(G,\Binf)$ into the (unbounded)
operators in the Hilbert space $\cH$, it is clearly a $*$--representation of
the algebra of pseudodifferential multipliers of order $\le 0$ on $\cH$. We
will be interested in two cases:
\begin{enumerate}
\item $\spi=\tilde\pi=\pi_0$ resp. $\spi=\pi_w$,
 where $\tilde\pi$ is the representation defined
at the end of example \ref{SAppExample} and 
$\pi_0, \pi_w$ are the representations defined in \Eqref{pi-0},
\eqref{eq:1209061}.
\item $\spi=\tilde\ga$ is the unitary representation of $\R^2$ on
the Hilbert space $L^2(\sB,\vfB)$ induced by the normalized action
$\tilde\ga$ on $\sB$, \cf Sec.~\ref{AppB}.
\end{enumerate}
During this section, unless otherwise said, $\spi$ denotes one of
these two representations. The operator convention $\Op$ refers
to $\spi$ as explained above.

It is well-known that the singular support of
the Fourier transform of a symbol is contained in $\{0\}$.
This extends easily to $\vfB$--valued symbols. 
\begin{lemma}\label{LSymbolEstimates}
\textup{1. } Let $f\in\SmB$ be a symbol of order $m$ and let
$\chi\in\cinf{\R^n}$ be a smooth function which vanishes for
$|x|\le \delta_1$ and which is constant $1$ for $|x|\ge \delta_2$.
Then $\chi \hat f\in\sS(\R^n,\Binf)$.

\textup{2. } If $f\in \sym^m(\R^n\times\Gamma,\Binf)$ 
is a parameter dependent symbol then 
$\chi\hat f\in \sS(\R^n\times \Gamma,\Binf)$.
\end{lemma}
For the parameter dependent symbols see the remarks at
the beginning of Sec.~\ref{SSPsiDO}.
\begin{proof} The proof is standard. We just sketch the main
step in the more general case 2. Let $\ga,\gb$ be multiindices,
$\gamma\in \Z_+$, and let $p$ be a seminorm on $\Binf$.
We may for convenience assume that the parameter $\gl$ in
the parameter dependent calculus is treated as a covariable of
order $1$. Otherwise replace $\gl$ by $\gl^{1/\operatorname{ord}(\gl)}$.
Then 
\begin{align*}
   p\Bl \xi^\ga \pl_\xi^\gb \pl_\gl^\gamma \hat f(\xi,\gl) \Br
      & = p \Bl 
     \int_{\R^n} e^{-i \inn {x,\xi }} \pl_x^\ga \bl x^\gb \pl_\gl^\gamma
                    f(x,\gl)\br \, dx      \Br \\
          & \le C \int_{\R^n}   (1+|x|+|\gl|)^{m -|\gamma|-|\ga|+|\gb|} dx \\
      & \le C (1+|\gl|)^{n+m-|\gamma|-|\ga|+|\gb|},
\end{align*}
as long as $n+m-|\gamma|-|\ga|+|\gb|< 0$. Since for given $\gamma,\gb$ we
may choose $\ga$ as large as we please the claim follows.
\end{proof}

The next Theorem relates the Hilbert space trace of $\Op\bl f(\cdot,\gl) \br$
to the natural trace on the pseudodifferential multipliers
(\cf Remark \ref{r.trace_psido}).\MLrevision\mpar{Sentence inserted}
 
\begin{theorem}\label{TEffTrace} Let $n=2$ and let
$f\in\sym^m(\R^2\times\Gamma,\Binf)$ be a parameter dependent symbol of order
$m<-n$. Then $\Op(f)$\MLrevision\mpar{$\pi(f)$ replaced by $\Op(f)$} is of trace class. Moreover, there is a constant
$N(\spi)$ depending only on the representation $\spi$ such that
\begin{align*}
   \Tr\Bl \Op\bl f(\cdot,\gl)\br\Br 
    & =  N(\spi) \vfB\bl  f^\vee(0,\gl)\br + O(\gl^{-\infty}), \quad
             \gl\to\infty,\\
    & =  N(\spi) \int_{\R^2} \vfB\bl f(\xi,\gl)\br \dxi + O(\gl^{-\infty}), \quad
             \gl\to\infty.
\end{align*}
We have explicitly,
\[
   N(\spi) = \begin{cases}
             |\rk\cE(g,\gt)| = |c\theta +d|,& \spi=\pi_w,\\
             \rk\cE(g,\gt)^2 = |c\gt+d|^2,& \spi = \tilde \ga.
             \end{cases}
\]
\end{theorem}
 \begin{proof} 
We preface the proof with two comments.
\begin{enumerate}
\item  First, the precise threshold $m<-n$ for 
the trace class property is proved as in the case of classical symbols (see \eg
\cite[Sec.~27]{Shu:POS}), and will not be reproduced below. 

\item Secondly,
the rank factor in $N(\tilde \ga)$ is an artifact of the fact that the action
$\tilde\ga$ is normalized such that its infinitesimal generators are the
derivations $\delta'$, \cf\Eqref{Eq052313}.
\end{enumerate}

Given now $f\in\sym^m(\R^2\times\Gamma,\Binf)$ we
extend it to $\R\times\Z_c\times\R\times \Z_c$ by putting
$g(x_1,\ga_1,x_2,\ga_2,\gl):= f(x_1,x_2,\gl)\delta_{\ga_1,0}\delta_{\ga_2,0}$.
Then, by definition, $\Op\bl f(\cdot,\gl)\br = \tilde\pi(f^\vee(\cdot,\gl))$, and
\begin{align*}
   \tilde\pi(f^\vee(\cdot,\gl)) & = \int_{\R^2} f^\vee(x,\gl)
\tilde\pi(x) dx \\
     & = |c|\cdot \int_G g(x_1,\ga_1,x_2,\ga_2,\gl) \pi(x_1,\ga_1,\mu x_2,\ga_2) \,
            d\mu_G(x_1,\ga_1,x_2,\ga_2) \\
     & = \frac {|c|}{|\mu|}\cdot \int_G g(x_1,\ga_1,\frac{x_2}\mu,\ga_2;\gl) 
          \pi(x_1,\ga_1,x_2,\ga_2)  \,d\mu_G(x_1,\ga_1,x_2,\ga_2).
\end{align*}
The constant $|c|$ in the numerator appears since the Haar measure
$\mu_G$ equals $\frac 1{|c|}\gl\otimes\#\otimes\gl\otimes \#$,
\cf Example \ref{SAppExample}.
The trace formula Theorem \ref{TAppTraceFormula} now implies     
\[
  \Tr\bl\Op(f(\cdot,\gl))\br
      = \sum_{\ell\in L} \vfB\bl g(\ell,\cdot)\pi(\ell) \br.
\]
By Lemma \ref{LSymbolEstimates} the sum $\sum_{\ell\in L\setminus\{0\}}$
is $O(\gl^{-\infty})$ as $\gl\to\infty$ and the remaining summand
for $\ell=0$ equals indeed $|c\gt+d|\cdot \vfB\bl f^\vee(0,\gl)\br $.
\smallskip

 Recall that $\tilde\ga_x(\pi(\ell_1\go_1+\ell_2\go_2)) = e^{\tpii\frac{\ell_1
x_1+\ell_2 x_2 }{\rk\cE(g,\gt)}} \pi(\ell_1\go_1+\ell_2\go_2)$.
Thus
\begin{align*}
  \Op(f) \pi(\ell_1\go_1+\ell_2\go_2)
     & = \int_{\R^2} f^\vee(x) 
         e^{\tpii\frac{\ell_1 x_1+\ell_2 x_2 }{\rk\cE(g,\gt)}} 
           \pi(\ell_1\go_1+\ell_2\go_2) dx \\
   & = f (-\frac {2\pi}{\rk\cE(g,\gt)}\ell_1, -\frac{2\pi}{\rk\cE(g,\gt)}\ell_2)
           \pi(\ell_1\go_1+\ell_2\go_2) .
\end{align*}
Applying the Poisson summation formula we find
\begin{align*}
   \Tr\bl\Op(f(\cdot,\gl))\br
      & = \sum_{\ell\in\Z^2} f (\frac {2\pi}{\rk\cE(g,\gt)}\ell_1,
\frac{2\pi}{\rk\cE(g,\gt)}\ell_2) 
       = \frac{\rk\cE(g,\gt)^2}{4\pi^2}\sum_{\gamma\in \bl\frac{2\pi}{\rk}\Z^2\br^\wedge}
                  \hat f(\gamma,\gl)\\
     & = \rk\cE(g,\gt)^2 \int_{\R^2} f(\xi) d\xi + O(\gl^{-\infty}).
\end{align*}
As above, the sum $\sum_{\gamma\not=0}$ is $O(\gl^{-\infty})$ and
the summand corresponding to $\gamma=0$ gives the claimed formula.
\end{proof}


\begin{theorem}\label{TEffHeatExp} Let $k^2=e^h\in\Binf$ be self-adjoint
and positive definite and let $a_0\in\Binf$. Furthermore, let 
\[
   P = \Op\bl k^2 |\eta|^2 + \eps_1 (\dtau k^2) \etabar + \eps_2
  (\dtau^*k^2) \eta + a_0\br, \quad \eta=\xi_1+\taubar\xi_2,
\]
be the effective realization of the differential multiplier
\Eqref{EqDiffMulP1} w.r.t. the representation $\spi$.
Then for $a\in\Binf$ we have an asymptotic expansion
\[
  \Tr\bl a e^{-t P} \br \sim_{t \searrow 0} 
       \sum_{j=0}^\infty a_{2j}(a,P,\spi)\, t^{j-1},
\]
where 
\[
a_{2j}(a,P,\spi) = 
   N(\spi)\cdot a_{2j}(a,P) =  N(\spi)\cdot \vfB\bl a A_{2j}(P)\br
\]
with $A_{2j}(P)\in \Binf$. Explicitly,
\begin{equation}\label{EqA0}
    a_0(a,P,\spi) = \frac{N(\spi)}{4\pi |\Im\tau|} \vfB\bl a k^{-2}\br,
\end{equation}
and $a_2(a,P)$ is given in \Eqref{EqbInt}.
\end{theorem}
\begin{proof} 
\MLrevision\mpar{calculation of $a_0$ moved to beginning of 
Section \ref{b-4}}
This follows immediately from Theorem \ref{TbInt} and \ref{TEffTrace},
\cf also the beginning of Section \ref{b-4}.
\end{proof}

\begin{example}\label{ExCheck} We check the normalization constants by
calculating explicit examples. Let $P$ be the effective realization of the
multiplier $\udtau^*\udtau$ with symbol 
$|\eta|^2 - \frac 12 c_\tau$, \cf Section \ref{SSSDiffOp2}.

1. Let $\spi=\tilde\pi=\pi_0$. Assume $\Im\tau>0$.
Then, \cf Prop. \ref{harmonic-osc} and \Eqref{Eq:1209064},
$P$ is a direct sum of $|c|$ copies of the operator 
$D^*D$ with
\[
  D = \frac{\pl}{\pl t} + \gl t, D^* = -\frac{\pl}{\pl t} + \ovl{\gl} t,
   \quad \gl = \tpii \mu \taubar,
\]
acting on $L^2(\R$). We have $c_\tau = [D,D^*] = 2 \Re\gl = 4\pi\mu\Im\tau$.

\paragraph*{$\mu>0$} Then $\spec D^* D = 4\pi \mu\Im \tau \cdot \Z_+$
and 
\begin{equation}\label{EqCheck1}
   \Tr ( e^{-t P})  = \frac{|c|}{1-e^{-4\pi\mu \Im\tau \, t}}
    = \frac{|c\gt+d|}{4\pi \Im \tau} t\ii + \frac{|c|}2+ O(t),\text{ as } t\to
0.
\end{equation}

\paragraph*{$\mu<0$} Then
$\spec D^* D = 4\pi |\mu|\Im\tau \cdot \Z_+\setminus\{0\}$ and
\begin{equation}\label{EqCheck2}
   \Tr ( e^{-t P})  = \frac{|c|}{e^{4\pi|\mu| \Im\tau\, t}-1}
    = \frac{|c\gt+d|}{4\pi \Im \tau} t\ii - \frac{|c|}2+ O(t),\text{ as } t\to
0.
\end{equation}

2. Let $\spi=\tilde\ga$. Then $\spec P$ consists of 
\[
   \frac {4\pi^2}{\rk\cE(g,\gt)^2} \bl   k_1^2 + |\tau|^2 k_2^2 + 2\Re\tau k_1
k_2\br,
  \quad k_1, k_2\in\Z.
\]
Hence by the Poisson summation formula (resp. transformation formula for
the $\gt$--function) one finds
\begin{align}
  \Tr(e^{-tP} ) & = \sum_{k\in\Z^2} e^{-4\pi^2/\rk^2 (k_1^2+|\tau|^2 k_2^2+2\Re
\tau_1\tau_2) t}\nonumber\\
    & = \frac{|\rk\cE(g,\gt)|^2}{4\pi |\Im\tau| t } 
       + O(t^\infty),\text{ as } t\to 0.\label{EqCheck3}
\end{align} 

\Eqref{EqCheck1}, \eqref{EqCheck2}, \eqref{EqCheck3} are consistent
with the values of $N(\spi)$ given in Theorem \ref{TEffTrace} and
with \Eqref{EqA0}. Furthermore, by \Eqref{EqbIntCor} and Theorem
\ref{TEffHeatExp} the constant term in the heat expansion
of $\Tr(e^{-tP})$ equals indeed
\[
   N(\spi) a_2(a,P) = \frac{N(\spi)}{4\pi\Im\tau} \vfB(- \frac12
(-c_\tau)) = \frac 12  N(\spi) \mu = \frac{|c\gt+d|}{2 (\gt+\frac dc)}=
      \frac{|c|}2 \sgn (\mu),
\]
and this is consistent with 
\Eqref{EqCheck1}, \eqref{EqCheck2}.

\end{example}

\appendix

\section{Heisenberg equivalence bimodules and the trace formula}\label{AppA}
\setcounter{subsubsection}{0}

For the convenience of the reader and to fix some notation
we briefly summarize here the main facts about Rieffel's general
construction~\cite[Sec.~2,3]{Rie1988} of equivalence bimodules
and how it specializes to our setting. We slightly modify
his Heisenberg cocycle since for $\R^n$ we prefer to have 
a skew bicharacter instead of the usual one, \cf Example 
\ref{SAppExample} below.

\subsubsection{Notation} In the sequel all groups will
be locally compact abelian (lca).  In all our examples, 
groups will be of the form $\R^n\times F$
with a finite abelian group. Therefore, we will confine 
ourselves to \emph{elementary} locally compact
abelian groups, that is groups of the form $\R^n\times T^m\times
\Z^k\times F$. For such groups, the Schwartz space $\sS(G)$ and
the space of smooth functions on $G$ can be
defined as usual. There will be no loss of generality if the
reader assumes that $G=\R^n\times F$.

Unless otherwise said, the Haar measure on discrete 
groups will be the counting measure, the Haar measure on 
compact groups will usually be normalized to $1$. 
For a finite group $F$ (compact and discrete)
the Haar measure will always be specified. 

For a lca group $G$ we denote
by $G^\wedge$ its Pontryagin dual, and by 
$\inn{\cdot,\cdot}:G\times G^\wedge\to T=\bigsetdef{z\in\C}{|z|=1}=\R/2\pi\Z$
the bicharacter implementing the duality between $G$ and
$G^\wedge$.

\subsection{Poisson summation for discrete cocompact subgroups}

Let $\Gamma\subset G$ be a discrete cocompact subgroup.
$\vol_G(F_\Gamma)=\vol_G(G/\Gamma)$ denotes the $G$--volume of a measurable
fundamental domain for the action of $\Gamma$ on $G$.  With the Haar
measure on the compact quotient normalized to $1$ we have for $f\in\sS(G)$
\begin{equation}\label{EqApp1}
  \int_G f(x) dx = \vol_G(F_\Gamma)\cdot\int_{G/\Gamma} \sum_{\grg\in\gG}
  f(x+\grg) \,d\dot x.
\end{equation}
Let $\Gamma^\perp\subset G^\wedge$ be the discrete group consisting 
of those $y\in G^\wedge$ with $\inn{\grg,y}=1$ for all $\grg\in\Gamma$.
Clearly, $\Gamma^\perp$ is naturally isomorphic to
$(G/\Gamma)^\wedge$. Then for $f\in\sS(G)$ one has the
\emph{Poisson summation formula}
\begin{equation}\label{EqPoisson}
  \sum_{\grg\in \Gamma} f(\grg) = \frac1{\vol_G(F_\Gamma)} \cdot
     \sum_{\ga\in\Gamma^\wedge} \hat f(\ga).
\end{equation}

\subsection{Covariance algebras twisted by a cocycle} \label{SSApp2}

Let $(\sA,G,\ga)$ be a $C^*$--dynamical system. By 
$\sA^\infty$ we denote as usual the smooth subalgebra of $\sA$
w.r.t.~the action $\ga$. Let $\go:G\times G \to T$ be a 
smooth \emph{cocycle} (aka multiplier \cite{Kle1962}). That is, $\go$ is a 
smooth map satisfying
\begin{align}
  \go(x,0_G) & = \go(0_G,x) = 1,\\
  \go(x,y) \go(x+y,z) & = \go(y,z) \go(x,y+z).
\end{align}
Applying this with $y=-x, z=x$ we see $\go(x,-x)=\go(-x,x)$.

Define a product and involution on $\sS(G,\sA^\infty)$ as
follows:
\begin{align}
     f^*(x)     &= \ovl{\go(x,-x)}\ga_x\bl f(-x)^*\br, \\
      (f*g)(x)  &= \int_{G} f(y) \ga_y\bigl( g(x-y)\bigr) \go(y,x-y) dy.
\end{align}
With this product and involution $\sS(G,\sA^\infty)$ becomes
a $*$--algebra. A canonical representation of this $*$--algebra
is obtained as follows:
$\sS(G,\sA^\infty)$ is a pre-$C^*$--module with $\Ainf$--valued inner product
\begin{equation}
      \inn{f,g}=\int_{G} f(x)^* g(x)dx.    
\end{equation}
Put 
\begin{align} \label{EqAppMultiplier1} 
  (a \, f)(t) &= \ga_{-t} (a) f(t)\,, \qquad a \in \Ainf ; \\
  (\uU_y f)(t) &=  \go(y,t-y) f(t-y).\label{EqAppMultiplier2}
\end{align}
Then $\uU_y, y\in G$, is a projective family of unitaries which implements the
group of automorphisms $\ga_y, y\in G$:
\begin{align}
  \uU_x^* & = \ovl{\go(x,-x)}\uU_{-x}, \quad \uU_x\, \uU_y= \go(x,y) \uU_{x+y}, \quad x,y\in G,\\
     \uU_x a \uU_{x}^*&= \ga_x(a), \quad a\in\sA^\infty.
\end{align}
By associating to $f\in\sS(G,\sA^\infty)$ the multiplier $M_f=\int_{G} f(x) \uU_x dx$
we obtain the \emph{left regular representation} of the $\go$--twisted
convolution algebra $\sS(G,\sA^\infty)$ on the Hilbert $C^*$--module
$L^2(G,\sA)$. The definition of \Eqref{EqAppMultiplier2} is
deliberately chosen such that $\uU_y\circ M_f = M_{\uU_y f}$. This
differs from the definition in \cite{Kle1962}. 

\subsubsection{Dual Trace}\label{SSSAppDualTrace} 
If $\psi$ is an $\ga$--invariant (finite) trace on $\sA$ then the
\emph{dual trace} $\psihat$ on $\sS(G,\sA^\infty)$ is given by
\begin{equation}\label{EqAppDualTrace}
 \psihat(f)= \psi\bl f(0)\br=\int_{G^\wedge} \psi\bl\hat{f}(\xi)\br
 d\xi.    
\end{equation}
In particular, on the $\go$--twisted convolution algebra $\sS(G)$
(that is $\sA=\C$ with trivial $G$--action) we have the trace
\begin{equation}\label{EqAppTrivDualTrace}
  \psihat(f) = f(0) = \int_{G^\wedge} \hat f(\xi) \, d\xi.
\end{equation}
In fact, the function $f$ can be expressed as
\begin{equation}\label{EqAppfTrace}
  \psihat(U_x^* f) = (U_x^* f)(0) = f(x).
\end{equation}

\subsection{Rieffel's construction of Heisenberg module}

Let $M$ be a lca group, $G:=M\times M^\wedge$, 
$\inn{\cdot,\cdot}:G\to T$ the duality pairing. We choose the
Haar measure on $G$ to be $\mu_G:=\mu_M\otimes \mu_{M^\wedge}$ with a Haar measure
$\mu_M$ on $M$ and its corresponding Plancherel measure $\mu_{M^\wedge}$ on
$M^\wedge$. This Haar measure on $G$ is self-dual w.r.t.~the 
bicharacter $\rho(\sx,\sy):=\inn{x_1,y_2}\ovl{\inn{y_1,x_2}},
\sx, \sy\in G$. Furthermore, the Fourier transform which sends
$f\in \sS(G)$ to $\hat f(\xi):=\int_G \rho(\xi,x) f(x) d\mu_G(x)$
is self-inverse.

We prefer to have a skew symmetric bicharacter on the $\R^n$ factor
of $G$. To allow for some flexibility we thus
define the Heisenberg representation of $G$ on $L^2(M)$
as follows. Fix a bicharacter $\gl:M\times M^\wedge\to T$ and put
\begin{equation}\label{EqApp13}
  (\pi(\sy)u)(t):= \gl(y)\inn{t,y_2} u(t+y_1),
         \qquad  \sy = (y_1, y_2)  \in G, \, t \in M .
\end{equation}
One has
\begin{equation}\label{EqApp14}
  \pi(\sx)\pi(\sy)=e(\sx,\sy) \pi(\sx+\sy), \qquad  \sx, \sy \in G ,
\end{equation}
where $e(\sx,\sy)=\inn{x_1,y_2}\ovl{\gl(x_1,y_2) \gl(y_1,x_2)}$ is
a cocycle satisfying
\begin{equation}
   e(\sx,\sy)\ovl{e(\sy,\sx)} = 
    \rho(\sx,\sy) = \inn{x_1,y_2}\ovl{\inn{y_1,x_2}},
\qquad \sx, \sy\in G.
\end{equation}
Note that $\rho$ is independent of the choice of $\gl$ and equals
therefore the bicharacter induced by the canonical Heisenberg
cocycle $\gb(\sx,\sy)=\inn{x_1,y_2}$. Furthermore,
$\pi(\sx)\pi(\sy)=\rho(\sx,\sy)\pi(\sy)\pi(\sx)$.

The representation $\pi$ integrates
to a representation of the $e$--twisted 
convolution algebra $\sS(G)$ on $L^2(M)$,
\begin{align} \label{EqEffRep}
\pi(f):=\int_G f(\sy) \pi(\sy) d\sy , \qquad f \in \sS(G).
\end{align}

\newcommand{\op}{\textup{op}}
\subsection{The algebras associated to a lattice} 
\label{SSAppLattice}
Let $L\subset G$ be a lattice, that is a discrete cocompact subgroup and let
$L^\perp\subset G$ be the dual lattice with respect to the skew bicharacter
$\rho$. Let $\sB$ be the $C^*$--algebra generated by $\pi(l), l\in L$ and let
$\sA$ be the $C^*$--algebra generated by $\pi(l')^{\op}, l'\in L^\perp$.

$x\mapsto \pi(x)\cdot \pi(x)^*$ defines a $G$--action $\ga_\cdot$ on $\sB$,
which is the integrated form of 
\[
  \pi(x)\pi(\ell)\pi(x)^*=\rho(x,\ell) \pi(\ell).
\]
$\Binf$ consists of the smooth vectors w.r.t.~this action.
Elements of $\Binf$ are of the form $\sum_{\ell\in L} f(\ell)\pi(\ell)$
with $f\in\sS(L)$. Thus $\Binf$ is nothing but the $e$--twisted
convolution algebra $\sS(L)$ with the cocycle $e$ restricted to $L$.
Since $L$ is discrete, the algebra $\Binf$ is unital with unit
$\delta_0(\ell)=\delta_{\ell 0}$. The dual trace $\vfB(f)=f(0)$ is a
normalized trace on $\sB$.

Needless to say, the same remarks apply to $\sA$.
%

\subsubsection{The 
integrated action of $\sS(G,\Binf)$ on $L^2(M)$}

We now apply the construction of Sec.~\ref{SSApp2} to the
dynamical system $(\sB,G,\ga)$. The representation
\Eqref{EqEffRep} extends to a representation of the
$e$--twisted convolution algebra $\sS(G,\Binf)$ on
the Heisenberg module $L^2(M)$. Namely, for 
$f\in\sS(G,\Binf)=\sS(G\times L)$, we put, by slight
abuse of notation:
\begin{equation}
  \begin{split}
     \label{EqAppProductAction} 
    \pi(f) & := \int_G \pi(f(x,\cdot)) \pi(x) dx 
      = \int_G \sum_{\ell\in L} f(x,\ell) \pi(\ell)\pi(x) dx \\
     & = \int_G \sum_{\ell\in L} f(x,\ell) e(\ell,x) \pi(\ell+x) dx 
       = \sum_{\ell\in L } \int_G f(x-\ell,\ell) e(\ell,x-\ell) \pi(x) dx \\
     & = \int_G \Bl \sum_{\ell\in L} f(x-\ell,\ell) e(\ell,x-\ell) \Br \pi(x) dx 
      =: \pi \bl \Phi_L(f) \br,
  \end{split}
\end{equation}
with
\begin{equation}\label{EqAppContraction} 
   \Phi_L(f)(x) =  \sum_{\ell\in L} f(x-\ell,\ell) e(\ell,x-\ell).
\end{equation}
It is not hard to see that $\Phi_L$ maps $\sS(G\times L)$ continuously
to $\sS(G)$, hence the action of $\pi(f)$ is in fact given
by the action of the \emph{scalar} Schwartz function $\Phi_L(f)$.

Thus we have the following commutative diagram of maps
\[
\xymatrix{ \sS(G\times L) \ar[dr]^{\pi}\ar[d]_{\Phi_L} & \\
       \sS(G) \ar[r] & \sL(L^2(M)).}
\]

\subsubsection{The trace formula}

\begin{theorem}\label{TAppTraceFormula}
Let $G=M\times M^\wedge$ be an elementary group,
$L\subset G$ a discrete cocompact subgroup and let
$\pi$ be the Heisenberg representation as outlined above. 

\paragraph*{1.} For $f\in\sS(G,\sB^\infty)\simeq \sS(G\times L)$
the action $\pi(f)$ on $L^2(M)$ is an integral operator with
Schwartz kernel $k\in \sS(M\times M)$. Concretely, 
with $g=\Phi_L(f)\in\sS(G)$ the Schwartz kernel is given by
\begin{equation}\label{EqApp20}
     k(t,s) = \bl g(s-t,\cdot) \gl(s-t,\cdot)\br^\vee(t),
       \qquad t,s \in M.
\end{equation}
Furthermore, we have the trace formula
\begin{equation}\label{EqAppTraceFormula}
  \Tr\bl \pi(f)\br = \widehat\psi\bl \Phi_L(f) \br
      = \sum_{\ell\in L} f(\ell,-\ell) e (-\ell,\ell) 
      = \sum_{\ell\in L} \psi^\sB\bl f(\ell,\cdot) \pi(\ell) \br.
\end{equation}

\paragraph*{2.} For any integral operator $K$ on $L^2(M)$
with kernel $k\in \sS(M\times M)$ the function
\begin{equation}\label{EqApp22}
  g(\sx) = \Tr\bl \pi(\sx)^* K \br = 
    \ovl{\gl(x_1,x_2)} \int_M k(t,x_1+t) \ovl{\inn{t,x_2}} dt
\end{equation}
is a Schwartz function such that $K=\pi(g)$.
\end{theorem}
\begin{proof} 1. We have already seen in \Eqref{EqAppProductAction}
that $\pi(f)=\pi(\Phi_L(f))$ and that $g=\Phi_L(f)$ is in the Schwartz
space $\sS(G)$.  In terms of $g$ the action of $\pi(f)$ on $L^2(M)$ is
indeed given by\MLrevision\mpar{$y_1$ replaced by $x_1$}
 \begin{align*}
(\pi(f) u) (t) & = \int_G g(x_1, x_2) \gl(x_1,x_2) 
                   \inn{t , x_2} u(t+x_1) d\sx \\
       & = \int_M \Bl \int_{M^\wedge}
             g(s-t, x_2) \gl(s-t,x_2) 
                   \inn{t , x_2} \, d x_2 \Br u(s) ds,
\end{align*}
and so its Schwartz kernel is  
 $k(t,s) = \bl g(s-t,\cdot) \gl(s-t,\cdot)\br^\vee(t), \, t,s \in M$.

Clearly, $k\in\sS(M\times M)$. It is well-known that integral operators whose
Schwartz kernels are Schwartz functions are trace class.
Applying Mercer's Theorem \cite[Prop.~3.1.1 p.~102]{Ber1972} for trace class
operators, one obtains 
\begin{align*}
\Tr\bl \pi(f)\br
   & = \int_M  k(t,t) dt 
     = \int_M  \bl g(0,\cdot) \gl(0,\cdot)\br^\vee(t) dt \\
   & = g(0_G) = \widehat\psi\bl \Phi_L(f) \br 
    = \sum_{\ell\in L} f(-\ell,\ell) e (\ell,-\ell).
 \end{align*} 
To see the last equality of \Eqref{EqAppTraceFormula}
one computes
\begin{align*}
   \sum_{\tilde\ell\in L} f(\ell,\tilde\ell) 
                  \pi(\tilde \ell) \pi(\ell)
        = \sum_{\tilde\ell\in L} f(\ell,\tilde\ell)
e(\tilde \ell,\ell) \pi(\tilde\ell+\ell) 
        = \sum_{\tilde\ell\in L} 
                  f(\ell,\tilde\ell-\ell)e(\tilde \ell-\ell,\ell)
                      \pi(\tilde\ell),
\end{align*}
thus $\vfB\bl f(\ell,\cdot)\pi(\ell)\br 
= f(-\ell,\ell) e(-\ell,\ell)$ and 1. ist proved.
       

2. Given $k\in \sS(M\times M)$ we invert \Eqref{EqApp20} and put
\begin{equation}
  g(x_1,x_2) \gl(x_1,x_2) := k(\cdot,x_1+\cdot)^\wedge(x_2)
      =  
         \int_M k(t,x_1+t) \ovl{\inn{t,x_2} }dt;
\end{equation}
certainly $g\in \sS(G)$.
\end{proof}
\begin{remark} The equations \Eqref{EqApp20}, \eqref{EqApp22} provide
a linear isomorphism and its inverse between the $e$--twisted
convolution algebra $\sS(G)$ and the algebra of integral operators
$\sS(M\times M)$ on $L^2(M)$ with product and involution given by
\begin{align*}
     (k_1 * k_2) (x_1,x_2) & = \int_M k_1(x_1,y) k_2(y,x_2) dy, \\
        k^*(x_1,x_2) &= \ovl{k(x_2,x_1)}.
\end{align*}
\end{remark}

As an application we reprove the orthogonality relations for
the Heisenberg representation.
\begin{cor}[Orthogonality Relations] \label{CorAppOrthRel}
Let $K_1, K_2\in\cL(L^2(M))$ be integral operators with
Schwartz kernels $k_1, k_2\in\sS(M\times M)$ of Schwartz class.
Then the Fourier transform of the function
\begin{equation}\label{EqAppOrthRel1}
   F(x):= \ovl{\Tr\bl \pi(x)^* K_1 \br } \Tr\bl \pi(x)^* K_2 \br 
\end{equation}
is given by 
\begin{equation}\label{EqAppOrthRel2}
   F^\wedge(\xi) = \Tr\bl K_1^* \pi(\xi) K_2 \pi(\xi)^*\br.
\end{equation}
In particular
\begin{equation}\label{EqAppOrthRel3}
  \int_G \ovl{\Tr\bl \pi(x)^* K_1 \br } \Tr\bl \pi(x)^* K_2 \br dx
     = \Tr\bl K_1^* K_2\br,
\end{equation}
and for the discrete cocompact subgroup $L$ we have
\begin{equation}\label{EqAppOrthRel4}
\sum_{\ell\in L} 
     \ovl{\Tr\bl \pi(x)^* K_1 \br } \Tr\bl \pi(x)^* K_2 \br 
     = \frac1{\vol_G(F_L)} \sum_{\xi\in L^\perp}
               \Tr\bl K_1^* \pi(\xi) K_2 \pi(\xi)^*\br.
\end{equation}
In particular, for Schwartz functions $f,g,h,k\in\sS(M)$
\begin{equation}\label{EqAppOrthRel5}
\int_G \ovl{\scalarL{f, \pi(x)^* g }} \scalarL{h, \pi(x)^* k} dx
   = \ovl{\scalarL{f,h}} \scalarL{g,k}.
\end{equation}
\end{cor}
\begin{proof} \Eqref{EqApp14} implies $\pi(\xi)^*\pi(y)^*\pi(\xi) =
\rho(\xi,y) \pi(y)^*$. Writing $K_1=\pi(f_1), K_2=\pi(f_2)$ 
Theorem \plref{TAppTraceFormula} and \Eqref{EqAppfTrace} imply
\begin{align*}
  F^\wedge(\xi) & = \int_G \rho(\xi,y) F(y) dy
     = \int_G  \ovl{\Tr\bl \pi(y)^* \pi(f_1)\br} \Tr\bl \pi(y)^*
                       (\pi(\xi) \pi(f_2) \pi(\xi)^*)\br dy \\
    & = \int_G \ovl{f_1(y)} (\uU_\xi f_2 \uU_\xi^*)(y) dy 
      = \bl f_1^**(\uU_\xi f_2 \uU_\xi^*)\br (0)\\
    & = \psihat\bl f_1^**(\uU_\xi f_2 \uU_\xi^* ) \br
      = \Tr\Bl \pi\bl f_1^**(\uU_\xi f_2 \uU_\xi^*)\br\Br 
      = \Tr\bl K_1^* \pi(\xi) K_2 \pi(\xi)^*\br,
\end{align*}
proving \Eqref{EqAppOrthRel2}.
Specializing $\xi=0$ gives \Eqref{EqAppOrthRel3} and 
applying the Poisson summation formula
to $F$ gives \Eqref{EqAppOrthRel4}. Finally,
\Eqref{EqAppOrthRel5}  is obtained from 
\Eqref{EqAppOrthRel3} with
$K_1= \scalarL{f,\cdot} g, K_2=\scalarL{h,\cdot} k $.
\end{proof}

\subsection{Inner product} Put for $f,g\in\sS(M)$
\begin{align*}
     \innl[\sB]{f,g} & = \vol_G(F_L)\cdot \sum_{l\in L}
                          \scalarL{\pi(l) g,f} \pi(l),\\
     \innr[\sA]{f,g}      & = \sum_{\xi\in L^\perp}               
                         \scalarL{f \pi(\xi)^\op, g} \pi(\xi)^\op 
                      = \sum_{\xi\in L^\perp}               
                    \scalarL{\pi(\xi) f , g} \pi(\xi)^\op.
\end{align*}
The orthogonality relations Cor. \ref{CorAppOrthRel}
 then imply  
that for any $f,g,h\in\sS(M)$
\[
    f\cdot \innr[\sA]{g,h} = \innl[\sB]{f,g}\cdot h,
\]
see \cite[Sec.~2]{Rie1988}.  Furthermore, with the normalized traces
$\varphi^\sA,\varphi^\sB$ we have 
\[
   \frac1{\vol_G(F_L)} \cdot \varphi^\sB(\innl[\sB]{g,f})
     = \scalarL{f,g} = \varphi^\sA(\innr[\sA]{f,g}).
\]                        
The constant $\vol_G(F_L)$ comes from the Poisson summation
formula. In our main example it is $\frac1{|c\gt+d|}$.

\subsection{Example}\label{SAppExample}
As an important example we specialize what we have presented so far
in this Appendix to the situation considered in the main body of the paper:
namely let $M=\R\times \Z/c\Z$ with $\mu_M=\nu\otimes \#$, \ie Lebesgue
measure tensored by the counting measure.  The pairing
\begin{equation}\label{EqAppEx1}
  \inn{x_1,\ga_1;x_2,\ga_2}:=
  e^{\tpii\bl x_1 x_2 + \ga_1\ga_2/c\br },
\qquad (x_1,\ga_1), (x_2,\ga_2)\in M,
\end{equation}
identifies $M^\wedge$ with $M$. 
The Plancherel measure is $\mu_{M^\wedge}=\nu\otimes \frac 1{|c|} \#$.
With 
\begin{equation}\label{EqAppEx2}
\gl(x_1,\ga_1;x_2,\ga_2)=\gl(x_1,x_2)=\inn{x_1,x_2/2},
\end{equation}
we have
\begin{align}
e(\sx,\sy) & = 
              e \bl (x_1,\ga_1),(x_2,\ga_2);(y_1,\gb_1),(y_2,\gb_2)\br 
        = e^{\tpii\bl         \frac{x_1y_2-x_2 y_1}2 +
                \frac{\ga_1\gb_2}c\br},\label{EqAppEx3} \\
    \rho(\sx,\sy) & = e(\sx,\sy) \ovl{e(\sy,\sx)}
        =  e^{\tpii\bl        x_1y_2-x_2 y_1 +
          \frac{\ga_1\gb_2-\ga_2\gb_1}c\br},\label{EqAppEx4}
\end{align}
and hence the projective representation $\pi$ of $G$
on $L^2(M)$ is explicitly given by
\begin{equation}
 \begin{split}
  \bl \pi(y_1,\gb_1; y_2,\gb_2) u \br(t,\ga)
       & = \gl(y_1,y_2) \cdot  \inn{(t,\ga),(y_2,\gb_2)} \cdot u(t+y_1,\ga+\gb_1) \\
       & = e^{\tpii\bl \frac{y_1y_2}2 + t y_2 + \frac{\ga\gb_2}c    \br}
             \cdot u(t+y_1,\ga+\gb_1),
\end{split}\label{EqAppEx5}
\end{equation}
and infinitesimally (\cf\Eqref{Eq:1209064})
\begin{align}
    \frac{\pl}{\pl y_1}\big|_{y=0} (\pi(\sy)u)(t,\ga) 
              &= \frac{\pl}{\pl t} u(t,\ga), \label{EqAppEx6} \\ 
    \frac{\pl}{\pl y_2}\big|_{y=0} (\pi(\sy)u)(t,\ga) 
              &= \tpii\cdot t \cdot u(t,\ga).\label{EqAppEx7}
\end{align}

Consider the following discrete cocompact subgroups of $G=M\times
M^{(\wedge)}$:
\begin{align}
     L  & = \Z \go_1 \oplus \Z \go_2, & 
     \go_1  &= (0,0;\frac1{c\gt+d},-1), \quad \go_2 = (-\frac 1c,-a;0,0),
              \label{EqAppEx8} \\
L^\perp & = \Z \tilde\go_1 \oplus \Z \tilde \go_2, &
     \tilde \go_1 & = (0,0;1, -d), \quad \tilde\go_2 = (-\gt-\frac dc,-1;0,0).
              \label{EqAppEx9} 
\end{align}
One checks that $L^\perp$ is indeed the group of those $\xi\in G$
such that $\rho(\xi,\ell)=1$ for all $\ell\in L$ and vice versa.
Furthermore, w.r.t. the self-dual Haar measure $\gl\otimes\#\otimes
\gl\otimes \# /|c|$ one has
\begin{equation}
\vol_G(F_L)  = \frac1{|c\gt+d|} = \vol_G(F_{L^\perp})\ii.
              \label{EqAppEx10} 
\end{equation}

Comparing with \Eqref{Eq05234} we see that
$V_1=\pi(\go_1), V_2=\pi(\go_2)$ and furthermore
\[
     V_2^l V_1^k  = e^{\tpii kl \bl \frac1{c(c\gt+d)}+ \gt'
\br}\cdot   \pi(k\go_1+l \go_2) = e^{\tpii k l \gt'} \cdot V_1^k V_2^l.
\]

Similarly, comparing with \Eqref{Eq05251} 
we find $U_1 = \pi (\tilde\go_1)^\op, U_2 = \pi(\tilde \go_2)^\op$
and
\begin{align*}
   U_2^l U_1^k  & = \pi(l\tilde \go_2)^\op \pi(k\tilde\go_1)^\op 
     = \bl \pi (k\tilde\go_1) \pi(l\tilde \go_2)\br^\op \\
    & = e^{\tpii kl \frac{\gt + d/c}2} \cdot  
    \pi(k\tilde \go_1+ l \tilde \go_2)^\op = e^{\tpii kl \gt } 
       \bl \pi(l\tilde \go_2) \pi(k\tilde\go_1)\br^\op \\
    & = e^{\tpii kl \gt } \cdot U_1^k U_2^l.
\end{align*}

The natural action $\ga$ of $G$ on the algebras 
$\Binf, \Ainf$ is given by
\begin{align*}
    \ga_\sx(V_1) & = e^{\tpii \bl \frac{x_1}{c\gt+d}-\frac{\ga_1}c \br}
                  \cdot V_1,\quad
                           \sx=(x_1,\ga_1;x_2,\ga_2),\\
    \ga_\sx(V_2) & = e^{\tpii \frac{x_1+a\ga_2}c}\cdot V_2, \\
    \ga_\sx(U_1) & = e^{\tpii \bl x_1-\frac{\ga_1}c \br}\cdot U_1,\\
    \ga_\sx(U_2) & = e^{\tpii \bl (\gt+\frac dc) x_2+\frac{\ga_2}c\br }\cdot U_2.
\end{align*}

Putting $\tilde \pi(x_1,x_2):= \pi(x_1,0;\mu x_2,0), \mu=\frac1{\gt+\frac dc}$
and $\tilde\ga_{x_1,x_2}(b):= \tilde\pi(x_1,x_2) \cdot b\cdot \tilde\pi(x_1,x_2)^*$
we find, \cf\Eqref{pi-0}, \eqref{Eq:1209064}, indeed
\begin{equation}
  \begin{split}
    (\tilde\pi(x)u)(t,\ga) & = e^{\pi i \mu\bl x_1 x_2 + 2 t x_2\br} u(t+x_1),\\
    \tilde\ga_{x_1,x_2}(V_j) & = e^{\frac{\tpii}{c\gt+d} x_j}\cdot V_j,\quad j= 1,2, \\
    \tilde\ga_{x_1,x_2}(U_j) & = e^{\tpii x_j}\cdot U_j,\quad j= 1,2.
\end{split}
              \label{EqAppEx14} 
\end{equation}

\subsection{The trace formula for ``trivial vector bundles''}
\label{AppB}

For completeness we discuss here the analog of the trace formula
Theorem \ref{TAppTraceFormula} for the trivial $\sB$--vector bundle.

With the notation of Example \ref{SAppExample} let $\tilde\ga$ be the
normalized action of $G$ on $\sB$. The family $(\pi(\ell)\br_{\ell\in L}$
is an orthonormal basis of the Hilbert space $L^2(\sB,\vfB)$
which is the completion of $\sB$ with respect to the scalar
product $\scalarL{a,b}=\vfB(a^*b)$. Furthermore, the action
$\tilde\ga$ is at the same time a unitary representation of $G$
on $L^2(\sB,\vfB)$. Letting $\sB$ act on the left on $L^2(\sB,\vfB)$
the unitary action $\tilde \ga$ on $L^2(\sB,\vfB)$ implements the
action $\tilde\ga$ on $\sB$. Clearly, $\tilde\ga$ integrates in
the usual way to a representation of the (untwisted) convolution
algebra $\sS(G,\Binf)\simeq \sS(G\times L)$. A calculation similar
to \Eqref{EqAppProductAction} gives
\begin{align*}
    \tilde\ga(f)(\pi(l)) & = \int_G f(x) \tilde\ga_x(\pi(l)) dx 
      = \int_G f(x) \rho(x,l) \pi(l) dx \\
     & = \int_G \sum_{k\in L} f(x,k) \rho(x,l) e(k,l) \pi(k+l) dx \\
     & = \sum_{k\in L} \int_G f(x,k-l) \rho(x,l) e(k-l,l)\, dx\; \pi(k).
\end{align*}
This shows that with respect to the basis $\pi(\ell), l\in L$, the operator
$\tilde\ga(f)$ is given by a rapidly decreasing (\ie Schwartz class function)
kernel in $\sS(L\times L)$.
Furthermore, the trace is given by
\begin{align*}
   \Tr(\tilde\ga(f)) & = \sum_{\ell\in L} \int_G f(x,0) \rho(x,l) dx 
       = \sum_{\ell\in L} \vfB(\hat f(\ell)),
\end{align*}
where $\hat f$ denotes the Fourier transform on the self-dual group
$G=M\times M^\wedge$ w.r.t. the bicharacter $\rho$.

\bibliography{mlbib.bib,localbib.bib}

\newcommand{\etalchar}[1]{$^{#1}$}
\def\cprime{$'$}
\providecommand{\bysame}{\leavevmode\hbox to3em{\hrulefill}\thinspace}
\providecommand{\MR}{\relax\ifhmode\unskip\space\fi MR }
\providecommand{\MRhref}[2]{%
  \href{http://www.ams.org/mathscinet-getitem?mr=#1}{#2}
}
\providecommand{\href}[2]{#2}
\begin{thebibliography}{\textsc{BCD{\etalchar{+}}72}}

\bibitem[\textsc{Baa88a}]{Baa:CPDII}
\textsc{S.~Baaj}, \emph{Calcul pseudo-diff\'erentiel et produits crois\'es de
  {$C\sp *$}-alg\`ebres. {I}}, C. R. Acad. Sci. Paris S\'er. I Math.
  \textbf{307} (1988), no.~11, 581--586. \MR{967366 (90a:46171)}

\bibitem[\textsc{Baa88b}]{Baa:CPDI}
\bysame, \emph{Calcul pseudo-diff\'erentiel et produits crois\'es de {$C\sp
  *$}-alg\`ebres. {II}}, C. R. Acad. Sci. Paris S\'er. I Math. \textbf{307}
  (1988), no.~12, 663--666. \MR{967808 (90a:46172)}

\bibitem[\textsc{BCD{\etalchar{+}}72}]{Ber1972}
\textsc{P.~Bernat}, \textsc{N.~Conze}, \textsc{M.~Duflo},
  \textsc{M.~L{\'e}vy-Nahas}, \textsc{M.~Ra{\"{\i}}s}, \textsc{P.~Renouard},
  and \textsc{M.~Vergne}, \emph{Repr\'esentations des groupes de {L}ie
  r\'esolubles}, Dunod, Paris, 1972, Monographies de la Soci{\'e}t{\'e}
  Math{\'e}matique de France, No. 4. \MR{0444836 (56 \#3183)}

\bibitem[\textsc{Bos87}]{Bos1987}
\textsc{J.-B. Bost}, \emph{Fibr\'es d\'eterminants, d\'eterminants
  r\'egularis\'es et mesures sur le espaces de modules des courbes complexes},
  Ast\'erisque (1987), no.~152-153, 4, 113--149 (1988), S{\'e}minaire Bourbaki,
  Vol. 1986/87. \MR{936852 (90d:58142)}

\bibitem[\textsc{CoCo93}]{CohCon1992}
\textsc{P.~B. Cohen} and \textsc{A.~Connes}, \emph{Conformal geometry of the
  irrational rotation algebra}, Preprint MPI Bonn, 1992-93.

\bibitem[\textsc{CoMo14}]{ConMos2011}
\textsc{A.~Connes} and \textsc{H.~Moscovici}, \emph{Modular curvature for
  noncommutative two-tori}, J. Amer. Math. Soc. \textbf{27} (2014), no.~3,
  639--684. \texttt{arXiv:1110.3500 [math.QA]}, \MR{3194491}

\bibitem[\textsc{Con80}]{Con:CAG}
\textsc{A.~Connes}, \emph{{$C\sp{\ast} $} alg\`ebres et g\'eom\'etrie
  diff\'erentielle}, C. R. Acad. Sci. Paris S\'er. A-B \textbf{290} (1980),
  no.~13, A599--A604. \MR{572645 (81c:46053)}

\bibitem[\textsc{Con82}]{Con1982}
\textsc{A.~Connes}, \emph{A survey of foliations and operator algebras},
  Operator algebras and applications, {P}art {I} ({K}ingston, {O}nt., 1980),
  Proc. Sympos. Pure Math., vol.~38, Amer. Math. Soc., Providence, R.I., 1982,
  pp.~521--628. \MR{679730 (84m:58140)}

\bibitem[\textsc{Con94}]{Con1994}
\textsc{A.~Connes}, \emph{Noncommutative geometry}, Academic Press, Inc., San
  Diego, CA, 1994. \MR{1303779 (95j:46063)}

\bibitem[\textsc{Con96}]{Con1996}
\bysame, \emph{Gravity coupled with matter and the foundation of
  non-commutative geometry}, Comm. Math. Phys. \textbf{182} (1996), no.~1,
  155--176. \MR{1441908 (98f:58024)}

\bibitem[\textsc{Con13}]{Con2013}
\bysame, \emph{On the spectral characterization of manifolds}, J. Noncommut.
  Geom. \textbf{7} (2013), no.~1, 1--82. \MR{3032810}

\bibitem[\textsc{CoRi87}]{ConRie1987}
\textsc{A.~Connes} and \textsc{M.~A. Rieffel}, \emph{Yang-{M}ills for
  noncommutative two-tori}, Operator algebras and mathematical physics ({I}owa
  {C}ity, {I}owa, 1985), Contemp. Math., vol.~62, Amer. Math. Soc., Providence,
  RI, 1987, pp.~237--266. \MR{878383 (88b:58033)}

\bibitem[\textsc{CoTr11}]{ConTre2011}
\textsc{A.~Connes} and \textsc{P.~Tretkoff}, \emph{The {G}auss-{B}onnet theorem
  for the noncommutative two torus}, Noncommutative geometry, arithmetic, and
  related topics, Johns Hopkins Univ. Press, Baltimore, MD, 2011, pp.~141--158.
  \texttt{arXiv:0910.0188 [math.QA]}, \MR{2907006}

\bibitem[\textsc{FaKh13}]{FarKha2013}
\textsc{F.~Fathizadeh} and \textsc{M.~Khalkhali}, \emph{Scalar curvature for
  the noncommutative two torus}, J. Noncommut. Geom. \textbf{7} (2013), no.~4,
  1145--1183. \texttt{arXiv:1110.3511 [math.QA]}, \MR{3148618}

\bibitem[\textsc{Gil95}]{Gil:ITH}
\textsc{P.~B. Gilkey}, \emph{Invariance theory, the heat equation, and the
  {A}tiyah-{S}inger index theorem}, second ed., Studies in Advanced
  Mathematics, CRC Press, Boca Raton, FL, 1995. \MR{1396308 (98b:58156)}

\bibitem[\textsc{KhMo14}]{KhaMoa2014}
\textsc{M.~Khalkhali} and \textsc{A.~Moatadelro}, \emph{A {R}iemann-{R}och
  theorem for the noncommutative two torus}, J. Geom. Phys. \textbf{86} (2014),
  19--30. \MR{3282309}

\bibitem[\textsc{Kle62}]{Kle1962}
\textsc{A.~Kleppner}, \emph{The structure of some induced representations},
  Duke Math. J. \textbf{29} (1962), 555--572. \MR{0141735 (25 \#5132)}

\bibitem[\textsc{Les10}]{Les:PDO}
\textsc{M.~Lesch}, \emph{Pseudodifferential operators and regularized traces},
  Motives, quantum field theory, and pseudodifferential operators, Clay Math.
  Proc., vol.~12, Amer. Math. Soc., Providence, RI, 2010, pp.~37--72.
  \texttt{arXiv:0901.1689 [math.OA]}, \MR{2762524}

\bibitem[\textsc{Les14}]{Les2014}
\bysame, \emph{Divided differences in noncommutative geometry: Rearrangement
  lemma, functional calculus and expansional formula}, J. Noncomm. Geom. (to
  appear), 2014. \texttt{arXiv:1405.0863v2 [math.OA]}

\bibitem[\textsc{OPS88}]{OPS1988}
\textsc{B.~Osgood}, \textsc{R.~Phillips}, and \textsc{P.~Sarnak},
  \emph{Extremals of determinants of {L}aplacians}, J. Funct. Anal. \textbf{80}
  (1988), no.~1, 148--211. \MR{960228 (90d:58159)}

\bibitem[\textsc{Pol04}]{Pol2004}
\textsc{A.~Polishchuk}, \emph{Classification of holomorphic vector bundles on
  noncommutative two-tori}, Doc. Math. \textbf{9} (2004), 163--181
  (electronic). \texttt{arXiv:math/0308136v1 [math.QA]}, \MR{2054986
  (2005c:58013)}

\bibitem[\textsc{PoSc03}]{PolSchwa2003}
\textsc{A.~Polishchuk} and \textsc{A.~Schwarz}, \emph{Categories of holomorphic
  vector bundles on noncommutative two-tori}, Comm. Math. Phys. \textbf{236}
  (2003), no.~1, 135--159. \texttt{arXiv:math/0211262v2 [math.QA]}, \MR{1977884
  (2004k:58011)}

\bibitem[\textsc{Rie81}]{Rie1981}
\textsc{M.~A. Rieffel}, \emph{{$C\sp{\ast} $}-algebras associated with
  irrational rotations}, Pacific J. Math. \textbf{93} (1981), no.~2, 415--429.
  \MR{623572 (83b:46087)}

\bibitem[\textsc{Rie88}]{Rie1988}
\bysame, \emph{Projective modules over higher-dimensional noncommutative tori},
  Canad. J. Math. \textbf{40} (1988), no.~2, 257--338. \MR{941652 (89m:46110)}

\bibitem[\textsc{Ros13}]{Ros2013}
\textsc{J.~Rosenberg}, \emph{Levi-{C}ivita's theorem for noncommutative tori},
  SIGMA Symmetry Integrability Geom. Methods Appl. \textbf{9} (2013), Paper
  071, 9. \MR{3141539}

\bibitem[\textsc{Shu01}]{Shu:POS}
\textsc{M.~A. Shubin}, \emph{Pseudodifferential operators and spectral theory},
  second ed., Springer-Verlag, Berlin, 2001, Translated from the 1978 Russian
  original by Stig I. Andersson. \MR{1852334 (2002d:47073)}

\end{thebibliography}
\bibliographystyle{amsalpha-lmp}
\end{document}